\documentclass[11pt,reqno]{amsart}

\usepackage{amsmath,amssymb,amsthm}
\usepackage{amsaddr}
\usepackage{amsfonts}
\usepackage{natbib}
\usepackage{mathtools}
\usepackage{mathrsfs}
\usepackage{enumitem}
\usepackage{bm}
\usepackage{hyperref}
\usepackage{url}
\usepackage{subcaption}
\usepackage[table,xcdraw]{xcolor}
\usepackage{multirow}
\usepackage{mathtools}
\mathtoolsset{showonlyrefs=true}

\renewcommand{\baselinestretch}{1.2}
\usepackage{fullpage}

\usepackage{tikz}
\usetikzlibrary{positioning,shapes,arrows}

\newtheorem{theorem}{Theorem}
\newtheorem{corollary}[theorem]{Corollary}
\newtheorem{lemma}[theorem]{Lemma}
\newtheorem{proposition}[theorem]{Proposition}
\newtheorem{assumption}{Assumption}
\theoremstyle{definition}
\newtheorem{definition}{Definition}
\newtheorem{remark}{Remark}
\newtheorem{example}{Example}

\newcommand{\R}{\mathbb{R}}
\newcommand{\N}{\mathbb{N}}
\newcommand{\mB}{\mathcal{B}}
\renewcommand{\d}{\mathrm{d}}

\renewcommand{\hat}{\widehat}
\renewcommand{\tilde}{\widetilde}
\newcommand{\Ep}{\mathbb{E}}
\newcommand{\mE}{\mathcal{E}}
\newcommand{\mF}{\mathcal{F}}
\newcommand{\mG}{\mathcal{G}}
\newcommand{\mH}{\mathcal{H}}

\newcommand{\mM}{\mathcal{M}}

\newcommand{\mT}{\mathcal{T}}
\newcommand{\mP}{\mathcal{P}}
\newcommand{\mX}{\mathcal{X}}
\newcommand{\Z}{\mathbb{Z}}

\newcommand{\argmin}{\operatornamewithlimits{argmin}}
\newcommand{\argmax}{\operatornamewithlimits{argmax}}

\newcommand{\deemph}[1]{{\small\color{black!60}#1}}

\renewcommand{\epsilon}{\varepsilon}

\usepackage{tikz}
\usepackage{tikz-cd}
\usetikzlibrary{positioning,shapes,arrows}

\allowdisplaybreaks

\begin{document}

\title{Minimax rates of estimation for optimal transport map between infinite-dimensional spaces}

\author{Donlapark Ponnoprat$^1$ \and Masaaki Imaizumi$^{2,3}$}
\address{$^1$Chiang Mai University, $^2$The University of Tokyo, $^3$RIKEN AIP}

\begin{abstract}
We investigate the estimation of an optimal transport map between probability measures on an infinite-dimensional space and reveal its minimax optimal rate. Optimal transport theory defines distances within a space of probability measures, utilizing an optimal transport map as its key component. Estimating the optimal transport map from samples finds several applications, such as simulating dynamics between probability measures and functional data analysis. However, some transport maps on infinite-dimensional spaces require exponential-order data for estimation, which undermines their applicability. In this paper, we investigate the estimation of an optimal transport map between infinite-dimensional spaces, focusing on optimal transport maps characterized by the notion of $\gamma$-smoothness. Consequently, we show that the order of the minimax risk is polynomial rate in the sample size even in the infinite-dimensional setup. With these results, we obtain a class of reasonably estimable optimal transport maps on infinite-dimensional spaces and a method for their estimation. We also develop an estimator whose estimation error matches the minimax optimal rate. Our experiments validate the theory and practical utility of our approach with application to functional data analysis.
\end{abstract}

\maketitle

\section{Introduction}

\subsection{Background}
The optimal transport theory \cite{villani2009optimal,santambrogio2015optimal,peyre2019computational} is an important tool to measure the distance between probability measures and is actively used in data science these days. 
It is formulated by the Monge problem \cite{monge1781memoire} and Kantorovich problem \cite{kantorovich1942translocation,kantorovich1960mathematical}.
The theory is widely used in various fields such as statistics \cite{sommerfeld2018inference,del2022central,okano2023distribution,okano2024wasserstein}, machine learning \cite{cuturi2013sinkhorn,arjovsky2017wasserstein,sugimoto2024augmented}, and many others.
Its statistical properties, such as an estimation error and asymptotic distributions are actively studied \cite{weed2019sharp,sommerfeld2018inference,bigot2019central,okano2024inference,manole2021sharp,imaizumi2022hypothesis,ponnoprat2024uniform}.
By utilizing their advantages, several studies have been developed that advance functional data analysis using optimal transport \cite{mallasto2017learning,Korotin2021,zhu2021functional}.

We consider the statistical estimation problem for an optimal transport map between infinite-dimensional spaces. 
Formally, we consider two probability measures $P,Q$ with supports $\Omega_P,\Omega_Q \subset \Omega \subset [0,1]^{\infty}$ where $\Omega$ is a  bounded set with respect to the $\ell^2$-norm $\|\cdot\|$. 
An \textit{optimal transport map (OT map)} $T_0: \Omega \to \Omega$ is defined as the minimizer of the following Monge problem \cite{monge1781memoire}:
\begin{align} \label{eq:Monge}
    \min_{T : T_{\#} P = Q} \int \tfrac{1}{2} \|T(x) - x\|^2 ~ dP(x).
\end{align}
The statistical estimation for the OT map is formulated as follows: suppose that the measures $P$ and $Q$ are unknown, and $\Omega$-valued $n$ samples of each of them are observed, i.e., $X_1,...,X_n \sim P$ and $Y_1,...,Y_n \sim Q = (T_0)_{\#} P$ on $\Omega$. 
Our interest is in the \textit{minimax risk of estimation} to assess the difficulty of estimating the OT map.
Given a set of probability measures on $\Omega$ as $\mP(\Omega)$ and a certain set of transport maps $\overline{\mF} \subset \{T :  \Omega \to \Omega\}$,
we define the following {minimax risk} of estimating the OT map from $\overline{\mF}$:
\begin{align}
    \mathscr{R}_n(\overline{\mF}) := \inf_{\overline{T}_n} \sup_{P \in \mP(\Omega),  T_0 \in \overline{\mF}} \mathbb{E}\left[ \int \lVert \overline{T}_n(x) - T_0(x)\rVert^2 dP(x)\right], 
    \label{def:minimax_risk}
\end{align}
where $\overline{T}_n$ is taken from all estimators for $T_0$ based on the $n$ observations $\{X_i\}_{i=1}^n \cup \{Y_i\}_{i=1}^n$.

Estimating the OT map $T_0$ is one of the key problems in optimal transport theory for the following reasons. 
In general, OT maps are applicable to several statistical problems: unlike the ordinary regression problem in which a map is estimated from a set of $n$ paired data points $\{(X_i, Y_i)\}_{i=1}^n$, the OT map estimation allows estimating a map from two sets of unpaired data points, $\{(X_i)\}_{i=1}^n$ and $\{(Y_i)\}_{i=1}^n$. This unpaired setting arises in many situations such as distributional regression, clustering, and simulation-based interpolation \cite{chen2023wasserstein,bigot2017geodesic}.
Even for the infinite-dimensional setup, this usage is employed in several applications such as functional data analysis \cite{mallasto2017learning,zhu2021functional,jiang2024training}, and simulating physical dynamics \cite{jiang2024training}. 
Motivated by the usage, several studies have investigated the OT map estimation problem.
\cite{manole2021plugin} and \cite{deb2021rates} propose a plugin estimation method for smooth OT maps, while \cite{divol2022optimal} extends this to a broader function space. \cite{pooladian2021entropic} explore entropic regularization for estimating the OT map, and later \cite{pooladian2023minimax} and \cite{sadhu2023stability} explore discontinuous OT maps in semi-discrete scenarios. \cite{hutter2021minimax} focus on deriving minimax optimal rates for estimating smooth OT maps, emphasizing convergence rates within differentiable classes.

One of the challenges in estimating OT maps lies in clarifying the efficiency of estimation, such as the convergence rate of the estimation error in the sample size $n$. Although this analysis requires performing estimation on infinite-dimensional spaces, this setup cannot be easily handled as an extension of the previous estimation in finite-dimensional cases, due to the curse of dimensionality of the estimation.
For a case with $d$-dimensional data, several studies \cite{hutter2021minimax,divol2022optimal,pooladian2021entropic} show that the minimax risk in estimating an OT map from $\bar{\alpha}$-H\"older space has a rate  ${O}(n^{-2\bar{\alpha}/(2\bar{\alpha} - 2 + d)})$ up to logarithmic factors.  
In contrast, when the data are infinite-dimensional  as an extreme case $(d = \infty)$, the existing results are no longer valid.
Although the related works \cite{mallasto2017learning,zhu2021functional,minh2023entropic} have developed the consistent estimator for the OT map, a rate of their risk is not clarified.
Indeed, in our preliminary analysis, we prove that for twice-differentiable OT maps between infinite-dimensional spaces, the minimax estimation error is lower bounded by a poly-logarithmic rate.
Specifically, we consider a $C^1$-class of OT maps $\mT_{\eta,\beta}$ on infinite-dimensional space (defined in Section \ref{sec:sobolevellipsoid}), which is the natural extension of the common class of OT maps commonly studied in finite‐dimensional setups \cite{hutter2021minimax}. 
Under this $C^1$-class, we have the following result:
\begin{proposition}[Informal result from Section \ref{sec:sobolevellipsoid}]
    We have the following as $n \to \infty$:
    \begin{equation}
        \label{eq:3}
\mathscr{R}_n(\mT_{\eta,\beta}) \gtrsim \frac1{(\log n)^{2b+1}}. 
    \end{equation}
\end{proposition}
This result implies that the minimax risk decays logarithmically in the sample size $n$, so that efficient estimation is fundamentally challenging.  Accordingly, our findings suggest the necessity of developing a novel class of OT maps together with corresponding estimators.

\subsection{Our result}
In this study, we investigate the minimax risk of the OT map estimation with the infinite-dimensional data. 
To this aim, we introduce the notion of $\gamma$-smoothness of maps on infinite-dimensional spaces, which generalizes the Sobolev space on finite-dimensions and also flexibly represents multiple function spaces on infinite-dimensional spaces.
Also, we utilize the fact the OT map $T_0$ is related to the gradients of a Brenier potential $\varphi_0$ and a Kantorovich potential $\phi_0$ by the form $T_0 = \nabla \varphi_0 =  \operatorname{Id} - \frac1{2} \nabla \phi_0$, where $\operatorname{Id}$ is the identity map.
Then, we assume that the Kantorovich potential has the $\gamma$-smoothness, which is a generalization of the  smoothness based on the decay rate of Fourier coefficients to infinite-dimensions \cite{schmeisser1987unconditional,nikol2012approximation,okumoto2021learnability,takakura2023approximation}.

With this setup, we derive the minimax rate of the error of the OT map estimation. 
Specifically, we show that, in our setting, the minimax rate is 
    ${\Theta}(n^{-2/(2 + \alpha(\gamma))})$ up to logarithmic factors,
where $\alpha(\gamma) > 0$ is an inverse-smoothness index of the smoothness map $\gamma(\cdot)$, which means that a smaller $\alpha(\gamma)$ makes the OT map $T_0$ smoother.
\begin{theorem}[Informal version of Corollary 
 \ref{cor:upplow}]
There exists a constant $\alpha(\gamma) > 0$ depending on $\gamma$ satisfying the following up to logarithmic factors in $n$:
    \begin{align}
\mathscr{R}_n(\mF_T^\gamma) \asymp n^{-2/(2+\alpha(\gamma))}.
    \end{align}
\end{theorem}
This result gives several insights: (i) the rate is a polynomial rate in the sample size $n$ and avoids the logarithmic rate above, and (ii) the rate is dimension-free and implies that the estimation is reasonable in practice.
We summarize the minimax risk of the OT map estimation in Table \ref{tab:rates}.
The proof of the minimax rate utilizes the information-theoretic lower bound and the localization technique in the infinite-dimensional spaces.

Furthermore, as an alternative to the plug-in estimator, we propose another estimator that leverages neural networks to minimize a loss function based on the \emph{semi-dual problem} of the optimal transport map. Then, we theoretically show that the estimator achieves near-optimal minimax rates. We further validate this result via simulation.

\begin{table}[htbp]
    \centering
    \begin{tabular}{r|cc|cc}
        Study & \cite{hutter2021minimax,manole2021plugin}  & \cite{divol2022optimal} &  {Ours} (preliminary) & \textbf{Ours (main)} \\ \hline
        Dimension & $d<  \infty $  & $d<  \infty $ &  $d=\infty$ &  $d=\infty$ \\ 
        Minimax rate & ${\Theta}\left( n^{-{2\bar{\alpha}}/({2\bar{\alpha} -2 + d})}  \right)$ & ${O}\left(n^{-{2}/({2+r})}\right)$ &   $\Omega((\log n)^{2b+1})$  & ${\Theta}\left(n^{-{2}/({2+\alpha(\gamma)})}\right)$ \\ 
        OT map & \begin{tabular}[c]{@{}c@{}} smooth \\ ($\bar{\alpha}$-H\"older) \end{tabular}  & \begin{tabular}[c]{@{}c@{}} general \\ ($r$-entropy)\end{tabular}   &   \begin{tabular}[c]{@{}c@{}}  smooth  \\($C^1$-class)\end{tabular}  & $\gamma$-smooth
     \end{tabular}
    \caption{Comparison of the rates of the minimax risk ignoring logarithmic factors. $ r, b \geq 1 $ and $ \bar{\alpha},\alpha(\gamma) > 0$ are parameters determined by a setup or the function spaces, respectively. Details will be provided, e.g., Section \ref{sec:mixedaniso}. While it is non-trivial to obtain a polynomial rate with $d=\infty$ as shown in our preliminary analysis, our main result achieves the polynomial rate under $\gamma$-smoothness.
    \label{tab:rates}}
\end{table}

\subsection{Related work}
We comprehensively refer to related studies on the relevant topics.

\textbf{OT Map Estimation and Inference:}
Several studies have investigated the OT map estimation problem with finite-dimensional data: \cite{hutter2021minimax} established minimax optimal rates up to poly-logarithmic factors for estimating OT maps in H\"older spaces and introduced a semi-dual estimator achieving these rates. 
\cite{divol2022optimal} further extended the analysis of this estimator to OT maps in broader function spaces. More recently, \cite{Yizhe2024} generalized these results by relaxing the underlying distributions' smoothness and boundedness assumptions, providing upper bounds for the semi-dual estimator. For alternative approaches, \cite{manole2021plugin} and \cite{deb2021rates} investigated plug-in estimators, where transport maps are defined based on empirical distributions, and proved that such estimators achieve minimax rates. \cite{pooladian2021entropic} proposed to optimize the entropic regularized semi-dual problem, which allows the use of Sinkhorn's algorithm \cite{peyre2019computational} for scalable computations in high dimensions.
\cite{ponnoprat2024uniform} developed a statistical inference method for an OT map when data points are one-dimensional.
\cite{pooladian2023minimax} and \cite{sadhu2023stability} explored discontinuous OT maps in semi-discrete scenarios.

\textbf{Application of Optimal Transport:}
A typical application of OT with infinite-dimensional spaces is the functional data analysis has been an active research field \cite{ramsay2005functional,ferraty2006nonparametric,hsing2015theoretical,wang2016functional} and various methods for regression, classification, and clustering have been proposed \cite{hall2007methodology,jacques2014functional,imaizumi2018pca,imaizumi2019simple}.
About the optimal transport theory is another emerging research topic \cite{monge1781memoire,kantorovich1942translocation,kantorovich1960mathematical,villani2009optimal,santambrogio2015optimal,peyre2019computational}.
Its statistical properties are actively studied \cite{weed2019sharp,sommerfeld2018inference,bigot2019central,okano2024inference,manole2021sharp,imaizumi2022hypothesis} and widely used in various application topics in statistics \cite{cuturi2013sinkhorn,arjovsky2017wasserstein,sugimoto2024augmented,okano2023distribution,okano2024wasserstein}.
About the functional data analysis with the OT theory, the following papers have novel contributions.
\cite{mallasto2017learning} introduces a framework for analyzing populations of Gaussian processes using the optimal transport metric and demonstrates the convergence of finite-dimensional approximations.
\cite{Korotin2021} provides a quantitative evaluation framework for neural network-based solvers of optimal transport and rigorously tests solver accuracy across high-dimensional settings and generative modeling tasks.
\cite{zhu2021functional} introduces a novel functional optimal transport framework to estimate OT maps between distributions supported on function spaces and propose an efficient algorithm based on Hilbert-Schmidt operators.

\subsection{Notation}
Let $\lambda(\cdot)$ be the Lebesgue measure on $\R$.
For $x = (x_1,...,x_d)^\top \in \R^d$, we define its norm as $\|x\|_p \coloneqq (\sum_{j=1}^d |x_j|^p)^{1/p}$ for $p \in (0,\infty)$, $\|x\|_{\infty}=\max_{j=1,...,d}|x_j|$, and $\|x\|_0$ is a number of non-zero elements of $x$.
We denote $\|\cdot\| = \|\cdot\|_2$ for simplicity.
Let $\ell^2 := \{x \in \R^\infty : \|x\| < \infty\}$ be the $\ell^2$ space. Let $\R_{> 0}^\infty = \{ x \in \R^{\infty} : x_i > 0 \text{ for all } i \}$. Let $\N^{\infty}_0 =\{ l \in (\N \cup \{0\})^{\infty} : l_i =0 \text{ for all but finitely many }i\in \N \}$ and define $\Z^{\infty}_0$ in the same way. 
Let $\mP(\Omega)$ be a set of all probability measures on a set $\Omega$.
For probability measure $P \in \mP(\Omega)$ and a map $T:\Omega \to \Omega$, we define a push-forward measure as $T_{ \#} P$.
We denote the support of $P$ as $\Omega_P$. 
We denote the Kullback-Leibler divergence between $P,P' \in \mP(\Omega)$ as $D(P \| P')$. 
$\delta_x$ denotes the Dirac measure at $x$.
For $P \in \mP(\Omega)$ and $p \geq 1$, we define $L^p(P) := \{f: \Omega \to \R : (\int |f|^p dP)^{1/p} < \infty\}$ to be the $L^p$-space with $P$.
For a set $\Omega \subset \R^d$, $L^p(\Omega)$ denotes the $L^p$-space on $\Omega$ with the Lebesgue measure $\lambda^d$.
We also define an infinite product of uniform measure on $[0,1]$ as $\lambda^{\infty}: = \bigotimes_{j=1}^\infty \lambda$.
For an operator $T: \Omega \to \R^d$ and $P \in \mP(\Omega)$, we define $\|T\|_{L^2(P)}^2 := \int \|T(x)\|^2 dP(x)$.
Also, $\|\cdot\|_{\mathrm{op}}$ denotes the operator norm.
For $z \in \R$, $\lfloor z \rfloor$ denotes the largest integer that is no more than $z$. 
For real sequences $\{a_n\}_{n \in \N}$ and $ \{b_n\}_{n \in \N}$, $a_n = O(b_n)$ and $a_n \lesssim b_n$ denote $\lim_{n \to \infty} a_n/b_n < \infty$, $a_n = \Omega(b_n)$ denotes $\lim_{n \to \infty} a_n/b_n > 0$, $a_n = \Theta(b_n)$ denotes both $a_n = O(b_n)$ and $b_n = O(a_n)$ hold.
$a_n \asymp b_n$ denotes both $a_n \lesssim b_n$ and $b_n \lesssim a_n$ holds.
$\Tilde{O}(\cdot)$ denotes ${O}(\cdot)$ ignoring multiplied polynomials of $\log(n)$.
For a set $\Omega$ with a norm $\|\cdot\|$ and $\varepsilon > 0$, $N(\varepsilon, \Omega, \|\cdot\|)$ denotes a covering number, which is the smallest number of $\varepsilon$-balls to cover $\Omega$ in terms of $\|\cdot\|$.
For a map $\phi: [0,1]^\infty \to \R$, $\nabla \phi$ denotes the Fr\'echet derivative of $\phi$, and $\nabla^2 \phi$ denotes the Hessian of $\phi$, that is, the matrix of all second-order Fr\'echet derivative of $\phi$. We refer to Appendix~\ref{sec:frechet} for their definitions and properties relevant to this study.
A table of all major notations will be included in Section \ref{sec:notation_table}.

\section{Problem setting and notations} \label{sec:setup}

\subsection{OT on infinite-dimensional space and estimation of OT map}
We define the OT map $T_0: \Omega \to \Omega$ as the minimizer of the Monge problem \eqref{eq:Monge}, whose uniqueness will be discussed in Section \ref{sec:brenier}.
It is easy to consider $\Omega \subset [-c,c]^{\infty}$ for any $c > 0$ by a simple scaling argument.
Since integrable functions are characterized by infinite sequences of coefficients, considering an infinite-dimensional cube corresponds to considering functions.

We consider the estimation problem for the OT map $T_0$ as in \eqref{eq:Monge}.
Suppose we observe $n$ i.i.d. samples $X_1,...,X_n \sim P$ and additional $m$ i.i.d. samples $Y_1,...,Y_m \sim Q$.
For brevity, we consider the case with $n=m$.
Here, we aim to estimate the OT map $T_0$ from the samples. 

Our interest is in the \textit{minimax risk of estimation}, as defined in \eqref{def:minimax_risk},  to assess the difficulty of estimating the OT map.
This minimax risk is commonly used as a measure of the difficulty of the data-driven estimation problems \cite{gine2021mathematical,wainwright2019high} including the OT map estimation problem \cite{hutter2021minimax,divol2022optimal,manole2021plugin}.

\subsection{Potential functions}\label{sec:potentials}
We introduce several potential functions used in the optimal transport problem, specifically, the Kantorovich potential and the Brenier potential.
These are used to characterize the OT map $T_0$.

\textbf{Kantorovich potential.}
We introduce the Kantorovich potential by considering the Kantorovich problem \cite{kantorovich1942translocation}, which is a relaxation of the Monge problem, defined as follows:
\begin{align}
    \min_{\Gamma \in \Pi(P,Q)} \int \tfrac12 \|x - y\|^2 d\Gamma(x,y),
\end{align}
where $\Pi(P,Q)$ is the set of probability measures $\Gamma$ on $[0,1]^{\infty} \times [0,1]^{\infty}$ such that $\Gamma(A\times [0,1]^{\infty}) = P(A)$ and $\Gamma([0,1]^{\infty}\times A) = Q(A)$ for all Borel measurable set $A \in [0,1]^{\infty}$.
We also consider a dual problem of the Kantorovich problem (see \cite{villani2009optimal}) as follows:
\begin{align}
        \sup_{\substack{\phi \in L^1(P) ,\tilde{\phi} \in L^1(Q) \\ \phi(x) + \tilde{\phi}(y) \le \tfrac12\lVert x - y \rVert^2}} \int \phi(x) \ \d P(x) + \int \tilde{\phi}(y) \ \d Q(y). \label{eq:kantdual}
\end{align}
Here, the functions $\phi,\tilde{\phi}$ are called \emph{potentials}.
We will present that there exists $\phi_0$ such that a pair $(\phi_0, \phi_0^c)$ with $\phi^c_0(y) \coloneqq \inf_{x \in \Omega_P} \lVert x-y\rVert^2 / 2 - \phi_0(x)$ minimizes the optimization problem \eqref{eq:kantdual}, see Theorem \ref{thm:brenier_infinite}.
We call $\phi_0$ a \emph{Kantorovich potential}.

\textbf{Brenier potential.}
An alternative formulation of~\eqref{eq:kantdual} can be obtained by defining $\varphi \coloneqq \lVert \cdot \rVert^2/2 - \phi$ and $\tilde{\varphi} \coloneqq \lVert \cdot \rVert^2/2 - \tilde{\phi}$, assuming that $ P$ and $Q$ have finite second moments, resulting in:
\begin{align}
        \inf_{\substack{\varphi \in L^1(P) ,\tilde{\varphi} \in L^1(Q) \\ \varphi(x) + \tilde{\varphi}(y) \ge \langle x,y \rangle}} \int \varphi(x) \ \d P(x) + \int \tilde{\varphi}(y) \ \d Q(y). \label{eq:kantdual2}
\end{align}
For a given $\varphi \in L^1(P)$, an optimal $\tilde{\varphi}$ that solves~\eqref{eq:kantdual2} is the Legendre-Fenchel conjugate of $\varphi$, defined as
    $\varphi^{*}(y) \coloneqq \sup_{x \in \Omega_P} \left\langle x,y \right\rangle - \varphi(x)$. 
Taking $\tilde{\varphi}=\varphi^*$ in~\eqref{eq:kantdual2} results in the \emph{semi-dual problem}:
\begin{align}\label{eq:semidual}
        \inf_{\varphi \in L^1(P)} S(\varphi) \coloneqq \inf_{\varphi \in L^1(P)} \int \varphi(x) \ \d P(x) + \int \varphi^{*}(y) \ \d Q(y).
\end{align}
An optimal solution $\varphi_0$ to this problem is called a \emph{Brenier potential}.

\subsection{Brenier's theorem in infinite-dimensional space.} \label{sec:brenier}

We introduce theorems that guarantee the existence and uniqueness of the OT map and the potentials on infinite-dimensional spaces.
While the celebrated Brenier's theorem~\cite{Brenier91} states 
the existence of potentials and the uniqueness of the OT map, 
in a finite-dimensional cases,
we need additional discussion for the infinite-dimensional case. 
In particular, we introduce some definitions regarding measures on infinite-dimensional spaces below:
\begin{definition} \label{def:gauss}
Let $\mH$ be a separable Hilbert space. We define the following notion:
    \begin{enumerate}
        \item[(i)] A measure $\mu$ on $\mH$ is a \emph{nondegenerate Gaussian} measure if, for any continuous linear map $f: \mH \to \R$, $\mu \circ f^{-1}$ is a one-dimensional Gaussian measure.
        \item[(ii)] A Borel-measurable set $E \subset \mH$ is a \emph{Gaussian null set} if $\mu(E)=0$ for any nondegenerate Gaussian measure $\mu$ on $\mH$.
        \item[(iii)] A measure $P$ on $\mH$ is \emph{Gaussian-regular} if $P(E)=0$ for any Gaussian null set $E$.
    \end{enumerate}
\end{definition}

We introduce a general extension of Brenier's theorem.
This theorem not only guarantees the uniqueness of the OT map $T_0$ and the existence of the potentials $\phi_0$ and $\varphi_0$, but also clarifies the relationship among them.
\begin{theorem}[{Brenier's Theorem on Hilbert Spaces~\cite[Theorem 6.2.10]{ambrosio2005gradient}}] \label{thm:brenier_infinite}
    Let $\mH$ be a separable Hilbert space, and $P,Q$ be probability measures on $\mH$ with support $\Omega_P$ and $\Omega_Q$, respectively. If $P$ is regular (see Definition \ref{def:gauss}), $Q$ has a bounded support, and both $P$ and $Q$ have finite second moments, then the following statements hold:
    \begin{enumerate}
      \setlength{\parskip}{0cm}
      \setlength{\itemsep}{0cm}
        \item[(i)] There exists a unique transport map $T_0:\Omega_P \to \Omega_Q$ that solves~\eqref{eq:Monge}.
        \item[(ii)] There exists a convex function $\varphi_0 \in L^1(P)$ such that $\nabla \varphi_0 = T_0$. 
        \item[(iii)] The associated pair $(\phi_0,\phi^c_0)$ solves the dual Kantorovich problem~\eqref{eq:kantdual}.
    \end{enumerate}
\end{theorem}

To apply this theorem to our analysis, we fix the Hilbert space $\mH$ to be $\ell^2$, and we regard the set $\Omega \subset [0,1]^\infty$ as a bounded subset of $\ell^2$ via its natural embedding.

\section{Function space for optimal transport maps and potentials} \label{sec:map_space}

\subsection{$\gamma$-smooth space} \label{sec:mixedaniso}

We give a characteristic of a space of the OT map $T_0 \in \overline{\mF}$ being estimated.
Specifically, we introduce the notion of $\gamma$-smooth spaces, which has been used in the analysis on the Besov spaces with infinite-dimensional inputs \cite{schmeisser1987unconditional,nikol2012approximation,okumoto2021learnability,takakura2023approximation}. The $\gamma$-smooth spaces, in particular, are generalizations of the commonly used function spaces such as the Sobolev space on finite-dimensional spaces, and also flexibly represent multiple function spaces over infinite-dimensional spaces.
We will discuss this point in Section \ref{sec:relspaces}.

First, we construct a basis for $L^2([0,1]^{\infty})$.
For $l \in \Z^{\infty}_0$ such that $l_i \not= 0$ for finitely many $i \in \N$, we denote by $\psi_l:[0,1]^{\infty} \to \R$ an infinite-dimensional Fourier basis with an input $x = (x_1,x_2,...)^\top \in [0,1]^\infty$:
\begin{equation} \label{eq:basisfunc}
    \psi_l(x) \coloneqq \prod_{i=1}^\infty \psi_{l_i}(x_i), \quad \text{where } \psi_{l_i}(x_i) \coloneqq
    \begin{cases}
        \sqrt{2}\cos(2\pi \lvert l_i\rvert x_i), \quad  &l_i < 0 \\
        \sqrt{2}\sin(2\pi \lvert l_i\rvert x_i), \quad  &l_i > 0 \\
        1, \quad  &l_i = 0,
    \end{cases}
\end{equation}
By the Stone-Weierstrass theorem, $\{\psi_l(\cdot)\}_{l \in \Z^{\infty}_{0}}$ forms a complete orthonormal system of $L^2([0,1]^{\infty})$, hence we have that any $f\in L^2([0,1]^{\infty})$ can be written as $f=\sum_{l\in \Z^{\infty}_0} \left\langle f, \psi_l \right\rangle\psi_l$.
For a tuple of \textit{infinite dyadic scales} $s=  (s_1,s_2,...)^\top \in \N^{\infty}_0$, we define a \textit{projection operator} $\delta_s:L^2([0,1]^{\infty}) \to L^2([0,1]^{\infty})$ for $f \in L^2([0,1]^\infty)$ as
$$
\delta_s(f) :=  \sum_{l\in \Z^{\infty}_0-\bm{0}: \lfloor 2^{s_i-1} \rfloor \le \lvert l_i\rvert < 2^{s_i}}  \left\langle f, \psi_l \right\rangle\psi_l. 
$$
This operator restricts coefficients of $f $ with the $i$-th axis to those with frequencies $\approx 2^{s_i}$.

We introduce a function space for {OT maps} using the basis. 
With a \textit{smoothness map} $$\gamma:\mathbb{N}^{\infty}_0 \to \mathbb{R}_{>0},$$
that we will specify later, we introduce a norm $\lVert\cdot\rVert_{H^{\gamma}}$ for controlling the smoothness of $f \in L^2([0,1]^{\infty})$ by penalizing components $\delta_s(f)$ with large frequencies $s \in \mathbb{N}^{\infty}_0$. 
In particular, given $\gamma(\cdot)$, we define the $\gamma$-space:
\begin{definition}[$\gamma$-space] \label{def:gammaspc}
For $f \in L^2([0,1]^\infty)$, we define a $H^\gamma$-norm as 
\begin{align}
    \lVert f\rVert_{H^{\gamma}} 
    =\lVert f\rVert_{H^{\gamma}([0,1]^{\infty})} 
    = \left( \sum_{s\in \N^{\infty}_0} 2^{2\gamma(s)}  \lVert \delta_s(f)\rVert^{2}_2 \right)^{1/2}.
\end{align}
Then, the $\gamma$-space 
of $f : [0,1]^\infty \to \R$ 
is defined as $$ H^{\gamma}([0,1]^{\infty}) \coloneqq \left\{ f \in L^2([0,1]^\infty)  \mid  \lVert f\rVert_{H^{\gamma}} 
    < \infty \right\}.$$
Further, for a map $T: [0,1]^\infty \ni x \mapsto T(x) = (T_1(x),T_2(x),...) \in \ell^2$, we define its $\bar{H}^\gamma$-norm as $\|T\|_{\bar{H}^\gamma}^2 
:= \sum_{j=1}^\infty \|T_j\|_{H^\gamma}^2$. 
The $\gamma$-space of a map $T:[0,1]^\infty \to \ell^2$ is defined as 
\begin{align}
    H^{\gamma}([0,1]^{\infty}; \ell^2) \coloneqq \left\{ T: [0,1]^\infty \to \ell^2  \mid \|T\|_{\bar{H}^\gamma} < \infty \right\}.
\end{align}

\end{definition}

This space requires that the projection of $T_j$ to the $s$-th dyadic subspace $\delta_s(T_j)$ decays like $o(2^{-\gamma(s)})$, where the rate of decay is determined by the function $\gamma(\cdot)$. 
For example, the $\gamma$-space with $\gamma(s) = \max\{s_1, 2s_2\}$ consists of functions whose projections decay linearly along $x_1$ and quadratically along $x_2$, while leaving the other coordinates uncontrolled.

\subsection{Example of $\gamma$-smooth spaces} \label{sec:relspaces}

We give examples with a specific choice of $\gamma(\cdot)$. 
For each choice, we introduce an \textit{inverse smoothness index of} $\gamma(\cdot)$ as 
\begin{align}
    \alpha(\gamma):= \sup_{s \in \N^{\infty}_0} (\gamma(s))^{-1}\sum_{i=1}^\infty s_i, \label{def:inverse_index}
\end{align}
which measures the maximum ratio between the sum of frequencies in $s$ and $\gamma(s)$. Thus, a $\gamma$-space with a smaller $\alpha$ has smoother functions. 
Note that we allow $\alpha(\gamma) = \infty$ without the proof breaking down.
Several claims given in this section are rigorously proven in Section \ref{sec:proof_relspaces}.

In the first example, we consider the case where the input variables are finite-dimensional for simplicity and demonstrate that it can reproduce the commonly used Sobolev spaces.

\begin{example}[$d$-dimensional Sobolev spaces] \label{exp:d-dim-sobolev}
Let $d \in \N$ be dimension of an input space $[0,1]^d$ and $k \in \N$ be an index of smoothness.
For the indices $s \in \N_0^\infty$, we set $$\gamma (s) = \gamma^{d,k}(s) := k\cdot \max\{s_1,\ldots, s_d\},$$ then we have $\alpha(\gamma^{d,k}) = d/k$.
In this case, the corresponding $\gamma$-space $H^{\gamma}([0,1]^\infty)$ with setting $\gamma(\cdot) = \gamma^{d,k}(\cdot)$ corresponds to the $d$-dimensional Sobolev space of $k$-th order $H^k([0,1]^d)$ defined as follows:
\begin{equation*}
    H^k([0,1]^d) \coloneqq \left\{ f \in L^2([0,1]^{\infty}) : \sum_{l \in \Z^d_0} (1+\lVert l \rVert^2)^k \langle f , \psi_l  \rangle^2 < \infty  \right\}.
\end{equation*}
\end{example}

In the examples that follow, we assume the input dimension is infinite-dimensional set $[0,1]^\infty$, as originally specified. 
We also introduce an infinite sequence of \textit{frequency weights} $a =(a_i)_{i=1}^{\infty} \in \R_{> 0}^\infty$. For simplicity, we assume this sequence is nondecreasing.

\begin{example}[Mixed-smooth space] \label{ex:mixed-space}
Given a nondecreasing frequency weights $a \in \R_{> 0}^\infty$, we set 
\begin{equation}\label{eq:gamma1}
\gamma (s) = \gamma^{a,1}(s) := \sum_{i=1}^\infty a_i s_i,
\end{equation}
and we have $\alpha(\gamma^{a,1}) = a_1^{-1}$. We consider a corresponding $\gamma$-space $H^{a,1}([0,1]^\infty)$ as following Definition \ref{def:gammaspc}.
\end{example}

This function space evaluates the sum of the smoothness of each of an infinite number of elements of $x \in [0,1]^\infty$, weighted by an element with exponent $a$. Here, the inverse of the smoothness exponent is chosen as $a_1^{-1}$, utilizing the fact that it is nondecreasing.

The next example of function spaces is similar to the mixed-smooth space in Example \ref{ex:mixed-space} and evaluates the smoothest part.

\begin{example}[Anisotropic Sobolev spaces] \label{ex:anisotropic}

Given a monotonically nondecreasing frequency weights $a \in \R_{> 0}^\infty$, we set 
\begin{equation}\label{eq:gammainf}
\gamma(s) = \gamma^{a,\infty}(s) = \max_{i\in \N} a_is_i,    
\end{equation} 
then we have $\alpha(\gamma^{a,\infty}) = \sum_{i=1}^{\infty} a^{-1}_i$.
In this case, the $\gamma$-space $H^{\gamma}([0,1]^\infty)$ with setting $\gamma(\cdot) = \gamma^{a,\infty}(\cdot)$ is identical to the anisotropic Sobolev space \cite{Ingster2011} defined by
\begin{equation*}
    H^{a,\infty}([0,1]^{\infty}) \coloneqq \left\{ f \in L^2([0,1]^{\infty}) : \sum_{l \in \Z^{\infty}_0} \left\{\sum_{j=1}^{\infty} (2\pi \lvert l_j \rvert)^{2a_j}\right\} \langle f , \psi_l  \rangle^2 < \infty  \right\}.
\end{equation*}
\end{example}
The anisotropic Sobolev spaces can be regarded as a generalization of the $d$-dimensional Sobolev space in Example \ref{exp:d-dim-sobolev}. Specifically, if we set the sequence $a$ as $a_1=a_2=\ldots=a_d=k$ and $a_i = \infty$ for all $i > d$, we obtain the correspondence.

\subsection{$\gamma$-space for potential functions}

With the notion of the $\gamma$-space, we introduce an additional function space for the Brenier/Kantorovich potential functions that, naturally, requires one higher order of regularity compared to $H^{\gamma}([0,1]^{\infty})$ since the potentials are related to the OT maps through their gradients. To this end, we introduce a norm $$\lVert f\rVert_{H^{\gamma+2}} = \left( \sum_{s\in \N^{\infty}_0-\{\bm{0}\}} 2^{ 2(1+2\alpha(\gamma))\gamma(s)} \lVert \delta_s(f)\rVert^{2}_2 \right)^{1/2},$$ and denote the corresponding $\gamma$-space by $H^{\gamma+2}([0,1]^\infty)$. 
Roughly speaking, the $2$ factor in $2\alpha(\gamma)$ stems from the need for controlling the Hessian of the potentials.

We now introduce a function space for the Brenier potential $\varphi_0$, based on Kantorovich potentials, as a subset of $\eta$-strongly convex ($\langle \nabla \varphi(x) - \nabla \varphi(x'), x-x' \rangle \geq \eta \|x-x'\|^2, x,x' \in \Omega_P$) functions: 
\begin{align} \label{eq:Fbrenier}
    \mF_\varphi^{\gamma} = \left\{  {\varphi = \lVert \cdot\rVert^2}/{2} - \phi: \lVert \phi\rVert_{H^{\gamma+2}} \le 1 \mbox{~and~$\varphi$~is~$\eta$-strongly~convex}  \right\},
\end{align}
with $\eta > 0$.
Here, $\phi$ plays a role of the Kantorovich potential, which we assume to be ``($\gamma+2$)-smooth'' so that the first-order derivatives of $\phi$, which constitute the OT map $T_0$ will end up being $\gamma$-smooth. 
Using the set of Brenier potentials $\mF_\varphi^{\gamma}$, we define the function set for OT maps as 
\begin{align}\label{eq:FT}
    \mF_T^{\gamma} \coloneqq  \left\{ \nabla \varphi(\cdot) \in H^\gamma([0,1]^\infty; \ell^2) : \|\nabla \varphi(\cdot)\|_{\Bar{H}^\gamma} \leq 1, \varphi(\cdot) \in \mF_\varphi^{\gamma}   \right\}.
\end{align}

\begin{remark}[Approach with the Kantorovich functions]
    We characterize the class of smooth OT maps by associated Kantorovich potentials $\phi_0$. Indeed, if we employ a standard approach by the Brenier potential $\varphi_0$ as in the related works \cite{hutter2021minimax,manole2021plugin,divol2022optimal}, its strong convexity, which is required to preserve the injectivity of OT maps, violates periodicity and thus precludes the Fourier coefficient-based smoothness characterization. By contrast, since the Kantorovich potential does not demand convexity, it offers a viable framework for handling smoothness.
\end{remark}

\begin{remark}[$\gamma$-space and estimation error]
    We define the $\gamma$-space for functions using a $\lambda^\infty$-based Fourier expansion. In contrast, when evaluating the estimation error, we analyze it using the $L^2$ risk with respect to $P$. It is important to note that there is a difference here.
\end{remark}

\section{Fundamental limits: Lower bound of minimax risk} \label{sec:lower}

\subsection{Statement}
We derive one of our main results, a lower bound on the minimax risk \eqref{def:minimax_risk}.
Specifically, in the setting where the OT map $T_0$ belongs to the set of functions $\mF_T^\gamma$, we derive a lower bound of $\mathscr{R}_n(\mF_T^\gamma)$ in terms of the sample size $n$.

\begin{theorem} \label{thm:lower_bound} 
    With a bounded set $\Omega \subset [0,1]^\infty$ with respect to $\|\cdot\|$, consider a set of probability measures $\mP(\Omega)$ which satisfies the conditions of Theorem \ref{thm:brenier_infinite}.
    We also consider the smoothness map $\gamma(\cdot)$ is nondecreasing in each component of infinite-dimensional dyadic scales $s \in \N_0^\infty$.
    Then, we have the following lower bound on the minimax risk as $n \to \infty$:
    \begin{equation}
        \label{eq:g-smooth}
        \mathscr{R}_n(\mF_T^\gamma )\gtrsim   n^{-{2} / ({2+\alpha(\gamma)})}.
    \end{equation}
\end{theorem}
This result implies that (i) the minimax risk has a polynomial rate in $n$, which is in contrast to the fact that more general smoothness achieves a logarithmic rate in $n$ in an infinite-dimensional setup discussed in Section \ref{sec:sobolevellipsoid}, (ii) this rate depends solely on smoothness and is dimension-free, and (iii) the rate is identical to the rate of function estimation in the regression problem~\cite{Nussbaum85}, which is in contrast to the fact that the rate for OT maps is slower than usual in that of regression in the case of finite-dimensional inputs \cite{hutter2021minimax,divol2022optimal}. 
In other words, only in the infinite-dimensional case, the problems of estimating OT maps and the regression function are of equal complexity.

\subsection{Proof outline}
We present an outline of the proof of Theorem \ref{thm:lower_bound}.

\textbf{Strategy.}
We apply the technique of deriving lower bounds on risks based on information theory, which derives a lower bound when the problem of estimating parameters is attributed to the problem of testing from discrete points in the parameter space. 
Toward the goal, we apply the following result taken from \cite{tsybakov2008introduction}:
\begin{theorem}[Theorem  2.5 in \cite{tsybakov2008introduction}, adapted] \label{thm:tsybakov_lower}
    Let $\Theta := \{\theta_0,\theta_1,...,\theta_K\}$ be a collection of hypotheses and be equipped with a norm $\|\cdot \|$.
    Let $P_\theta$ be a distribution for each $\theta \in \Theta$.
    Suppose that the following hold with $s > 0$: (i) $\|\theta_j -\theta_{j'}\| \geq 2s > 0$ for every $0 < j < j' < K$, and (ii) $P_{\theta_0}$ is absolutely continuous with respect to $P_{\theta_j}$ for every $j \in [K]$,  and (iii) $K^{-1} \sum_{j=1}^K D(P_{\theta_j} \Vert P_{\theta_0}) \leq C \log K$ with $C \in (0,1/8]$.
    Then, it holds that
    \begin{align}
        \inf_{\bar{\theta}} \sup_{\theta \in \Theta} \Pr (\|\Bar{\theta} - \theta \| \geq s) \geq \frac{\sqrt{K}}{1 + \sqrt{K}} \left( 1 - 2C - \sqrt{2C (\log K)^{-1}} \right),
    \end{align}
    where $\bar{\theta}$ is taken from all possible estimators depending on $n$ samples.
\end{theorem}
Using this result, we can achieve the lower bound when the corresponding distance between distributions is small, even between parameter points that are some distance apart.
A key point of this proof is to construct a finite set $\tilde{\mF} \subset \mF_T^\gamma$ that satisfies the conditions required by Theorem \ref{thm:tsybakov_lower}.

\textbf{Construction.}
We give an outline of the construction of $\tilde{\mF}$.
In preparation, given $S \in \N$ and $d \in \N$, we consider a family of functions $g_l$ indexed by the set $I(S) \coloneqq \{l \in \N^d_0, 0 \le l_i \le 2^S-1\}$.
Note that $\psi_l (x)$ is a product of cosines for any $l\in I(S)$. Denote $M \coloneqq \lvert I(S)\rvert = 2^{dS}$. With an abuse of notations, we denote $\gamma(\bm{S})$ where $\bm{S} = (S,S,...,S) \in \N^d$ by $\gamma(S)$.
Then, we define a collection of functions indexed by elements in $I(S)$:
\[
g_0 \equiv 0, \qquad g_l(x) = (2\sqrt{2}\pi)^{-2}2^{-\gamma(S)}M^{-1/2}  \lVert l \rVert^{-1}  \psi_l(x), \qquad l\in I(S) - \{0 \}. 
\]
We also define binary vectors $\tau^{(0)},\ldots,\tau^{(K)}\in \{0,1\}^M$ with some $K \ge 2^{M/8}$.
Using the functions and vectors, we define a collection of Brenier potentials $ {\mH}_\varphi^\gamma := \{\varphi_m:[0,1]^d \to \R, m=1,...,K\}$ whose elements are defined as
\[
\varphi_m(x) = \frac{1}{2} \lVert x\rVert^2 + \sum_{l\in I(S)} \tau^{(m)}_l g_l(x), \qquad m=0,\ldots,K.
\]
Finally, we consider a finite subset of $\mF_T^\gamma$ as ${\mH}_T^{\gamma} \coloneqq \{ \nabla \varphi(\cdot) \in H^\gamma([0,1]^\infty; \ell^2) : \varphi(\cdot) \in {\mH}_\varphi^\gamma \}$ and define probability distributions $P_0 = \operatorname{Uniform([0,1]^d)}$ and $Q_m = (\nabla \varphi_m)_{\#}P_0$ for $m = 1,\ldots, K$.
With this construction, we can show that (i) the transport maps in ${\mH}_T^{\gamma}$ are well-separated from each other, and (ii) the associated probability distributions $Q_m$'s are all indistinguishable from $P$ in terms of the KL divergences $D(Q_m\Vert P_0)$. 
Then, we make $d$ sufficiently large in order to obtain the lower bound.
Loosely speaking, these results imply the difficulty of learning OT maps based on samples from these distributions, and more precisely, they satisfy the conditions of Theorem \ref{thm:tsybakov_lower} and yield the desired lower bound.

\textbf{Technical challenge.}
The technical challenge in our derivation of the lower bound lies in deriving \textit{dimension-independent constants}. 
In particular, our construction carefully removes hidden constants that grow with the dimension $d$ of the input space.
While dimension-dependent constants are often acceptable in finite-dimensional settings as \cite{hutter2021minimax} and \cite{divol2022optimal}, they become problematic when $d \to \infty$. 
For instance, analysis on $D(P_{\theta_j} \| P_0)$ for Theorem \ref{thm:tsybakov_lower} yields a bound $D(P_{\theta_j} \| P_0) \lesssim d2^{-2\gamma(S)}2^{2S}$, where the factor $d$ would be treated as a constant diverging to infinity as $d \to \infty$. 
To avoid this issue, we control the design of $\mathbf{S}$ and $\gamma(S)$ to eliminate an effect of the diverging constant.
To the best of our knowledge, our work is the first to establish lower bounds for this problem that are fully valid in the infinite-dimensional setting by explicitly avoiding dimension-dependent constants.

\section{Attainability: Upper bound of minimax risk} \label{sec:upper}

We derive an upper bound for the minimax risk. 
Specifically, we will develop an estimator $\hat{T}$ and study its  error $\|\hat{T} - T_0\|^2_{L^2(P)}$, which 
provides an upper bound of the minimax risk as
    \begin{equation}
        \mathscr{R}_n(\mF_T^\gamma) \leq \mathbb{E}\lVert \hat{T} - T_0(x) \rVert^2_{L^2(P)}.
    \end{equation}
We assume without loss of generality that $\int \phi_0  \d P = 0$, since the optimal transport is invariant to shifting the Kantorovich potential $\phi_0$ by a constant, 

\subsection{Construction of the estimator} We introduce an estimator of the OT map $T_0$ and extending the empirical estimation strategy by \cite{hutter2021minimax} to our infinite-dimensional case.
Firstly, we introduce an estimator for the Brenier potential $\varphi_0$, using a function space spanned by the Fourier basis up to frequency sum of $2^J$ with $J \in \N$.
To formalize this approach, we define the function space $\mF_J \subset H^{\gamma+2}([0,1]^{\infty})$ for Kantorovich potentials:

    \begin{equation} \label{eq:defFJ}
    \mF_J \coloneqq \left\{ \phi_J = \sum_{\substack{s\in \N^{\infty}_0: \\ (1+2\alpha(\gamma))\gamma(s) \le J}}\sum_{\substack{l \in \Z^{\infty}_0: \\ \lfloor 2^{s_i-1} \rfloor \le \lvert l_i\rvert < 2^{s_i}}} \omega_l\psi_l(x): \int \phi_J  \d P = 0 , \lVert \phi_J \rVert_{H^{\gamma+2}} \le 1 \right\}. 
\end{equation}
We also define a function space for Brenier potentials by the relation $\varphi_0 = \|\cdot\|^2 /2 - \phi_0$:
\begin{align}\label{eq:defFstarJ}
    \mF^{*}_J \coloneqq \{ {\lVert \cdot \rVert^2} / {2} - \phi_J: \phi_J \in \mF_J \}.
\end{align}
Then, we define an estimator for the Brenier potential $\varphi_0$ by an empirical analogue of the semi-dual optimization problem \eqref{eq:semidual} as
\begin{equation}\label{eq:samplesemidual}
  \hat{\varphi}_J  := \argmin_{\varphi \in \mF^*_J} \left\{ n^{-1} \sum_{i=1}^n \varphi(X_i) + n^{-1} \sum_{i=1}^n \varphi^{*}(Y_i) \right\}.
\end{equation}
Finally, we define an estimator for the OT map $T_0$ by following the relation $T_0 = \nabla \varphi_0$ as
$$\hat{T}_J= \nabla \hat{\varphi}_J.$$

\subsection{Statement}
We state an upper bound of the minimax risk $\mathscr{R}_n(\mF_T^\gamma)$, using an upper bound of an estimation error of the developed estimator $\hat{T}_J$ with a choice of $J$.
We now impose additional assumptions on $P$:
\begin{assumption} \label{asmp:condition_P}
    The probability measure $P \in \mP(\Omega)$ satisfies the following \emph{Poincar\'e inequality} for some $C > 0$ for any $f \in L^2([0,1]^\infty)$:
    \[ \int (f(x)-\mathbb{E}_P[f(X)])^2\,dP(x)
 \le
 C\int \|\nabla f(x)\|^2\,dP(x). \]
\end{assumption}
This is a common condition necessary for $P$ and its support to remain regular (see \cite{divol2022optimal} for details). 
We give an example of a class of distributions $P$ that satisfy Assumption \ref{asmp:condition_P}.

\begin{proposition} \label{prop:assumption_measure}
    Fix $(c_1,c_2,...) \in \ell^2$.
    For $i = 1,2,...$, consider any countable product of log-concave distributions $P_i$ with its support $[0,|c_i|]$ : $P=\bigotimes_{i=1}^{\infty} P_i$. In this case, each $P_i$ satisfies the Poincar\'e inequality with the constant bounded by the variance, which is bounded by $1/4$. By the tensorization property, $\bigotimes_{i=1}^{d} P_i$ satisfies the Poincar\'e inequality in Assumption \ref{asmp:condition_P} with $C=1/4$. By taking the limit $d \to \infty$, $P$ also satisfies the inequality with $C=1/4$.
\end{proposition}

Then, with the frequency weights $a = (a_i)_{i=1}^{\infty} \in \R_{\geq 0}^\infty$, we obtain the following statement:
\begin{theorem} \label{thm:upper_bound}
    With a bounded set $\Omega \subset [0,1]^\infty$ with respect to $\|\cdot\|$, 
    suppose that $P,Q \in \mP(\Omega)$ satisfy the conditions of Theorem~\ref{thm:brenier_infinite}, and further suppose that Assumption \ref{asmp:condition_P} holds.
    Suppose that $\gamma(\cdot)=\gamma^{a,1}$ or $\gamma(\cdot)=\gamma^{a,\infty}$ 
    satisfies $\alpha(\gamma) \le 1$, and that $a=(a_i)_{i=1}^{\infty}$ satisfies $a_i = \Omega(i^q)$ for some $q> 0$.
    Then, with a setting $J \asymp \frac{1+2\alpha(\gamma)}{2+\alpha(\gamma)}\log_2 n$ and obtain the following bound:
    \begin{equation}
        \label{eq:upper}
        \mathscr{R}_n(\mF_T^\gamma)  \lesssim   n^{-{2} / ({2+\alpha(\gamma)})} (\log n)^2.
    \end{equation}
\end{theorem}
The results show the following implications. First, most importantly, the convergence rate $O(n^{-{2} / ({2+\alpha(\gamma)})} (\log n)^2)$ of the derived upper bound given here coincides with the order of the lower bound $\Omega(n^{-{2} / ({2+\alpha(\gamma)})})$ in Theorem \ref{thm:lower_bound}, except for the logarithmic terms $O((\log n)^2)$.
In other words, we have identified the order of minimax risk $\mathscr{R}_n(\mF_T^\gamma)$ of the OT map estimation in the infinite-dimensional setting. 
Second, needless to say, this order of risk is not affected by (infinite) dimension, but only by smoothness. 
This resolves the curse of dimensionality of the OT map estimation problem.
Third, this rate $\tilde{O}(n^{-{2} / ({2+\alpha(\gamma)})})$  is identical to that of the regression problem.
This coincidence is a distinctly different result from the finite-dimensional case and is characteristic of the infinite-dimensional setting.

\subsection{Proof outline}

Our proof starts with the preparation of potential functions corresponding to the OT map and the estimator, then uses a technique called \textit{localization} on Kantorovich potentials.
Further, we utilize the technique of an \textit{interpolation space}, which controls norm of the potential functions without convexity of the estimator.
Specifically, we follow the following three steps.

\textbf{Step I}. We decompose the error by some potentials.
Recall that the Brenier potentials satisfies $\nabla\varphi_0=T_0$ and $\nabla \hat{\varphi}_J = \hat{T}_J$.
We also consider the corresponding Kantorovich potentials $\phi_0 \coloneqq \lVert \cdot\rVert^2/2 - \varphi_0$  and  $\hat{\phi}_J \coloneqq \lVert \cdot\rVert^2/2 - \hat{\varphi}_J$. 
We, then, write these potentials as Fourier series for the infinite-dimensional setup:
\begin{equation*}
    \phi_0 = \sum_{l \in \Z^{\infty}_0} \omega^0_l\psi_l, \mbox{~~and~~}
    \hat{\phi}_J = \sum_{s\in \N^{\infty}_0, (1+2\alpha(\gamma))\gamma(s) \le J}\sum_{l \in \Z^{\infty}_0: \lfloor2^{s_i-1} \rfloor \le \lvert l_i\rvert < 2^{s_i}} \hat{\omega}_l \psi_l,
\end{equation*}
where $\omega_l^0$ and $\hat{\omega}_l$ are corresponding coefficients.
Let $\bar{\phi}_J$ be the truncation of $\phi_0$ to the dyadic scales $s$ such that $(1+2\alpha(\gamma))\gamma(s) \leq J$; in other words, 
\[
    \bar{\phi}_J \coloneqq \sum_{s\in \N^{\infty}_0, (1+2\alpha(\gamma))\gamma(s) \le J}\sum_{l \in \Z^{\infty}_0: \lfloor2^{s_i-1} \rfloor  \le \lvert l_i\rvert < 2^{s_i}} \omega^0_l \psi_l.
\]
Using the notion, we can decompose the estimation error $\|\hat{T}_J - T_0\|_{L^2(P)}$ as
\begin{align}
    \|\hat{T}_J - T_0\|_{L^2(P)} &= 
    \|\nabla \hat{\varphi}_J - \nabla \varphi_0\|_{L^2(P)} \\ &= \|\nabla \hat{\phi}_J - \nabla \phi_0\|_{L^2(P)} \\&\leq \|\nabla \hat{\phi}_J - \nabla \bar{\phi}_J\|_{L^2(P)} + \|\nabla \bar{\phi}_J - \nabla \phi_0\|_{L^2(P)}. ~~~~\label{ineq:decomp_upper}
\end{align}
In the right-hand side, the first term is a stochastic error term, and the second term is an approximation term.

\textbf{Step II}. To bound the stochastic error term $\|\nabla \hat{\phi}_J - \nabla \bar{\phi}_J\|_{L^2(P)}$ in \eqref{ineq:decomp_upper}, we apply the \emph{localization technique} \cite{VanDeGeer2002}, which has been used for stochastic approximation of OT maps \cite{hutter2021minimax,divol2022optimal}. The technique comprises approximating $\hat{\phi}_J$ by $\hat{\phi}_t \coloneqq t \hat{\phi}_J + (1-t) \bar{\phi}_J$, where $t \coloneqq {\tau} / ({\tau + \lVert \nabla \hat{\phi}_J - \nabla \bar{\phi}_J\rVert_{L^2(P)}})$
and $\tau > 0$ is the localization parameter to be defined later. 
Then, $\hat{\phi}_t$ is localized near $\bar{\phi}_J$ in the sense that $\lVert \nabla\hat{\phi}_t - \nabla \bar{\phi}_J\rVert_{L^{2}(P)} \leq \tau$ holds. 
We can then write the stochastic error term as:
\begin{align*} 
\lVert \nabla\hat{\phi}_J - \nabla \bar{\phi}_J\rVert_{L^{2}(P)} &= t^{-1}\lVert \nabla\hat{\phi}_t - \nabla \bar{\phi}_J\rVert_{L^{2}(P)} \\
&= \left(1 + \tau^{-1}\lVert \nabla\hat{\phi}_J - \nabla \bar{\phi}_J\rVert_{L^{2}(P)}\right)\lVert \nabla\hat{\phi}_t - \nabla \bar{\phi}_J\rVert_{L^{2}(P)} \\
&= \lVert \nabla\hat{\phi}_t - \nabla \bar{\phi}_J\rVert_{L^{2}(P)} + \tau^{-1}\lVert \nabla\hat{\phi}_J - \nabla \bar{\phi}_J\rVert_{L^{2}(P)} \lVert \nabla\hat{\phi}_t - \nabla \bar{\phi}_J\rVert_{L^{2}(P)}.
\end{align*}
Our goal now is to choose $\tau$ as an increasing function of $n$ so that $\lVert \nabla\hat{\phi}_t - \nabla \bar{\phi}_J\rVert_{L^{2}(P)}$, which depends on $\tau$, is a decreasing function of $n$. As a result, we can combine the second term in the right-hand side with the left-hand side in the above display, resulting in a bound for $\lVert \nabla\hat{\phi}_J - \nabla \bar{\phi}_J\rVert_{L^{2}(P)}$ in terms of $\lVert \nabla\hat{\phi}_t - \nabla \bar{\phi}_J\rVert_{L^{2}(P)}$.

To find an appropriate choice of $\tau$, we first derive a finite-sample bound for $\lVert \nabla\hat{\phi}_t - \nabla \bar{\phi}_J\rVert_{L^{2}(P)}$. To this end, we consider the \emph{empirical process} with a function $f\in L^2([0,1]^\infty)$ as 
\[ \mathcal{E}_n(f) \coloneqq \int_{\Omega_P} f \, \d P -  \frac{1}{n}\sum_{i=1}^n f(X_i). \] 
We then show that, with high probability, $\lVert \nabla\hat{\phi}_t - \nabla \bar{\phi}_J\rVert_{L^{2}(P)}$ can be bounded in terms of the supremum of the empirical process over the function space $\mathcal{F}(\bar{\phi}_J)=\{\phi \in \mathcal{F}_J : \lVert \nabla \phi - \nabla \bar{\phi}_J\rVert_{L^{2}(P)} \leq \tau\}$ as follows (proof in Appendix \ref{sec:proofe_upper}):
\begin{align*}
    \lVert \nabla\hat{\phi}_t - \nabla \bar{\phi}_J\rVert^2_{L^{2}(P)} \lesssim \sup_{\phi \in \mathcal{F}(\bar{\phi}_J)}  \lvert \mathcal{E}_n(\phi) \rvert + \|\nabla \bar{\phi}_J - \nabla \phi_0\|^2_{L^2(P)}.
\end{align*}

\textbf{Step III}. The remaining is to bound the empirical process term $\sup_{\phi \in \mathcal{F}(\bar{\phi}_J)}  \lvert \mathcal{E}_n(\phi) \rvert$ and the approximation term $\|\nabla \bar{\phi}_J - \nabla \phi_0\|_{L^2(P)}$, separately. 
For the approximation term, an approximate evaluation (Lemma \ref{lemma:approx_error}) using the Fourier analysis on infinite-dimensional space gives 
\begin{equation}
    \lVert \nabla \bar{\phi}_J - \nabla \phi_{0} \rVert^2_{L^2(P)} \lesssim 2^{-{2J} / ({1+2\alpha(\gamma)})}.
    \end{equation}
For the empirical process term $\sup_{\phi \in \mathcal{F}(\bar{\phi}_J)}\lvert \mathcal{E}_n(\phi) \rvert$, we extend empirical process techniques~\cite{Vaart1998} with the interpolation space \cite{mendelson2002improving} to our infinite-dimensional spaces $H^\gamma([0,1]^\infty)$. To this end, we derive a covering number for $\gamma$-space $H^\gamma([0,1]^\infty)$ (Lemma \ref{lem:covering_number}), which in turn allows us to achieve the following bound with high probability for some some constant $C>0$:
 \begin{align*}
 \sup_{\phi \in \mathcal{F}(\bar{\phi}_J)} \lvert \mathcal{E}_n(\phi) \rvert  \lesssim  \frac{\tau2^{\frac{\alpha(\gamma)}{2(1+2\alpha(\gamma))} J }\sqrt{J \log (1+C/\tau)}}{\sqrt{n}} + \frac{J 2^{\frac{\alpha(\gamma)}{1+2\alpha(\gamma)} J} }{n}. 
\end{align*}
Choosing $\tau=2^{J \alpha(\gamma)/(2(1+2\alpha(\gamma)))}\sqrt{J/n}$ and $J = \frac{1+2\alpha(\gamma)}{2+\alpha(\gamma)}\log_2 n$ yields the bound in 
\eqref{eq:upper}.
\qed

\textbf{Technical challenge.}
We highlight two technically important contributions to provide the above proof: (i) the control of infinite‑dimensional gradients/Hessians of Brenier potentials, and (ii) a new covering number lemma for the $\gamma$-smooth space. We provide details bellow.

About the first contribution on the control of gradients/Hessians, we have developed from the ground up the notions of Fr\'echet gradient and Hessian of Kantorovich potentials in infinite-dimension, and show that they are well-controlled for $\gamma$-smooth functions. For example, we derive the following result:
\begin{lemma}[Hessian of Kantorovich potentials: informal version of Lemma \ref{lemma:frechetgrad}]
    Let $\phi_0 \in H^{\gamma+2}$ be the Kantorovich potential.
    Suppose that $\alpha(\gamma) \leq 1$ holds.
    Then, the Hessian $\nabla^2 \phi_0$ of $\phi_0$, which is a linear map from $[0,1]^\infty$ to the space of bounded linear operators, satisfies
    \begin{align}
        \|\nabla^2 \phi_0\|_{\mathrm{op}} \leq 4 \pi^2 \|\phi_0\|_{H^{\gamma+2}}.
    \end{align}
\end{lemma}
This result is crucial for associating gradients in infinite-dimensions with $\gamma$-smooth spaces, and establishing tight bounds necessary for Theorem \ref{thm:upper_bound}.
The condition $\alpha(\gamma) \leq 1$ is necessary to control the Fr\'echet gradients and Hessians of $\gamma$-smooth Kantorovich potentials (see Lemma \ref{lemma:oneandtwo}).

About the second contribution, we derive the covering number bound for $\gamma$-smooth function spaces, which appears for the first time in nonparametric estimation over $\gamma$-spaces to the best of our knowledge.
Specifically, we derive the following bound:
\begin{lemma}[Covering number of $\gamma$-smooth space: informal version of Lemma \ref{lem:covering_number}]            Assume that $\gamma = \gamma^{a,1}$ or $\gamma = \gamma^{a,\infty}$. Then, we have
    \begin{equation}
     \log  N(\epsilon, \mF_J, \ell^{\infty}_s(2^{\frac{\alpha(\gamma)}{1+2\alpha(\gamma)}J}\lVert \delta_s(\cdot)\rVert_{L^2([0,1]^{\infty})})) \lesssim J 2^{\frac{\alpha(\gamma)}{1+2\alpha(\gamma)} J} + 2^{\frac{\alpha(\gamma)}{1+2\alpha(\gamma)} J} \log_2 \left(\epsilon^{-1} \right),
    \end{equation}
    where $\ell^\infty_s(w\|\delta_s(\cdot)\|)$, given $w>0$, is a norm of $f \in \mF_J$ defined as $w \sup_{s} \|\delta_s(f)\|$.
\end{lemma}
This bound plays a key role to study the supremum of the empirical process term $\sup_{\phi \in \mathcal{F}(\bar{\phi}_J)}\lvert \mathcal{E}_n(\phi) \rvert$, and is crucial for obtaining the polynomial rate upper bound in the infinite-dimensional setup.

\subsection{Discussion on the rate of the minimax risk}

We finally summarize our main result on the minimax risk $\mathscr{R}_n(\mF_T^\gamma)$, which is immediately obtained from Theorem \ref{thm:lower_bound} and \ref{thm:upper_bound}.
\begin{corollary} \label{cor:upplow}
    Consider the setup of Theorem \ref{thm:lower_bound} and Theorem \ref{thm:upper_bound}.
    Then, we obtain the following result up to logarithmic factors in $n$:
    \begin{align}
        \mathscr{R}_n(\mF_T^\gamma) =  \Theta ( n^{-2/(2+\alpha(\gamma))} ).
    \end{align}
\end{corollary}
This result fully determines the minimax risk rate for OT maps on infinite-dimensional spaces, excluding the influence of logarithmic terms. To reiterate, the key features here are: (i) the rate is dimension-independent, enabling the treatment of infinite-dimensions, and (ii) it coincides with the rate for function estimation on infinite-dimensions \cite{Nussbaum85}, remaining unaffected by the restriction to the class of OT maps.

Using these results, we derive specific min-max risk rates for each concrete example of $\gamma$-smooth spaces considered above. Table \ref{tab:example_rates} shows the result. Each value is characterized by the smoothness map $\gamma (\cdot)$ and the frequency weights $a = (a_1,a_2,...) \in \N_{\>0}^\infty$.
In the case of the Sobolev space with $d$-dimensional inputs, the obtained rate ${\Theta}(n^{-2k/(2k+d)})$ coincides with those of density estimation and nonparametric regression with target functions in Sobolev spaces; see \cite[Chapter 2.8]{tsybakov2008introduction} and references therein.

\begin{table}[htbp]
\centering
\begin{tabular}{c||c|c|c}
$\gamma$-smooth space & smoothness map $\gamma(\cdot)$ & index $\alpha(\gamma)$ & minimax rate \\ \hline \hline
$d$-dim Sobolev (Ex. \ref{exp:d-dim-sobolev}) & $\gamma^{d,k}(s) := k\cdot \max\{s_1,\ldots, s_d\}$ & $d/k$ &  ${\Theta}(n^{-2k/(2k+d)})$ \\ \hline 
Mixed-smooth (Ex. \ref{ex:mixed-space}) & $\gamma^{a,1}(s) := \sum_{i=1}^\infty a_i s_i$ & $a_1^{-1}$ &  $\Theta(n^{-2a_1 / (2a_1 + 1)})$ \\ \hline
Anisotropic Sobolev (Ex. \ref{ex:anisotropic}) & $\gamma^{a,\infty}(s) = \max_{i\in \N} a_is_i$ & $\sum_{i=1}^\infty a_i^{-1}$  & $\Theta(n^{-2\tilde{a} / (2\tilde{a} + 1)})$ \\ 
\end{tabular}
\caption{Minimax rates of the examples of the $\gamma$-smooth spaces with ignoring logarithmic factors in $n$. Each smoothness map $\gamma(\cdot)$ is characterized by the frequency weights $a = (a_1,a_2,...) \in \N_{\>0}^\infty$. We denote $\tilde{a} = (\sum_i a_i^{-1})^{-1}$}.
\label{tab:example_rates}
\end{table}

\section{Application to neural network estimator and functional data analysis}

We propose practical methods and applications based on the theoretical results above. The first is a practical estimator for the OT mapping using neural networks, and the second is optimal transport for functional data using the neural estimator.
The design of the neural estimator and the evaluation of its error apply our theory developed in previous sections, specifically the estimation using semi-dual methods and the evaluation of the covering number.

\subsection{Neural network model for high-dimensional input}

We propose a neural network estimator for estimating the OT mapping, based on the following motivation.
While the estimator $\hat{\varphi}_J$ developed in Section \ref{sec:upper} is near-minimax optimal, it requires calculate over $J^d$ elements, which is computationally intractable for large $J$ and $d$. 
To solve this problem, we utilize the semi-dual objective \eqref{eq:samplesemidual}  developed for the theoretical purposes in the learning of neural networks.

We introduce a neural network model on $[0,1]^\infty$.
As preparation, we select a finite number of elements from an input $x \in [0,1]^\infty$  and put them into a neural network.
In particular, given $J>0$ and $\gamma (\cdot)$, we determine the maximum number of nontrivial dimensions of functions as 
\begin{align}
    d_{\max} \coloneqq d_{\max}(J, \gamma) \coloneqq \max \{i \in \N : \exists s \in \N^\infty_0 - \{\bm{0}\}, (1+2\alpha(\gamma))\gamma(s) < J \}. \label{def:d_max}
\end{align}
Next, given hyper-parameters $W,L \in \N$ and $R,B > 0$, we consider neural networks with $L$-layers, $W$-width, and the ReLU activation function.
For an $\ell$-th layer for $\ell = 1,...,L$, we introduce a weight matrix $W_\ell \in \R^{W \times W}$ for $\ell = 2,...,L-1$, $W_1 \in \R^{W \times d_{\mathrm{max}}}$, and  $W_L \in \R^{1 \times W}$.
Also, we introduce bias vectors $b_\ell \in \R^W$ for $\ell = 1,...,L-1$ and $b_L \in \R$.
The ReLU activation $\sigma: \R^k \to \R^k$ is defined as $\sigma((z_1,...,z_k)) \coloneqq (\max \{z_1, 0 \},...,\max \{z_k, 0 \})$.
Then, we define a function space for neural Kantorovich potentials as follows:
 \begin{equation} \label{eq:defFJnn}
 \begin{split}
    \widetilde\mF_J(W,L,R,B) & \coloneqq \biggl\{ \phi: [0,1]^{d_{\mathrm{max}}} \to \R, \phi(x) = (W_L \sigma(\cdot) + b_L) \circ \cdots \circ (W_1 x + b_1) \Bigm\vert  \\
    &\int \phi \, \d P = 0,
    \sum_{\ell=1}^L (\lVert W_{\ell} \rVert_0+ \lVert b_{\ell}\rVert_0) \leq R,
\max_{\ell} \lVert W_{\ell}\rVert_\infty + \rVert b_{\ell}\rVert_\infty \leq B \biggr\}.
  \end{split}
\end{equation}
Here, $R$ represents the maximum number of non-zero parameters in a neural network, serving as a measure of sparsity, and is typically controlled through methods like $L^1$ regularization. $B$ denotes the maximum value of parameters within the network, and is actually controlled through techniques such as clipping. Both play crucial roles in describing the size of the function set generated by the neural network and are important factors in characterizing the generalization error of the trained neural network.
The zero integral condition can be approximately achieved by subtracting by the empirical mean: $\phi \to \phi - n^{-1}\sum_{i=1} \phi(X_i)$. We also define a function space for Brenier potentials: 
\begin{align}\label{eq:defFstarJnn}
    \widetilde\mF^{*}_J(W,L, R, B) \coloneqq \{ {\lVert \cdot \rVert^2} / {2} - \phi: \phi \in \widetilde{\mF}_J(W,L, R, B) \}.
\end{align}
For brevity, we shall denote $\widetilde\mF_J \coloneqq \widetilde\mF_J(W,L,R,B)$ and $\widetilde\mF^*_J \coloneqq \widetilde\mF^*_J(W,L,R,B)$. 

\subsection{Neural estimator for functional data}
We provide how to implement the neural network estimator with functional data on an interval $I \subset \R$.
We consider that there exist unknown source measure $\hat{P}$ and target measure $\hat{Q}$ on a function space $L^2(I)$.
Suppose that we have $L^2(I)$-valued $n$ observations independently generated from $\widehat{P}$ and $\widehat{Q}$ as
\begin{align}
    \hat{X}_1,...,\hat{X}_n \sim \hat{P}, \qquad  \hat{Y}_1,...,\hat{Y}_n \sim \hat{Q}.
\end{align}
For simplicity, we assume that $\Ep[\hat{X}_i(t)] = \Ep[\hat{Y}_i(t)] =0$ for every $i = 1,...,n$ and $t \in [0,1]$.
The goal of this problem is to estimate the OT map from $\hat{P}$ to $\hat{Q}$, which is a typical example of the optimal transport theory for the functional data analysis \citep{zhu2021functional}. 

We apply the discrete Fourier transform or sine/cosine transform on the functional data $\tilde{X}_i$.
In particular, by the orthonormal system of basis functions $\{u_j:I \to \R\}_{j=1}^\infty$ of $L^2(I)$ for the transform, we obtain an infinite sequence of coefficients as
\begin{align}
    X_i \coloneqq (w_{i,1},w_{i,2},...)^\top \in [0,1]^\infty,  ~~i=1,...,n,
\end{align}
where $w_{i,j}$ is a coefficient defined as
\begin{align}
    w_{i,j} \coloneqq c_1 \int_I \hat{X}_i(t)  u_j(t) dt  + c_2 \in [0,1],  ~~j = 1,2,...,\infty,
\end{align}
and also obtain $Y_i$ by the same manner from $\hat{Y}_i$.
Here, $c_1 > 0$ and $c_2 \in \R$ are properly selected scaling coefficients to make $w_{i,j}$ belong to $[0,1]$.
Let $P$ be the distribution of $X_i$ and $Q$ be the distribution of $Y_i$.
Based on the transformation above, we consider an estimation problem for the OT map $T_0$ from $P$ to $Q$. Since the methods for obtaining $\hat{P}$ from $P$ and $\hat{Q}$ from $Q$ are trivial, this is equivalent to the problem of estimating the OT map from $\hat{P}$ to $\hat{Q}$.

Further, we approximate the infinite sequences $X_i$ and $Y_i$ by truncating them up to $d_{\mathrm{max}}$ number of elements as \eqref{def:d_max}. 
We define
\begin{align}
    x_i \coloneqq   (w_{i,1}, w_{i,2},...,w_{i,d_{\mathrm{max}}})^\top \in [0,1]^{d_{\mathrm{max}}}, ~~ i=1,...,n.
\end{align}
We also obtain $y_i \in  [0,1]^{d_{\mathrm{max}}}$ by the same manner.
This operation is crucial for reshaping input variables $X_i$ of infinite-dimensions into a form suitable for neural networks that only accept finite-dimensions. By increasing this dimension $d_{\max}$ at an appropriate rate in $n$, we can show that the dimension truncation does not affect a rate of an estimation error.

We define the neural network estimator as a solution of the empirical semi-dual problem with the truncated $d_{\mathrm{max}}$-dimensional sequences $x_i,y_i$ from the functional data.
In particular, we consider the following semi-dual problem over the class of neural networks $\tilde\mF^*_J$: 
\begin{align}    
\hat{\varphi}_{nn} \coloneqq \argmin_{\varphi \in \tilde\mF^*_J} \left\{ n^{-1} \sum_{i=1}^n \varphi(x_i) + n^{-1} \sum_{i=1}^n \varphi^{*}(y_i) \right\},  \label{def:semi-dual_fda}
\end{align}
which is an estimator for the Brenier potential in this setup.
Then, we obtain an estimator of the transport map by calculating the gradient:
\begin{equation*}
    \hat{T}_{nn} \coloneqq \mathrm{cl} \circ \nabla \hat{\varphi}_{nn}: [0,1]^{d_{\mathrm{max}}} \to [0,1]^{d_{\mathrm{max}}}.
\end{equation*}
$\mathrm{cl}:\R^{d_{\max}} \to [0,1]^{d_{\max}}, (z_1,...,z_{d_{\max}})^\top \mapsto (\max\{\min\{z_1,1\},0\},...,\max\{\min\{z_{d_{\max}},1\},0\})^\top$ is a clipping operator to restrict its output to  $[0,1]^{d_{\mathrm{max}}}$.
The gradient of a neural network with respect to the input can be easily computed using any automatic differentiation library.
    
\subsection{Optimality of the neural estimator for functional data}

We demonstrate that neural network estimators achieve optimal rates in the OT map estimation problem for the function data.
In particular, we show that, with a specific network configuration $W,L,R,B$, the neural estimator can achieve near minimax-optimal rates. See Appendix~\ref{sec:nn_proof} for the proof of this theorem.
    \begin{theorem}\label{thm:upper_bound_nn}
        Suppose that $P$ and $Q$ satisfy the conditions of Theorem~\ref{thm:brenier_infinite} and Assumption \ref{asmp:condition_P}, and the corresponding OT map $T_0:\Omega \to \Omega$ belongs to $\mF_T^\gamma$ with  $\gamma(\cdot)=\gamma^{a,1}$ or $\gamma(\cdot)=\gamma^{a,\infty}$ and $a = (a_i)_{i=1}^\infty$ satisfying $a_i = \Omega(i^q)$ for some $q > 0$. 
        Then, by setting $J \asymp \frac{1+2\alpha(\gamma)}{2+\alpha(\gamma)}\log_2 n$ and the following neural network parameters:
        \[ W \asymp J^{1/q} 2^{\frac{\alpha(\gamma)}{1+2\alpha(\gamma)}J},~
        L \asymp J^{2+2/q},~
        R \asymp J^{2 + 4/q}2^{\frac{\alpha(\gamma)}{1+2\alpha(\gamma)}J},~
        B \asymp 2^{J^{1/q}/2}, \] 
        we have the following upper bound for the estimator $\hat{T}_{nn} = \mathrm{cl} \circ \nabla \hat\varphi_{nn}$: 
        \begin{align}
             &\Ep_{(X_{1:n}, Y_{1:n})} \left[\lVert \iota^{-1}\circ \hat{T}_{nn} \circ \iota - T_0 \rVert^2 \right] \lesssim   n^{-{2} / ({2+\alpha(\gamma)})} (\log n)^2,
        \end{align}
        with $\iota: [0,1]^\infty \ni (z_1,z_2,...)\mapsto (z_1,...,z_{d_{\max}}) \in  [0,1]^{d_{\max}}$,  and $\iota^{-1}: [0,1]^{d_{\max}} \ni (z_1,...,z_{d_{\max}}) \mapsto (z_1,...,z_{d_{\max}},0,0,...) \in [0,1]^\infty$.
    \end{theorem}
    
This result demonstrates that even neural network estimators can achieve minimax optimal rates shown in Corollary \ref{cor:upplow}, except for the logarithmic term in $n$.
There are two caveats to this result. First, while optimal orders for $J,W,R,B$ and $d_{\mathrm{max}}$ are specified, in practice they are often determined through cross-validation or similar methods. Second, the semi-dual problem \eqref{def:semi-dual_fda} is non-convex; while it is typically solved heuristically, demonstrating global convergence theoretically requires further study.

This estimator requires approximating the infinite-dimensional variables $X_i,Y_i$ by the $d_{\mathrm{max}}$-dimensional variables.
In practice, we will have $J = \lfloor \log_2 n \rfloor$, as suggested in the theorem, and $d_{\mathrm{max}}$ increases linearly with $J$.
This scheme enables us to handle the infinite-dimensional variables asymptotically by the $d_{\mathrm{max}}$-dimensional approximation.

\section{Simulation study} \label{sec:sims}

To demonstrate the dimension-free estimation errors of $\gamma$-smooth OT maps as implied by our theory, we simulate data with varied numbers of dimensions and
define a transport map that is $\gamma$-smooth.

As a data generation process, for $q \in \{1,1.3,2\}$, we sample $n$ $d$-dimensional points from $P = \operatorname{Unif}([0,1]^d)$. Our OT map $T_0:[0,1]^d \to [0,1]^d$ is a high-dimensional ``hockey stick'' function that becomes smoother by the dimensions. The $i$-th component of $T_0$ is:
\begin{equation*}
    (T_0(x))_i = x_i - \frac1{\kappa(i)}\lvert x_i - 0.5 \rvert^{\kappa(i)}, \quad \text{where } \kappa(i) = i^{0.1q} + q + 0.6. 
\end{equation*}
 The map is visualized in Figure~\ref{fig:1}. As the $l$-th Fourier coefficient of $\lvert x_i - 0.5 \rvert^{\kappa(i)}$ on $[0,1]$ is $O(l^{-1-\kappa(i)})$, for any $s_i \in \N$, we have $\sum_{l=2^{s_i}}^{2^{s_i+1}} l^{2(-1-\kappa(i))}= O(2^{-(1 + 2\kappa(i))s_i})$. Consequently, $T_0 \in H^{\kappa + 1/2,1}([0,1]^d)$, which is a mixed-smooth space. Therefore, $\alpha = a_1 = \kappa(1) + 1/2 = q-0.4$.

With $T_0$ defined above, we generate another sample of size $n$ from $P$, and then transport the sample points with $T_0$ to obtain a sample from $Q = (T_0)_{\#}P$.

For all of our experiments, we use a specific architecture for the neural Kantorovich potentials $\phi$ and $\tilde\phi$. 
With a specified $d' \in \N$, for an input $x = (x_{1},\ldots, x_{d})$, we scale each of its components in multiple directions: $x'_{i} = (\theta_{i1} x_i,\ldots, \theta_{id'} x_i)$, where $(\theta_{ij})_{1 \le i \le d, 1 \le j \le d'}$ are trainable parameters. Then our embedding is $x' = (x'_1,\ldots,x'_d)^\top \in \R^{d \times k}$.
In addition, the embeddings are fed into a convolutional neural network consisting of two layers, both with 10 $1 \times 1$ filters. The output of the final convolutional layers is then flattened and passed to several dense layers into a single output. For all layers, ReLU is used as the activation function.

\begin{figure}
    \centering
        \begin{subfigure}{0.45\textwidth}
            \includegraphics[width=\textwidth]{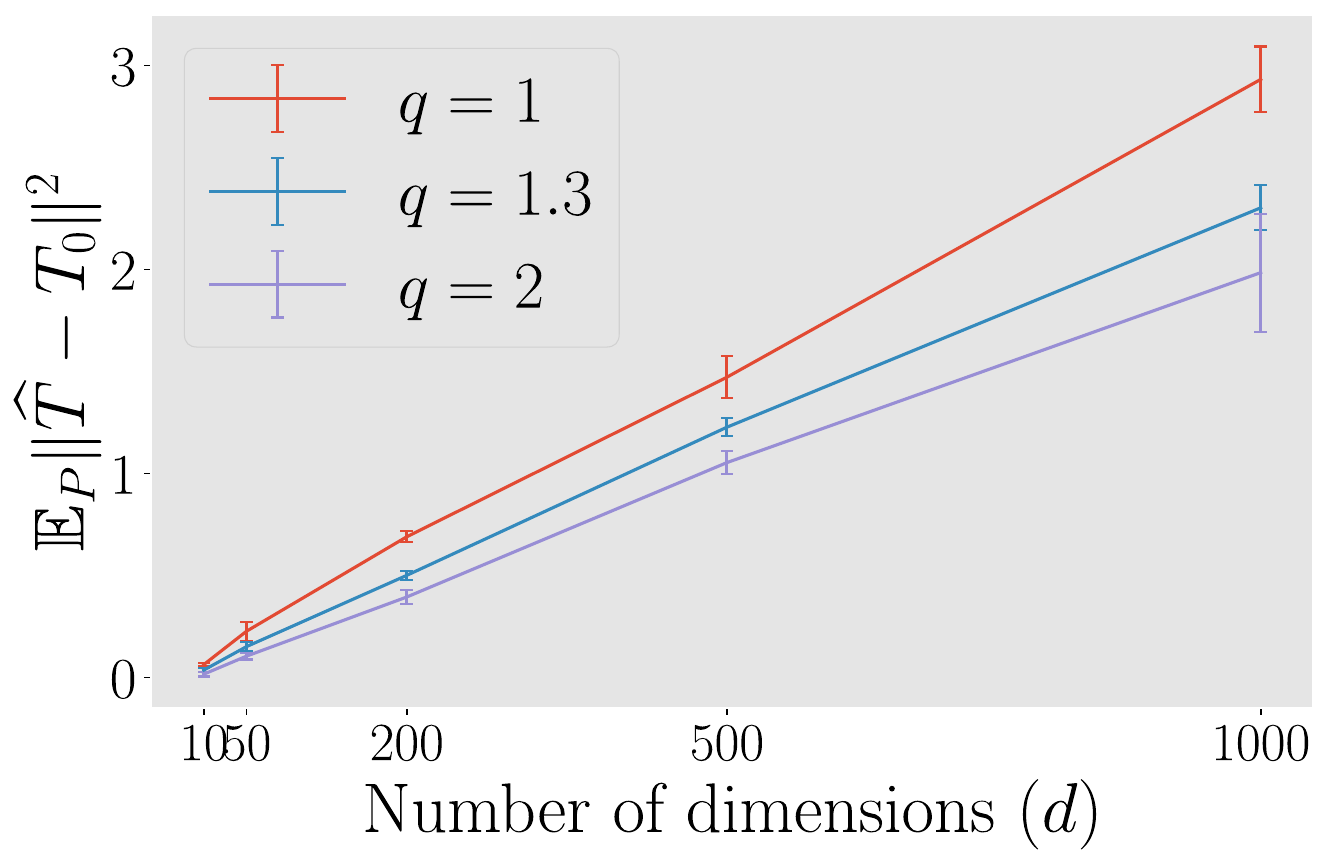}
            \caption{Estimation errors by the dimension \\with $n = 100$.}
            \label{fig:2b}
        \end{subfigure}
        \begin{subfigure}{0.54\textwidth}
            \includegraphics[width=\textwidth]{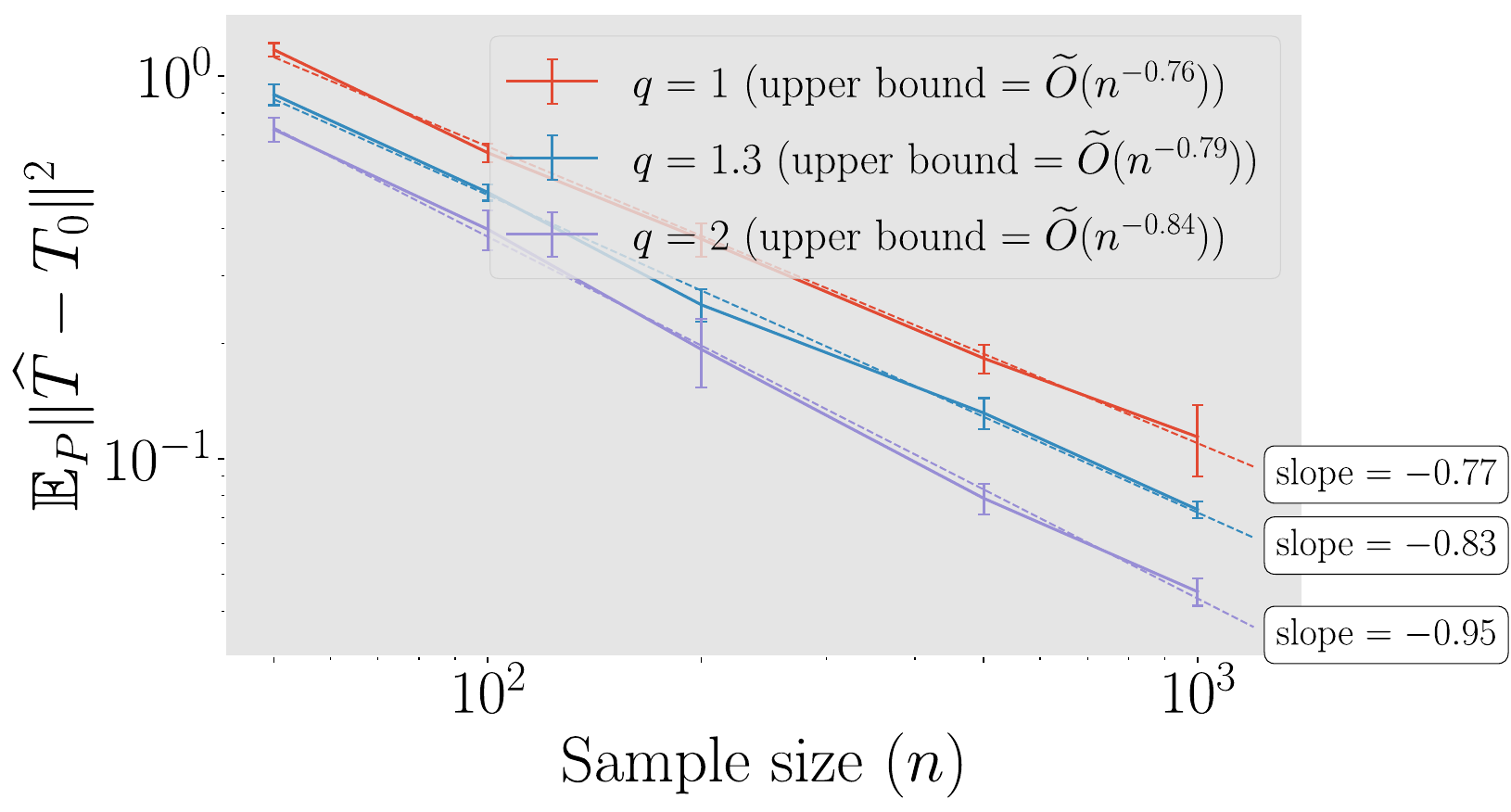}
            \caption{Estimation errors by the sample size \\with $d = 200$.}
            \label{fig:2c}
        \end{subfigure}
    \caption{(a) Quantitative validation of dimension-independence of the OT map estimation errors, and (b) Quantitative validation of polynomial rates of the estimation errors.}
    \label{fig:estimation_error_plot}
\end{figure}

\subsection{Dimension-independent estimation errors} 
We validate the theoretical claim that the estimation errors for $\gamma$-smooth OT maps on $d$-dimensional spaces are dimension free by training the neural estimator with varying dimensions $d \in \{10, 50, 200, 500, 1000\}$, $q\in\{1,1.3,2\}$ and fixed $n=100$. 
For each $q$ and $d$, we compute the estimation error: $\lVert T_0 - \hat{T}_{nn} \rVert^2_{L^2([0,1]^d)}$. 
The hyperparameters of the neural networks vary based on $q$ and $d$ as shown in Table~\ref{tab:hyperparameters}.

Figure~\ref{fig:2b} shows a plot of the estimation errors across the dimensions. We can see that the estimation error rates are linear and not exponential with the sample size.

\begin{table}[htbp]
    \centering
    \begin{tabular}{lcccc}
        \hline
        \multirow{2}{*}{\textbf{Parameter}} & \multicolumn{4}{c}{\textbf{Input Dimension} ($d$)} \\
        & 10 & 50 & 200 & 1000 \\
        \hline
        $q = 0$ & 1e-2 & 1e-2 & 1e-2 & 6e-3 \\
        $q = 1$ & 1e-2 & 1e-2 & 9e-3 & 6e-3 \\
        $q = 3$ & 1e-2 & 1e-2 & 1e-2 & 6e-3 \\
        \hline
    \end{tabular}
    \caption{Learning rates used in the numerical simulation for each $q$ and $d$. Additional fixed parameters for both networks include: training samples $n = 100$, batch size = 100, training iterations = 175, embedding dimension $d' = 20$. The CNN consists of two layers, both with 10 $1\times 1$ filters.
    For the training of neural networks, we use stochastic gradient descent with a cosine decay schedule. 
    }
    \label{tab:hyperparameters}
\end{table}

\subsection{Convergence rate} 
We compare the estimation error of the simulation and its theoretical upper bound derived in Theorem~\ref{thm:upper_bound}. To this end, we fix $d=200$ and train the neural Brenier potential with sample sizes $n \in \{50, 100, 200, 500, 1000\}$ and $q \in \{1,1.3,2\}$. 
Specifically, we compare the simulation result with the bound $\tilde{O}(n^{-{2} / ({2+\alpha(\gamma)})})=\tilde{O}(n^{-{2a_1}/  ({2a_1+1})})$ from Theorem~\ref{thm:upper_bound}, which is 
$\tilde{O}(n^{-0.76})$, $\tilde{O}(n^{-0.79})$ and $\tilde{O}(n^{-0.84})$ for $q=1,1.3$ and $2$, respectively.

Figure~\ref{fig:2c} shows a log-log plot of the estimation errors across the range of sample sizes for each value of $q$. 
Importantly, the slopes of the regression lines ($-0.77$, $-0.83$, $-0.95$) closely match the exponents of $n$ in the theoretical upper bounds ($-0.81$, $-0.86$, $-0.91$), supporting our theory.

\section{OT map estimation in functional data analysis} \label{sec:application}
We conduct experiments using real data, in order to demonstrate the practical utility of our neural estimator for the setup of OT with infinite-dimensional data.

We use OT maps to study the Daily Optimum Interpolation Sea Surface Temperature (OISST v2.1) \cite{Reynolds2007,Banzon2016,Huang2021}, which provides daily SST averages measured by the Advanced Very High-Resolution Radiometer (AVHRR) at a spatial grid resolution of $0.25^{\circ}$. Our study focuses on a specific region in the northeast Pacific Ocean between Hawaii and California from $128.12^{\circ}$W-$157.88^{\circ}$W longitude and $20.12^{\circ}$N - $49.88^{\circ}$N latitude, occupying a square grid of size $120\times 120$. We treat daily average temperatures across this region as continuous functions, where each grid point serves as an input and its corresponding temperature as the output. This yields two sets of 365 functions, representing daily average temperatures for both 1982 and 2023.
We process the temperature data by first converting the regional temperatures into two-dimensional single-channel images. We then apply cosine transform to these functions and truncate the transform coefficients to the first $40\times 40$ terms. These coefficients serve as inputs for our OT map estimation algorithm. 

To see how the SSTs' trends vary across the region, we plot the time series of the daily SSTs at four selected grid points in Figure~\ref{fig:10}. The plot shows that the transport map consistently shifts 1982 data upward by up to $4^{\circ}$C, with the trends closely following those of the actual SSTs in 2023. Notably, temperature increases are more pronounced in colder regions compared to warmer regions, particularly evident at the southernmost point.

We use the proposed neural Brenier potential with a 2D convolutional layer, which we shall refer to as RFF2D. 
We compare RFF2D against several OT map estimation algorithms: linear OT (Linear), functional OT (FOT) \cite{zhu2021functional}, amortized semi-dual solvers with CNN (ConvNet) \cite{Amos2023} and input-convex neural network (ConvICNN) \cite{Korotin2021}. 
To measure the discrepancy between the SST transports and the actual SSTs in 2023, we use the soft-dynamic time warping divergence (sDTW-div) \cite{Blondel21} and the average over the coordinate-wise dynamic time warpings (Avg-DTW).

\begin{figure}[t]
    \centering
        \begin{subfigure}[t]{0.36\textwidth}
            \includegraphics[width=\textwidth]{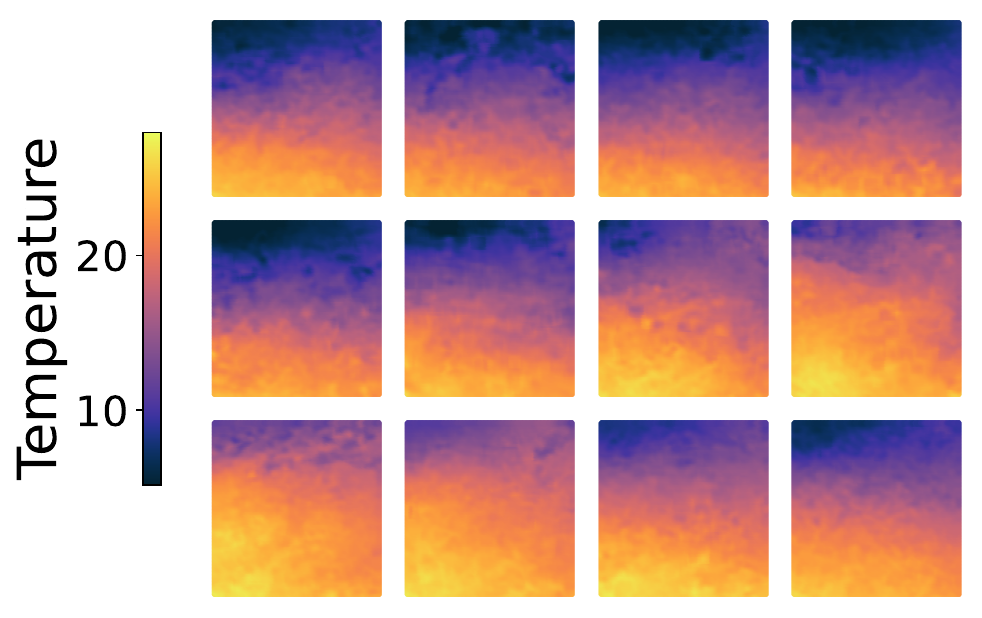}
            \caption{$X$ (1982)}
            \label{fig:5}
        \end{subfigure}
        \begin{subfigure}[t]{0.3\textwidth}
            \includegraphics[width=\textwidth]{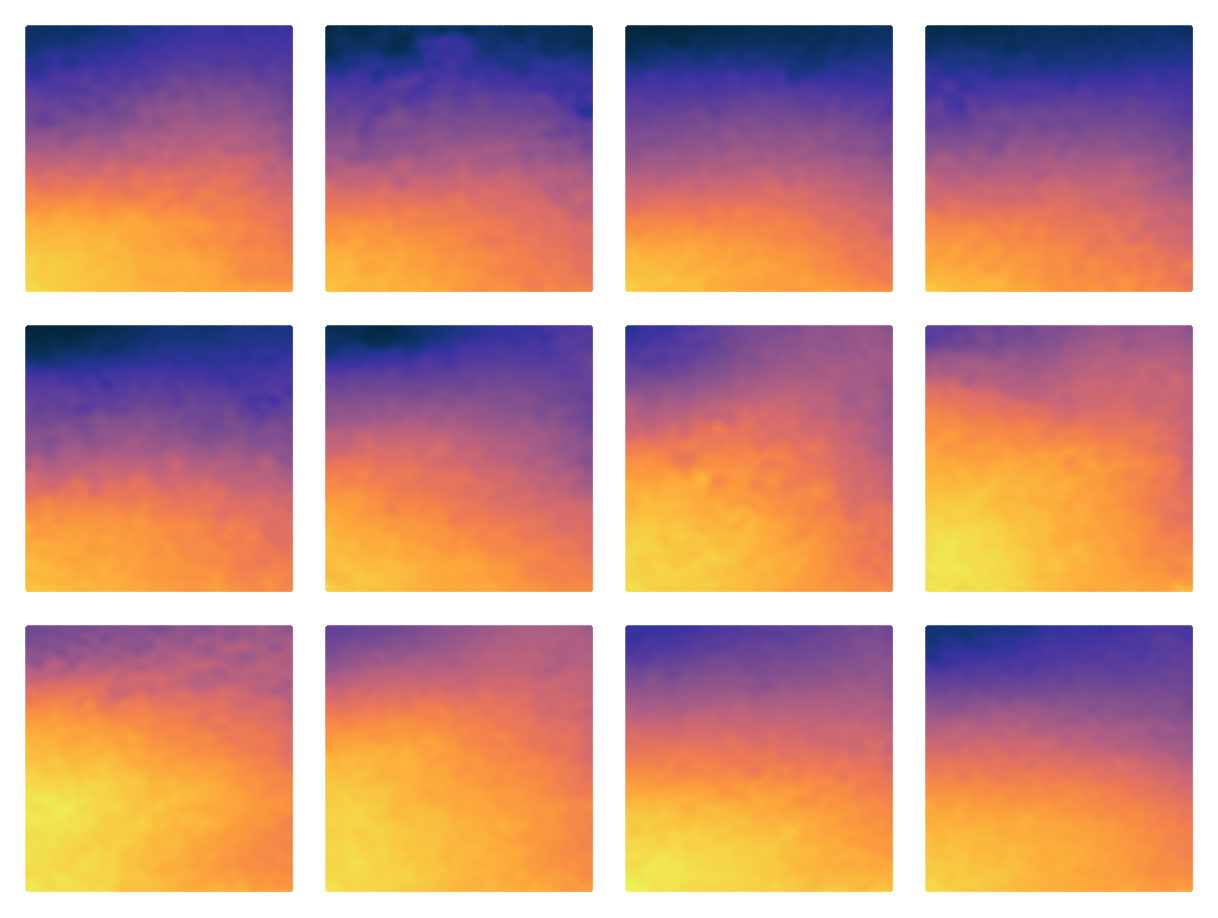}
            \caption{$\hat{T}(X)$}
            \label{fig:6}
        \end{subfigure}
        \begin{subfigure}[t]{0.3\textwidth}
            \includegraphics[width=\textwidth]{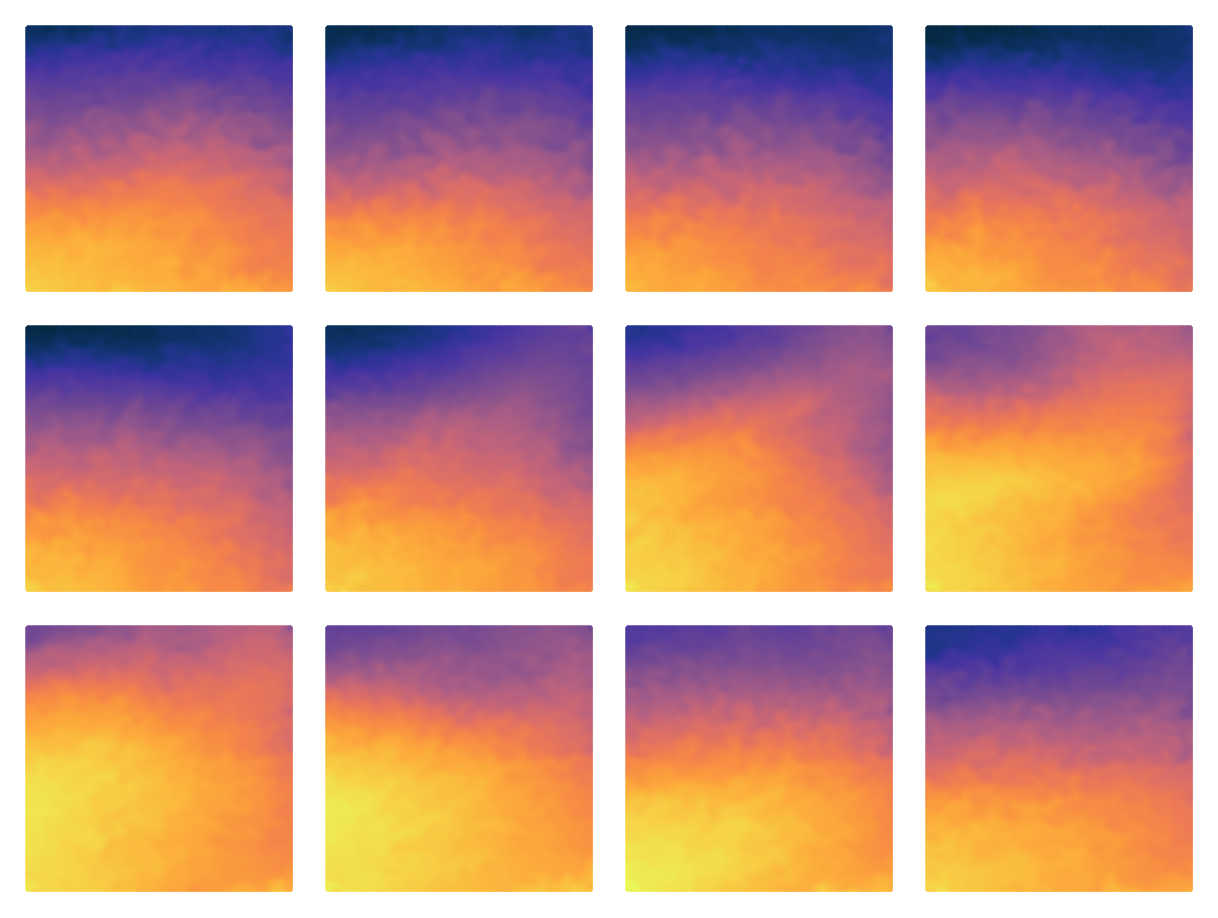} 
            \caption{$Y$ (2023)}
            \label{fig:7}
        \end{subfigure}\\
        \begin{subfigure}[t]{0.45\textwidth}
            \includegraphics[width=\textwidth]{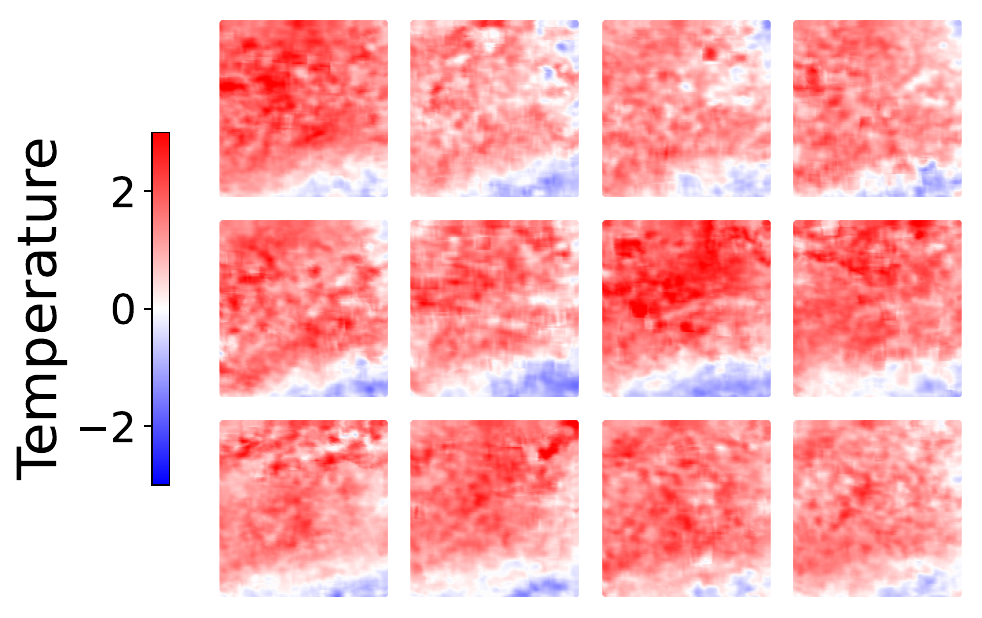}
            \caption{$\hat{T}(X)- X$}
            \label{fig:8}
        \end{subfigure}
        \begin{subfigure}[t]{0.375\textwidth}
            \includegraphics[width=\textwidth]{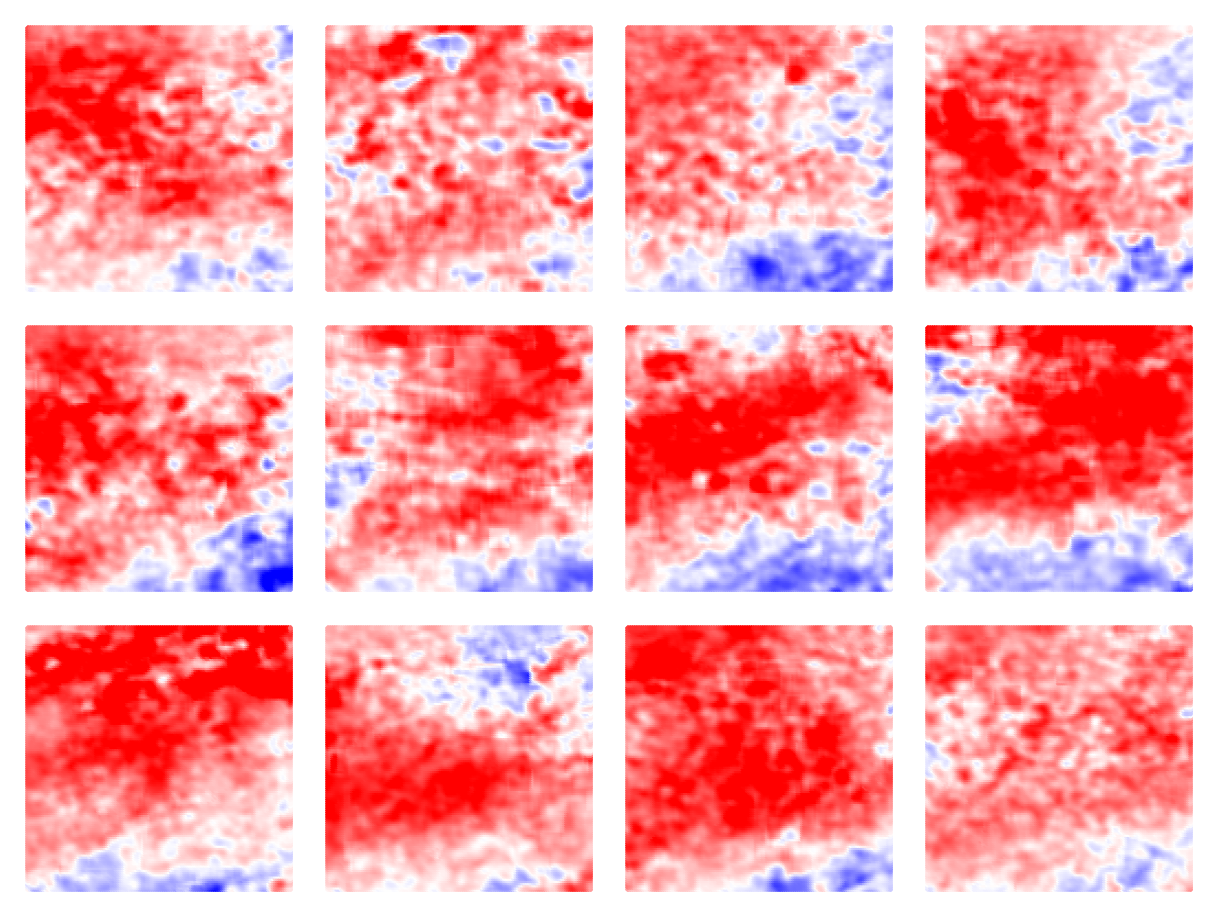} 
            \caption{$Y-X$}
            \label{fig:9}
        \end{subfigure}
    \caption{Monthly averages for 12 months of (a) SSTs in 1982, (b) Transports of SSTs in 1982, (c) SSTs in 2023, (d) The differences between SSTs in 1982 and the transports, and (e) The differences between SSTs in 1982 and SSTs in 2023. Warmer colors indicate higher values.}
    \label{fig:5-7}
\end{figure}

Figure~\ref{fig:5}-\ref{fig:7} show that our OT map's predictions are similar to the actual SSTs in 2023. To provide clearer views of both 2D series, we subtract them by the SSTs in 1982; see Figure~\ref{fig:8}-\ref{fig:9}. Both figures indicate increases in temperatures over almost all of the region throughout the year. We also notice that the actual SSTs from April to November 2023 are higher than our OT map's predictions, which align with the Copernicus's 2023 report \cite{copernicus2023} of ``record-breaking global-average SSTs from April through December''.
Table~\ref{tab:sst} shows that the neural semi-dual algorithm with RFF2D outperforms other algorithms, achieving the best scores in both sDTW-div and Avg-DTW.

\begin{table}
    \centering \small
    \begin{tabular}{cccccc}
    \cline{4-6}
    &        &                   & \multicolumn{3}{|c|}{Neural semi-dual solvers}                                                                                                             \\ \hline
    \multicolumn{1}{c}{Model}    & Linear & \begin{tabular}[c]{@{}c@{}}FOT\\ \cite{zhu2021functional}\end{tabular}               & \begin{tabular}[c]{@{}c@{}}RFF2D\\ (ours)\end{tabular} &  \begin{tabular}[c]{@{}c@{}}ConvICNN\\ \cite{Korotin2021}\end{tabular}             & \begin{tabular}[c]{@{}c@{}}ConvNet\\ \cite{Amos2023}\end{tabular}            \\ \hline
    \multicolumn{1}{c}{sDTW-div} & \multicolumn{1}{r}{638.5}  & \multicolumn{1}{r}{613.3 \deemph{$\pm 0.0$}} & \multicolumn{1}{r}{\cellcolor[HTML]{E2EBFB}\textbf{120.6} \deemph{$\pm 3.2$}}                                                                    & 208.9 \deemph{$\pm 16.8$} & 269.4 \deemph{$\pm 15.6$} \\
    \multicolumn{1}{c}{Avg-DTW}  & \multicolumn{1}{r}{30.8}   & \multicolumn{1}{r}{27.3 \deemph{$\pm 0.0\phantom{0}$}}    & \multicolumn{1}{r}{\cellcolor[HTML]{E2EBFB}\textbf{6.7} \deemph{$\pm 3.1$}}                                                                       & 11.2 \deemph{$\pm 1.0$}   & 12.7 \deemph{$\pm 0.4$}   \\ \hline
    \end{tabular}
    \caption{Comparison between OT map estimation algorithms (lower numbers are better).} 
    \label{tab:sst}
\end{table}

We evaluated different neural architectures and optimization approaches for learning the optimal transport map between the SSTs in year 1982 and 2023. Table~\ref{tab:sst-hyperparameters} presents the hyperparameters used for each model.
\begin{figure}[t]
    \centering
    \includegraphics[width=\textwidth]{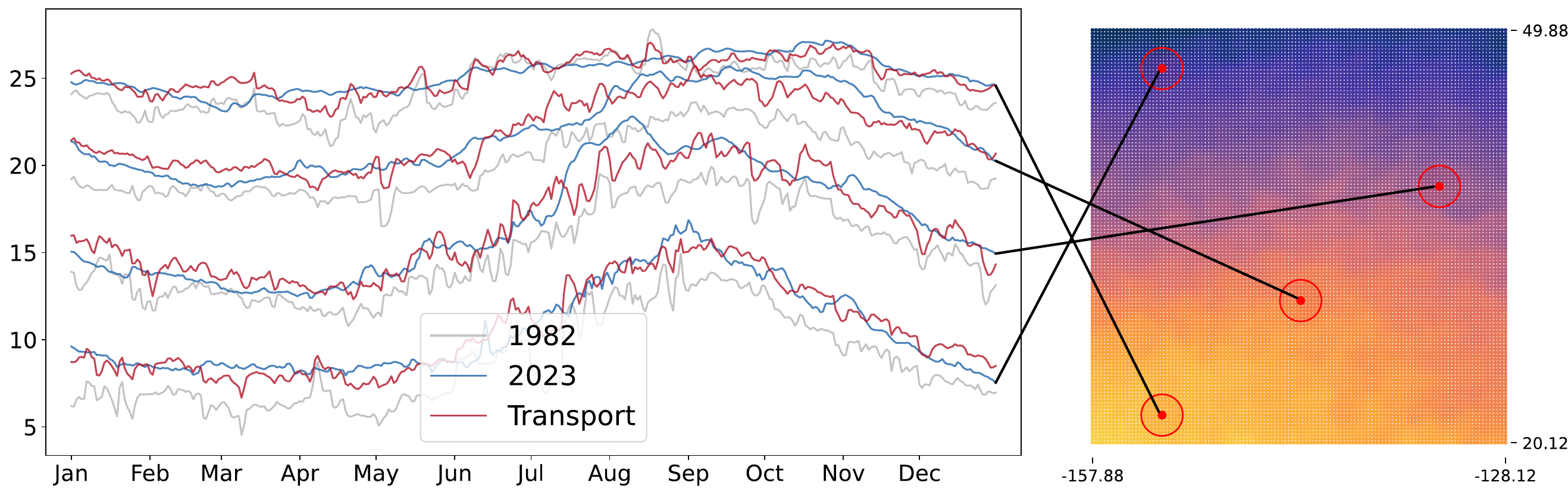}
    \caption{Daily sea surface temperatures at four different grid points.}
    \label{fig:10}
\end{figure}

\section{Poly-logarithmic rates of estimation with the $C^1$-class} \label{sec:sobolevellipsoid}

As a complement to our main study, we provide the minimax rate of learning the OT map with the $C^1$-class, which is closer to commonly used smoothness of functions~\citep{tsybakov2008introduction,hutter2021minimax}. In particular, we will show that learning the OT map in these spaces has a slow poly-logarithmic minimax rate. In contrast, our proposed functional spaces~\eqref{eq:FT} have a polynomial minimax rate as shown in Section~\ref{sec:lower}.

In preparation, we define a set of functions with coefficients from a Sobolev ellipsoid.
For any $\hat{x} \in L^2([0,1])$, we write $\hat{x}(t) = \sum_{j=1}^\infty \theta_j e_j(t)$ as an input function with a basis function $e_j \in L^2([0,1])$ and a coefficient $\theta_j$.
For brevity, we consider a case that $\theta_j \in [0,1]$ holds.
For the coefficient $\theta_1,\theta_2,...$, we define a set of infinite-dimensional ellipsoids for the coefficients:
\begin{align}
\Theta^{\infty}(b) := \left\{ \bm{\theta} \in [0,1]^{\infty} : \sum_{j=1}^{\infty} j^{2b}\theta_j^2 < 1 \right\}, \label{def:sobolev_ellipsoid}
\end{align}
with a smoothness index $b > 0$. The use of the Sobolev ellipse follows the standard setting for nonparametric regression \cite{tsybakov2008introduction}.

We consider the problem of estimating an OT map from $P \in \mP(\Theta^\infty(b))$ to $Q \in \mP(\Theta^\infty(b))$.
To introduce a function space for OT maps, we define the $C^k$-class of functions on the Sobolev ellipsoid. Given an integer $k \ge 0$, the space $\mathcal{C}^k(\Theta^\infty(b); \mathbb{R})$ consists of real-valued functions that admit derivatives up to order $k$, where the partial derivatives are defined in the usual sense with the bounded norm:
$\lVert f \rVert_{\mathcal{C}^k(\Theta^\infty(b); \mathbb{R})} \coloneqq \sum_{j=0}^k \sup_{\lvert \alpha \rvert = j} \lvert \partial^{\alpha} f \rvert < \infty. $
The space $\mathcal{C}^k(\Theta^\infty(b); \Theta^\infty(b))$ consists of vector-valued functions $T:\Theta^\infty(b) \to \Theta^\infty(b)$, $T(\bm{\theta}) = (T_1(\bm{\theta}), T_2(\bm{\theta}),\ldots)^\top \in \Theta^\infty(b)$ where $T_i \in \Theta^\infty(b)$ for all $i$ with bounded norm:
$\lVert T \rVert_{\mathcal{C}^k(\Theta^\infty(b); \Theta^\infty(b))} \coloneqq \sum_{i=1}^\infty \lVert T \rVert_{\mathcal{C}^k(\Theta^\infty(b);\R)} < \infty$. 
Then, we define a function space for OT maps with $\eta,\beta > 0$ is as follows:
\begin{align} \label{def:c1-smooth-map}
    \begin{split}
    \mathcal{T}_{\eta,\beta} = \bigr\{ T\in \mathcal{C}^1(\Theta^\infty(b); \Theta^\infty(b)) : T = &\nabla \varphi \text{ for some } \varphi \in \mathcal{C}^2(\Theta^\infty(b); \mathbb{R}), \\
    &\varphi \text{ is $\eta$-strongly convex and~}\lVert \nabla^2\varphi\rVert_{\tt{op}} \le \beta  \bigl\}.
    \end{split}
\end{align}

We derive a minimax lower bound with the poly-logarithmic rate with the $C^1$-class of OT maps.
\begin{theorem} \label{thm:lower_rate_log}
Let $T_0:\Theta^\infty(b)\to \Theta^\infty(b)$ be an OT map with corresponding probability measures $P,Q \in \mP(\Theta^\infty(b))$.
Then, we obtain the following bound as $n \to \infty$:
    \begin{equation}
        \label{eq:3new}
        \inf_{\overline{T} \in \mT_{\eta,\beta}} \sup_{P,Q \in \mP(\Theta^\infty(b))} \Ep \left[ \int \|\overline{T} (x) - T_0(x)\|^2 dP(x) \right] \gtrsim   \frac{1}{(\log n)^{2b+1}}, 
    \end{equation}
    where $\overline{T}$ is the minimizer taken from all estimators with $n$ observations $\{(X_i)\}_{i=1}^n \cup \{Y_i\}_{i=1}^n$.
\end{theorem}

This lower bound demonstrates that in the infinite-dimensional setting, the convergence rate becomes extremely slow when considering the H\"older smoothness. Given that the minimax  rate in the finite-dimensional case is affected by the curse of dimensionality \cite{hutter2021minimax,divol2022optimal}, it is reasonable that this logarithmic rate emerges when dimensions are increased to infinity.
An upper bound of the minimax error is discussed in Section \ref{sec:Ck-upper}.

\section{Conclusion} \label{sec:conclusion}
We revealed the minimax error rate of the OT map estimation when the data are infinite-dimensional. 
As a result, we find that the minimax estimation error of the OT map with $\gamma$-smoothness has a polynomial rate in $n$ and is independent of dimension. 
In the proofs, we have developed several proof techniques to handle infinite-dimensions.
This theory overcomes the difficulty of poly-logarithmic convergence rate under normal smoothness conditions in the infinite-dimensional setup.

\appendix

\section{Table of notations} \label{sec:notation_table}
\renewcommand{\baselinestretch}{1.0}
{\small
\begin{table}[!h]
\centering
\caption{Notations used in this paper.}
\begin{tabular}{ll}
\hline
\textbf{Notation} & \textbf{Description} \\
\hline
$d \in \mathbb{N} \cup \{\infty\}$ & The (possibly infinite) dimension \\
$\Omega$ & The bounded set in $[0,1]^\infty$ \\
$\mathcal{P}(\Omega)$ & The set of all probability measures on a set $\Omega$ \\
$\Omega_P, \Omega_Q$ & The support of probability measures $P$ and $Q$, $\Omega_P, \Omega_Q \subset \Omega$ \\
$\delta_x$ & The Dirac measure at $x$ \\
$L^p(P)$ & The space $\{f: \Omega \to \mathbb{R} : (\int |f|^p dP)^{1/p} < \infty\}$ for $P \in \mathcal{P}(\Omega)$ and $p \geq 1$ \\
$L^p(\Omega)$ & The $L^p$-space on $\Omega \subset \mathbb{R}^d$ with the Lebesgue measure $\lambda^d$ \\
$a^+$ & The positive part of $a \in \R$ i.e. $\max\{a,0\}$ \\
\hline
$T_0$     &  The optimal transport (OT) map \eqref{eq:Monge} \\
$\phi_0$    &   The Kantorovich potential \\
$\varphi_0$ &   The Brenier potential: $\varphi_0 = \|  \cdot \|^2/2 - \phi_0$     \\
$\varphi^*$ & the Legendre-Fenchel conjugate of $\varphi$: $\varphi^*(y) = \sup_{x \in \Omega_P} \langle x,y \rangle - \varphi(x)$ \\
$\phi_0^c$ & $\phi_0^c = \inf_{x \in \Omega_P} \|x-y\|^2/2 - \phi_0 (x)$\\
$S(\varphi)$ &  The semi-dual objective of $\varphi$: $S(\varphi) = \int \varphi(x) \ \d P(x) + \int \varphi^{*}(y) \ \d Q(y)$  \\
$S_n(\varphi)$ &  The empirical semi-dual objective of $\varphi$: $S_n(\varphi) = n^{-1}\sum_{i=1}^n \varphi(X_i) + n^{-1} \sum_{i=1}^n \varphi^{*}(Y_i)$  \\
\hline
$\mathscr{R}(\cdot)$ & The minimax risk of estimation \eqref{def:minimax_risk} \\
$s=(s_1,s_2,\ldots)$ & Dyadic scales (correspond to $\approx 2^{s_i}$ frequencies) along each coordinate \\
$a = (a_1,a_2,...)$ & Frequency weights: smoothness parameters along each coordinate \\
$\gamma(\cdot)$  & Smoothness map (Sec. \ref{sec:mixedaniso})                    \\ 
$\alpha(\gamma)$ & The inverse smoothness index of $\gamma$ (Sec. \ref{sec:mixedaniso}) \\
$ \psi_l$ & The $l$-th trigonometric basis function \eqref{eq:basisfunc} \\
$\delta_s$ & The projection of function on the $\approx 2^s$ frequencies (Sec. \ref{sec:mixedaniso}) \\
$H^\gamma([0,1]^\infty)$ & The space of $\gamma$-smooth functions (Def.~\ref{def:gammaspc}) \\
$H^\gamma([0,1]^\infty]; \ell^2)$ & The space of multi-output $\gamma$-smooth functions (Def.~\ref{def:gammaspc}) \\
$\gamma^{a,1}$ & The smoothness map for the mixed-smooth space \eqref{eq:gamma1} \\
$\gamma^{a,\infty}$ & The smoothness map for the anisotropic-smooth space \eqref{eq:gammainf} \\
$\mF_\varphi^{\gamma}$ & The function space of candidate Brenier potentials \eqref{eq:Fbrenier} \\
$\mF_T^{\gamma}$ & The function space of candidate transport maps \eqref{eq:FT} \\ \hline 
$I(S)$ & The $d$-dimensional index up to $2^S-1$: $I(S) = \{l \in \N^d_0, 0 \le l_i \le 2^S-1\}$ \\
$\mH_\varphi^\gamma$ & The function space of Brenier potentials in the proof of the lower bound (Thm. \ref{thm:lower_bound})\\ 
$\mH_T^\gamma$ & The function space of transport maps in the proof of the lower bound (Thm. \ref{thm:lower_bound})\\ \hline
$\mF_J$ & The function space of estimators for Kantorovich potentials \eqref{eq:defFJ} \\
$\tilde\mF_J$ & The function space of neural network estimators for Kantorovich potentials \eqref{eq:defFJnn} \\
$\mF^*_J$ & The function space of estimators for Brenier potentials \eqref{eq:defFstarJ} \\
$\tilde\mF^*_J$ & The function space of neural network estimators for Brenier potentials \eqref{eq:defFstarJnn} \\
$\hat{\varphi}_J$ & The estimator of the Brenier potential, obtained by solving \eqref{eq:samplesemidual} over $\mF^*_J$ \\
$\hat{\varphi}_{nn}$ & The neural estimator of the Brenier potential, obtained by solving \eqref{eq:samplesemidual} over $\tilde\mF^*_J$ \\
$\hat{\phi}_J, \hat{T}_J$ & The estimators of the Kantorovich potential and the transport map associated with $\hat{\varphi}_J$ \\
$\hat{\phi}_{nn}, \hat{T}_{nn}$ & The neural estimators of the Kantorovich potential and the transport map associated with $\hat{\varphi}_{nn}$ \\
$\bar{\phi}_J$ & The projection of $\phi_0$ on lower dyadic scales $s$ with $(1+2\alpha(\gamma))\gamma(s) \leq J$ \\
$\bar{\phi}_J$ & The best neural network approximation of $\bar\phi_J$ \\
$\hat{\phi}_t$ &  The ``localization'' of $\hat{\phi}_J$ near $\bar{\phi}_J$: $\hat\phi_t = t \hat{\phi}_J + (1-t) \bar{\phi}_J$ \\
$ \mE_n(f)$ & The empirical process of function $f$: $\mE_n(f) = \int f \d P - \tfrac{1}{n}\sum_{i=1}^n f(X_i)$ where $X_i \sim P$ \\
\hline
$\Theta^{\infty}(b)$ & The Sobolev ellipsoid for coefficients for input functions \eqref{def:sobolev_ellipsoid} \\
$\Theta^d(b)$ & The  $d$-dimensional truncated version of the Sobolev  ellipsoid  $\Theta^{\infty}(b)$\\
$\mT_{\eta,\beta}$ & The $C^1$-class of OT map \eqref{def:c1-smooth-map}\\
\hline
\end{tabular}
\end{table}
}

\section{Relation between the function spaces in Section \ref{sec:relspaces}} \label{sec:proof_relspaces}
\renewcommand{\baselinestretch}{1.2}
\textbf{Relation to the Gamma-space and anisotropic space}

Let $\lVert \cdot \rVert_{H^a}$ be the corresponding norm. One can show that the unit ball of $H^a([0,1]^\infty)$ is contained in the unit ball of $H^{a, \infty}([0,1]^{\infty})$: For any $f \in H^{a, \infty}([0,1]^{\infty})$, we have
\begin{align*}
    \lVert f \rVert^2_{H^{a, \infty}} &= \sum_{s \in \N^{\infty}_0} 2^{2\max_i \{a_is_i\}} \lVert \delta_s(f) \rVert^2 \\
    &= \sum_{s \in \N^{\infty}_0} 2^{2\max_i \{a_is_i\}} \sum_{l\in \Z^{\infty}_0-\bm{0}: \lfloor 2^{s_i-1} \rfloor \le \lvert l_i\rvert < 2^{s_i}}  \left\langle f, \psi_l \right\rangle^2 \\
    &= \sum_{s \in \N^{\infty}_0}  \sum_{l\in \Z^{\infty}_0-\bm{0}: \lfloor 2^{s_i-1} \rfloor \le \lvert l_i\rvert < 2^{s_i}}  2^{2\max_i \{a_i(s_i-1)+a_i\}}\left\langle f, \psi_l \right\rangle^2 \\
    &\leq \sum_{s \in \N^{\infty}_0} \sum_{l\in \Z^{\infty}_0-\bm{0}: \lfloor 2^{s_i-1} \rfloor \le \lvert l_i\rvert < 2^{s_i}} \max_i \left\{ (2 \lvert l_i \rvert)^{2 a_i} \right\} \left\langle f, \psi_l \right\rangle^2 \\
    &\leq \sum_{s \in \N^{\infty}_0}  \sum_{l\in \Z^{\infty}_0-\bm{0}: \lfloor 2^{s_i-1} \rfloor \le \lvert l_i\rvert < 2^{s_i}} \left\{\sum_{j=1}^{\infty} (2\pi \lvert l_j \rvert)^{2a_j}\right\} \left\langle f, \psi_l \right\rangle^2  = \lVert f \rVert^2_{H^a}.
\end{align*}
As a consequence, under the same conditions on $P,Q$ and $a$, the bounds for estimations of OT maps in $H^a([0,1]^{\infty})$ coincide with those in $H^{a, \infty}([0,1]^{\infty})$. Specifically, Corollary \ref{cor:upplow} yields a minimax rate of $\Tilde{\Theta}(n^{-2\tilde{a}/(2\tilde{a}+1)})$ where $\tilde{a} = (\sum_i a_i^{-1})^{-1}$.

\textbf{Relation between the anisotropic sobolev space and sobolev space}

The $d$-dimensional Sobolev space is a special case of anisotropic Sobolev spaces over functions in $L^2([0,1]^d)$ with $a_1=a_2=\ldots=a_d=k$ and $a_i = \infty$ for all $i > d$ (note that $\left\langle f, \psi_l \right\rangle = 0$ whenever $l_i \not = 0$ for some $i > d$ and we make a convention that $\infty \cdot 0 = 0$). For any $k \geq 1$, we have $\sum_{j=1}^d (2\pi \lvert l_j \rvert)^{2k} \leq (2\pi)^{2k}(1+\lVert l \rVert^2)^k$. Consequently, the unit ball of $H^k([0,1]^d)$ is contained in the $(2\pi)^k$-ball of $H^{a, \infty}([0,1]^{\infty})$ with $a = (k,\ldots,k,\infty,\infty,\ldots)$ where $k$ appears $d$ times, and so $\tilde{a} = (\sum_i a_i^{-1})^{-1} = k/d$.

\section{Fr\'echet derivatives on infinite-dimensional spaces} \label{sec:frechet}

In finite-dimensional spaces, existence and uniqueness of the optimal transport maps can be attained by assuming the smoothness and convexity of the map---we shall explore analogous concepts in infinite-dimensional spaces in this section.

Consider a Kantorovich potential $\phi: [0,1]^{\infty} \to \R$. Our notion of the derivative of $\phi$ at $x \in [0,1]^{\infty}$ is the \emph{Fr\'echet derivative}, defined as a bounded linear operator $\nabla \phi(x) : [0,1]^{\infty} \to \R$ such that
\begin{equation}\label{eq:frechetdef}
    \lim_{\lVert h\rVert \to 0} \frac{\lvert \phi(x+h) - \phi(x) - \nabla \phi (x) h\rvert}{\lVert h\rVert} = 0.
 \end{equation}
By writing $\phi$ as a Fourier series, the first and second order Fr\'echet derivatives of $\phi$ are quite intuitive:
\begin{lemma}[Fr\'echet derivatives of Kantorovich potentials] \label{lemma:frechetgrad}
Assume that $\alpha(\gamma) \le 1$. 
    Consider $\phi = \sum_{l\in \Z^{\infty}_0} \omega_l\psi_l \in H^{\gamma+2}$. Then, we obtain the followings:
\begin{enumerate}
   \item  The Fr\'echet derivative of $\phi$ at $x$ is a linear map $\nabla \phi(x): [0,1]^{\infty} \to \R$ given by
    \begin{equation}
        \label{eq:derphi}
           \nabla\phi(x)h = 2\pi \sum_{l \in \Z^{\infty}_0}(l^{\top} h)\omega_l \psi_l(x),
    \end{equation}
    whose operator norm is
        $\lVert \nabla \phi(x)\rVert_{\operatorname{op}} = ( 2\pi \sum_{l\in\Z^{\infty}_0} \lVert l\rVert^2 \omega^2_l \psi_l(x)^2 )^{1/2}$. The norm satisties the following bound:
    \begin{equation}
         \label{eq:gradphibdd}
             \lVert \nabla\phi\rVert_{\operatorname{op}}  \le  2\pi\lVert \phi\rVert_{H^{\gamma}}.
    \end{equation}
Consequently, the Fr\'echet derivative of $\varphi \coloneqq \frac1{2}\lVert \cdot\rVert^2 - \frac{1}{2}\phi$ is given by $\nabla \varphi = \operatorname{Id} - \frac1{2} \nabla \phi$.

    \item The Hessian of $\phi$ at $x$ is a linear map $\nabla^2 \phi(x):[0,1]^{\infty} \to B([0,1]^{\infty}, \R)$, where $B([0,1]^{\infty}, \R)$ is the space of bounded linear operators from $[0,1]^{\infty}$ to $\R$. Given $h\in [0,1]^{\infty}$, the map $\nabla^2\phi(x) h$ acts on $y\in [0,1]^{\infty}$ as follows:
    \begin{equation}
        \label{eq:hess}
              \bigl(\nabla^2\phi(x) h \bigr)y = 4\pi^2\sum_{l \in \N^{\infty}_0} \sum_{i,j=0}^{\infty} l_il_j h_iy_j \omega_l \psi_{l}(x),
          \end{equation}
          and the operator norm of $\nabla^2 \phi$ satisfies the following bound:
          \begin{equation}
              \label{eq:hessbdd}
                \lVert \nabla^2\phi\rVert_{\operatorname{op}}  \le  4\pi^2\lVert \phi\rVert_{H^{\gamma+2}}. 
          \end{equation}
      \end{enumerate}
\end{lemma}

We briefly note that, since each term is a continuous function and the sum converges uniformly, $\nabla \phi$ and $\nabla^2 \phi$ are measurable.
    To prove Lemma~\ref{lemma:frechetgrad}, we introduce two Sobolev-type norms that measure the smoothness of functions up to the second-order derivatives:
\begin{align}
    \lVert \phi(x)\rVert_{\ell^2\ell^1(H^1)} &= \Biggl( \sum_{i=1}^{\infty}\biggl( \sum_{s \in \N^{\infty}_0} \sum_{l \in \Z^{\infty}_0:\lfloor 2^{s-1} \rfloor \leq \lvert l_i \rvert <  2^s} \omega_l\partial_i\psi_l(x) \biggr)^2 \Biggr)^{1/2}, \label{eq:def211} \\
   \lVert \phi(x)\rVert_{\ell^2\ell^1(H^2)} &= \Biggl( \sum_{i,j=1}^{\infty}\biggl( \sum_{s \in \N^{\infty}_0} \sum_{l \in \Z^{\infty}_0:\lfloor 2^{s-1} \rfloor \leq \lvert l_i \rvert <  2^s} \omega_l\partial_i\partial_j\psi_l(x) \biggr)^2 \Biggr)^{1/2}. \label{eq:def212}
 \end{align}
The following lemma shows that these norms are controlled by that of $H^{\gamma+2}$:
\begin{lemma}\label{lemma:oneandtwo}
    Assume that $\alpha(\gamma) \le 1$. The following inequalities hold for any $\phi \in H^{\gamma+2}$:
    \begin{align}
               \sup_{x \in [0,1]^{\infty}}   \lVert \phi(x)\rVert_{\ell^2\ell^1(H^1)} &\le 2\pi\lVert \phi\rVert_{H^{\gamma+2}} ,\label{eq:211} \\
                \sup_{x \in [0,1]^{\infty}}  \lVert \phi(x)\rVert_{\ell^2\ell^1(H^2)} &\le 4\pi^2\lVert \phi\rVert_{H^{\gamma+2}}. \label{eq:212}
    \end{align}
    
\end{lemma}

\begin{proof}[Proof of Lemma \ref{lemma:oneandtwo}]
     We start with~\eqref{eq:211}. First, we observe that $\partial_i \psi_l = -2\pi l_i \psi_{\tau_i(l)}$ where $\tau_i(l) = (l_1,\ldots,l_{i-1},-l_i,l_{i+1},\ldots)$ which is a bijective map from $\{ l \in \Z^{\infty}_0:\lfloor 2^{s-1} \rfloor \leq \lvert l_i \rvert <  2^s \}$ to itself; so by the Cauchy-Schwarz inequality:
    \begin{align}
&\lVert \phi(x)\rVert^2_{\ell^2\ell^1(H^1)}    \\
&{}=  4\pi^2 \sum_{i=1}^{\infty}\biggl( \sum_{s \in \N^{\infty}_0} \sum_{l \in \Z^{\infty}_0:\lfloor 2^{s-1} \rfloor \leq \lvert l_i \rvert <  2^s} -l_i\omega_l\psi_{\tau_i(l)}(x) \biggr)^2 \\
&{}=  4\pi^2 \lVert \phi\rVert^2_{H^{\gamma+2}}  \sum_{s \in \N^{\infty}} \sum_{l \in \Z^{\infty}: 2^{s-1} \leq \lvert l_i \rvert <  2^s} \sum_{i=1}^{\infty} 2^{-2(\gamma(s)+2\alpha(\gamma))\gamma(s)}l^2_i\psi_{l}(x)^2. \label{eq:cauchyphi}
\end{align}
Here, the sum is over $s \in \N^{\infty}$ and $l \in \Z^{\infty}$ because the summand becomes zero whenever $l_i =0$. Since $l^2_i \le (2^{s_i}-1)^2$, we have $\sum_{i=1}^{\infty} l^2_i \le \sum_{i=1}^{\infty}(2^{s_i} - 1)^2 \le  2^{2\sum_i s_i}$. Thus, we continue the bound:
\begin{align}
    \eqref{eq:cauchyphi} \leq 4\pi^2 \lVert \phi\rVert^2_{H^{\gamma+2}}  \sum_{s \in \N^{\infty}_0}  2^{2\sum_i s_i-2(\gamma(s)+2\alpha(\gamma))\gamma(s)}\sum_{l \in \Z^{\infty}_0:\lfloor 2^{s-1} \rfloor \leq \lvert l_i \rvert <  2^s} \psi_{l}(x)^2.  \label{eq:midseries}
\end{align}
Recall that $\psi_l(x)= \prod_{i=1}^{\infty} \psi_{l_i}(x_i)$ where each $\psi_{l_i}(x_i)$ is either $\sqrt{2}\cos(2\pi \lvert l_i\rvert x_i)$ or $\sqrt{2}\sin(2\pi \lvert l_i\rvert x_i)$. Thus, the sum $\sum_{l \in \Z^{\infty}_0} \psi_l(x)^2$ above can be greatly simplified by repeatedly using $\cos^2(\cdot) + \sin^2(\cdot) = 1$. In addition, the constant factor in $\psi_l(x)$ is $2^{\frac{1}{2}\sum_i 1}$, which is finite since we assume that $s_i>0$ for only finitely many $i$. Consequently, 
\begin{align} \label{eq:basisbdd}
    \sum_{l \in \Z^{\infty}:2^{s-1}  \leq \lvert l_i \rvert <  2^s} \psi_l(x)^2 \le \bigl\lvert \{ l \in \N^{\infty}: 2^{s-1}  \leq  l_i  <  2^s \} \bigr \rvert \cdot 2^{\sum_i 1} &= 2^{\sum_i (s_i-1) + 1} =  2^{\sum_i s_i}.
\end{align}
As we know $\alpha(\gamma) = \sup_{s \in \N^{\infty}_0} \sum_i s_i / \gamma(s)$, we have $\sum_i s_i \le \alpha(\gamma)\gamma(s)$. Thus, by taking the supremum of~\eqref{eq:midseries} over $x \in [0,1]^{\infty}$, we obtain 
 \begin{align}
   \sup_{x \in [0,1]^{\infty}}  \lVert \phi(x)\rVert^2_{\ell^2\ell^1(H^1)}       
   &\le 4\pi^2 \lVert \phi\rVert^2_{H^{\gamma+2}} \sum_{s \in \N^{\infty}} 2^{3 \sum_i s_i - 2(\gamma(s)+2\alpha(\gamma))\gamma(s) } \label{eq:l2l1_bdd} \\
    &\le 4\pi^2 \lVert \phi\rVert^2_{H^{\gamma+2}} \sum_{s \in \N^{\infty}} 2^{-\sum_i s_i + (4\sum_i s_i - 4\alpha(\gamma)\gamma(s))} \\
    &\le 4\pi^2 \lVert \phi\rVert^2_{H^{\gamma+2}} \sum_{s \in \N^{\infty}} 2^{-\sum_i s_i}  \\
    &= 4\pi^2 \lVert \phi\rVert^2_{H^{\gamma+2}}  \prod_{i=1}^{\infty} \sum_{k=1}^{\infty} 2^{-k} \\
    &= 4\pi^2 \lVert \phi\rVert^2_{H^{\gamma+2}}.
\end{align}
    
    The bound~\eqref{eq:212} can be proved similarly by observing that $\partial^2_i \psi_l = -4\pi^2 l^2_i \psi_{\tau_{ii}(l)}$ and  $\partial_i\partial_j \psi_l = 4\pi^2 l_i l_j \psi_{\tau_{ij}(l)}$ whenever $i\not= j$, where $\tau_{ii}(l)=l$ and $\tau_{ij}(l) = (\ldots,-l_i\ldots,-l_j,\ldots)$. Again, $\tau_{ij}$ is bijective on $\{ l \in \Z^{\infty}_0:\lfloor 2^{s-1} \rfloor \leq \lvert l_i \rvert <  2^s \}$, and therefore,
    \begin{align}
& \lVert \phi(x)\rVert^2_{\ell^2\ell^1(H^2)}  \\
&{} = 16\pi^4 \sum_{i,j=1}^{\infty}\biggl( \sum_{s \in \N^{\infty}_0} \sum_{l \in \Z^{\infty}_0:\lfloor 2^{s-1} \rfloor \leq \lvert l_i \rvert <  2^s} l_il_j\omega_l\psi_{\tau_{ij}(l)}(x) \biggr)^2 \\
&{} \le  16\pi^4  \biggl( \sum_{s \in \N^{\infty}_0} \sum_{l \in \Z^{\infty}_0:\lfloor 2^{s-1} \rfloor \leq \lvert l_i \rvert <  2^s} 2^{2(1+2\alpha(\gamma))\gamma(s)} \omega^2_l \biggr) \\ &\qquad \times \biggl( \sum_{i,j=1}^{\infty} \sum_{s \in \N^{\infty}_0} \sum_{l \in \Z^{\infty}_0:\lfloor 2^{s-1} \rfloor \leq \lvert l_i \rvert <  2^s} 2^{-2(1+2\alpha(\gamma))\gamma(s)} l^2_il^2_j \psi_{\tau_{ij}(l)}(x)^2 \biggr) \\
&{}=  16\pi^4  \lVert \phi\rVert^2_{H^{\gamma+2}} \sum_{s \in \N^{\infty}} 2^{-2(1+2\alpha(\gamma))\gamma(s)} \sum_{l \in \Z^{\infty}: 2^{s-1} \leq \lvert l_i \rvert <  2^s}  \lVert l \rVert^4 \psi_{l}(x)^2 .
\end{align}
For any $l$ in the summand, we have $\lVert l\rVert^4 \le 2^{4\sum_i s_i}$, $\sum_i s_i \le \gamma(s)$ and \eqref{eq:basisbdd}. In addition, the condition $\alpha(\gamma) = \sup_{s \in \N^{\infty}_0} \sum_i s_i / \gamma(s) \le 1$ yields $\sum_i s_i \le \gamma(s)$. As a result, we obtain      
 \begin{align}
 \sup_{x \in [0,1]^{\infty}}    \lVert \phi(x)\rVert^2_{\ell^2\ell^1(H^2)}        &\le  16\pi^4  \lVert \phi\rVert^2_{H^{\gamma+2}} \sum_{s \in \N^{\infty}} 2^{4\sum_i s_i  + \sum_i s_i -2(1+2\alpha(\gamma))\gamma(s)} \\
 &= 16\pi^4  \lVert \phi\rVert^2_{H^{\gamma+2}} \sum_{s \in \N^{\infty}} 2^{-\sum_i s_i + (6\sum_i s_i  - 2\gamma(s) -4\alpha(\gamma)\gamma(s))} \\
        &\le 16\pi^4  \lVert \phi\rVert^2_{H^{\gamma+2}} \sum_{s \in \N^{\infty}} 2^{-\sum_i s_i} \\
        &\le 16\pi^4 \lVert \phi\rVert^2_{H^{\gamma+2}}.
    \end{align}
\end{proof}

\begin{proof}[Proof of Lemma \ref{lemma:frechetgrad}]
    We just need to check that $\nabla \phi$ defined above satisfies the limit in~\eqref{eq:frechetdef}. 
    Let $\epsilon > 0$. As $\lVert \phi\rVert_{H^{\gamma+2}} \le 1$, and each $s \in \N^{\infty}_0$ has finitely many nonzero components $s_i,i\in \N$, there exists $d,K \in \N$ large enough such that 
\begin{align} 
\sum_{s \in \N^{\infty}_0 - \N^d_0}&2^{2(1+2\alpha(\gamma))\gamma(s)} \sum_{l \in \Z^{\infty}_0: \lfloor 2^{s_i-1}\rfloor \le l_i < 2^{s_i}} \omega^2_l < \frac{\epsilon^2}{144\pi^2} \label{eq:highbdd} \\
    \sum_{s \in \N^{\infty}_0: \sum_i s_i > K}&2^{2(1+2\alpha(\gamma))\gamma(s)} \sum_{l \in \Z^{\infty}_0: \lfloor 2^{s_i-1}\rfloor \le l_i < 2^{s_i}} \omega^2_l < \frac{\epsilon^2}{144\pi^2}. \label{eq:highKbdd}
\end{align}
With this $d$ and $K$, we split $\phi$ into three terms:
\begin{equation}\label{eq:dimfreqsplit}
    \begin{split}
    \phi_{d-,K-} &= \sum_{s \in \N^d_0: \sum_i s_i \le K}\sum_{l \in \Z^{\infty}_0: \lfloor 2^{s_i-1}\rfloor \le \lvert l_i\rvert < 2^{s_i}} \omega_l \psi_l, \\
    \phi_{d-,K+} &= \sum_{s \in \N^{d}_0: \sum_i s_i > K}\sum_{l \in \Z^{\infty}_0: \lfloor 2^{s_i-1}\rfloor \le \lvert l_i\rvert < 2^{s_i}} \omega_l \psi_l, \\
        \phi_{d+} &= \sum_{s \in \N^{\infty}_0 - \N^d_0}\sum_{l \in \Z^{\infty}_0: \lfloor 2^{s_i-1}\rfloor \le \lvert l_i\rvert < 2^{s_i}} \omega_l \psi_l,
\end{split}
    \end{equation}
and $\nabla \phi_{d-,K-},\nabla \phi_{d-,K+}$ and $\nabla \phi_{d+}$ be the restrictions of the series of $\nabla \phi$ in~\eqref{eq:derphi} to the corresponding ranges of $s$.

We first focus on $\phi_{d+}$. We denote for each $s \in \N^{\infty}_0, d'>d$ and $K' \in \N$,
    \begin{align}
        \phi_s &\coloneqq \sum_{l \in \Z^{\infty}_0: \lfloor 2^{s_i-1}  \rfloor \le \lvert l_i\rvert < 2^{s_i}} \omega_l \psi_l, \\
        \phi_{d \to d',K'} &\coloneqq \sum_{\substack{s\in\N^{d'}_0 - \N^d_0 \\ \sum_i s_i \le K'}}\sum_{l \in \Z^{\infty}_0: \lfloor 2^{s_i-1}  \rfloor \le \lvert l_i\rvert < 2^{s_i}} \omega_l \psi_l. \label{eq:phidK}
    \end{align}
Since there are only finitely many $s \in \N^{d'}_0$ such that $\sum_i s_i \le K'$, and each $s$ has only finitely many nonzero components, so $\phi_{d\to d',K'}$ can be viewed as a function on $\R^D$ for some sufficiently large $D \in \N$. Thus, $\nabla \phi_{d\to d',K'}$, the restriction of the sum in $\nabla \phi$~\eqref{eq:derphi} to this range of $s$, agrees with the usual gradient on $\R^D$; this allows us to apply the mean value theorem on $\phi_{d\to d',K'}$.
\begin{align}
  \lvert \phi_{d\to d',K'}(x+h) - \phi_{d\to d',K'}(x)\rvert &\le \sup_{x \in [0,1]^D} \lVert \nabla\phi_{d\to d',K'}(x) \rVert \lVert h\rVert. 
\end{align}
Consequently,
\begin{equation}\label{eq:psidiff}
    \frac{\lvert  \phi_{d\to d',K'}(x+h) - \phi_{d\to d',K'}(x) - \nabla \phi_{d\to d',K'}(x) h\rvert}{\lVert h\rVert} \le 
   2\sup_{x \in [0,1]^D} \lVert \nabla\phi_{d\to d',K'}(x) \rVert.
\end{equation}
Notice that $\lVert \nabla\phi_{d\to d',K'}(x) \rVert$ is exactly  $\lVert \phi_{d\to d',K'}(x)\rVert_{\ell^2\ell^1(H^1)}$ in~\eqref{eq:def211}, so it follows from~\eqref{eq:211} and~\eqref{eq:highbdd} that
\begin{align}
    \frac{\lvert \phi_{d\to d',K'}(x+h) - \phi_{d\to d',K'}(x) - \nabla \phi_{d\to d',K'}(x) h \rvert}{\lVert h\rVert} &\le 2 \sup_{x \in [0,1]^D} \lVert \phi_{d\to d',K'}(x)\rVert_{\ell^2\ell^1(H^1)} \\
    &\le 4\pi\lVert \phi_{d\to d',K'}\rVert_{H^{\gamma}} \\
    &\le \frac{\epsilon}{3}.
\end{align}
As the bound is uniform over all $d' \ge d$ and $K' \in \N$, it follows that
\begin{align}
    \label{eq:eps1}
    \frac{\lvert \phi_{d+}(x+h) - \phi_{d+}(x) - \nabla \phi_{d+}(x) h \rvert}{\lVert h\rVert} \le \frac{\epsilon}{3}.
\end{align}
Applying the same idea with~\eqref{eq:highKbdd} instead of~\eqref{eq:highbdd}, we obtain a bound for $\phi_{d-,K+}$:
\begin{align}
    \label{eq:eps2}
    \frac{\lvert \phi_{d-,K+}(x+h) - \phi_{d-,K+}(x) - \nabla \phi_{d-,K+}(x) h \rvert}{\lVert h\rVert} \le \frac{\epsilon}{3}.
\end{align}

Now we focus on $\phi_{d-,K-}$. Let $(h_n)_{n=1}^{\infty}$ be a sequence in $[0,1]^{\infty}$ such that $x + h_n \in [0,1]^{\infty}$ for all $n$ and $\lVert h_n\rVert \to 0$. Notice that there are only finitely many $s \in \N^d_0$ such that  $\sum_i s_i \le K$, each of which has finitely many nonzero components. Therefore, $\phi_{d-,K+}$ can be cast as a function on $\R^{D'}$ for a sufficiently large $D'$. In addition, $\nabla \phi_{d-,K+}$, the restriction of the sum in $\nabla \phi$ to this range of $s$, coincides with the usual gradient of $\phi_{d-,K+}$; So there exists $N\in \N$ such that
\begin{equation} \label{eq:eps3}
    \frac{\lvert \phi_{\le K}(x+h_n) - \phi_{\le K}(x) - \nabla \phi_{\le K} (x) h_n\rvert}{\lVert h_n\rVert} < \frac{\epsilon}{3} \quad \text{for all } n \ge N. 
\end{equation}

Combining~\eqref{eq:eps1},~\eqref{eq:eps2} and~\eqref{eq:eps3}, we obtain
\[ \frac{\lvert \phi(x+h_n) - \phi(x) - \nabla \phi (x) h_n\rvert}{\lVert h_n\rVert} < \epsilon, \]
for all $n \ge N$. Since the sequence $(h_n)_{n=1}^{\infty}$ is arbitrary, we conclude that the quotient converges to $0$ as $\lVert h\rVert \to 0$.

The operator norm of $\nabla \phi(x)$ can be obtained by taking $h$ in~\eqref{eq:derphi} to be the unit vector $h= \sum_{l \in \Z^{\infty}_0} l \omega_l \psi_l(x) /  \left(\sum_{l \in \Z^{\infty}_0}\lVert l\rVert^2_l \omega^2_l \psi_l(x)^2\right)^{1/2}$.

To prove the gradient bound, it follows from~\eqref{eq:211} for each $x \in [0,1]^\infty$ that,
\begin{align}
    \lVert \nabla\phi(x)\rVert_{\text{op}} \le \limsup_{d \to \infty, K\to \infty} \lVert \nabla \phi_{d-,K-}(x)\rVert_{\text{op}} &\le  \limsup_{d \to \infty, K\to \infty} \lVert \nabla \phi_{d-,K-}(x)\rVert_F \\
    &= \limsup_{d \to \infty, K\to \infty} \lVert \phi_{d-,K-}(x)\rVert_{\ell^2\ell^1(H^1)} \\
    &\le 2\pi\limsup_{d \to \infty, K\to \infty} \lVert \phi_{d-,K-}\rVert_{H^{\gamma}} \\
    &\le  2\pi\lVert \phi\rVert_{H^{\gamma}}. 
\end{align}

The results for the Hessian can be proved in the same manner. Given $\epsilon >0$, there exists $d,K \in \N$ large enough so that 
\begin{align} 
\sum_{s \in \N^{\infty}_0 - \N^d_0}&2^{2(1+2\alpha(\gamma))\gamma(s)} \sum_{l \in \Z^{\infty}_0: \lfloor 2^{s_i-1}\rfloor \le l_i < 2^{s_i}} \omega^2_l < \frac{\epsilon^2}{576\pi^4},\label{eq:highbdd2} \\
    \sum_{s \in \N^{\infty}_0: \sum_i s_i > K}&2^{2(1+2\alpha(\gamma))\gamma(s)} \sum_{l \in \Z^{\infty}_0: \lfloor 2^{s_i-1}\rfloor \le l_i < 2^{s_i}} \omega^2_l < \frac{\epsilon^2}{576\pi^4}. \label{eq:highKbdd2}
\end{align}
Recall the notations $\phi_{d-,K-},\phi_{d-,K+},\phi_{d+}$ and their gradients from~\eqref{eq:dimfreqsplit}. We also let $\nabla^2\phi_{d-,K-},\nabla^2\phi_{d-,K+},\nabla^2\phi_{d+}$ be the restrictions of the series of $\nabla^2\phi$ in~\eqref{eq:hess} to the corresponding ranges of $s$.

We start with $\phi_{d+}$. Let $d' >d$, $K' \in \N$ and $\phi_{d \to d',K'}$ be defined as in~\eqref{eq:phidK}. As argued earlier, $\phi_{d \to d',K'}$ can be viewed as a function on $\R^D$, and $\nabla^2 \phi_{d \to d',K'}$ coincides with the usual Hessian of $\phi_{d \to d',K'}$. Thus we can apply the second-order mean value theorem on $\nabla \phi_{d \to d',K'}$:
\begin{align}
  \lVert \nabla \phi_{d \to d',K'}(x+h) - \nabla \phi_{d \to d',K'}(x)\rVert &\le \sup_{x \in [0,1]^D} \lVert \nabla^2\phi_{d \to d',K'}(x) h \rVert. 
\end{align}
Consequently, we obtain
\begin{align}
    \frac{\lVert \nabla \phi_{d \to d',K'}(x+h) - \nabla \phi_{d \to d',K'}(x) - \nabla^2 \phi_{d \to d',K'}(x) h \rVert}{\lVert h\rVert} &\le 2  \sup_{x \in [0,1]^D} \lVert \nabla^2\phi_{d \to d',K'}(x) \rVert_{\text{op}} \\
    &\le 2 \sup_{x \in [0,1]^D} \lVert \nabla^2\phi_{d \to d',K'}(x) \rVert_F,
\end{align}
where $\lVert \cdot  \rVert_{\text{op}}$ is the matrix operator norm and $\lVert \cdot\rVert_F$ is the matrix Frobenius norm. Notice that $\lVert \nabla^2\phi_{d \to d',K'}(x) \rVert_F$ is exactly $\lVert \nabla^2\phi_{d \to d',K'}(x) \rVert_{\ell^2\ell^1(H^2)}$ in~\eqref{eq:def212}. Thus, it follows from~\eqref{eq:212} and~\eqref{eq:highbdd2} that
\begin{align}
    &\frac{\lVert \nabla \phi_{d \to d',K'}(x+h) - \nabla \phi_{d \to d',K'}(x) - \nabla^2 \phi_{d \to d',K'}(x) h \rVert}{\lVert h\rVert} \\
    &\le 2  \sup_{x \in [0,1]^D} \lVert \nabla^2\phi_{d \to d',K'}(x) \rVert_{\ell^2\ell^1(H^2)} \\
    &\le 8\pi^2  \sup_{x \in [0,1]^D} \lVert \nabla^2\phi_{d \to d',K'}(x) \rVert_{H^{\gamma+2}} \\
    &\le \frac{\epsilon}{3}.
\end{align}
Since the bound is uniform over all $d' \ge d$ and $K' \in \N$, we have
\begin{align}
    \label{eq:eps4}
    \frac{\lVert \nabla \phi_{d+}(x+h) - \nabla \phi_{d+}(x) - \nabla^2 \phi_{d+}(x) h \rVert}{\lVert h\rVert} \le \frac{\epsilon}{3}.
\end{align}
Analogously, we can show using~\eqref{eq:highKbdd2} instead of~\eqref{eq:highbdd2} that
\begin{align}
    \label{eq:eps5}
    \frac{\lvert \nabla \phi_{d-,K+}(x+h) - \nabla \phi_{d-,K+}(x) - \nabla^2 \phi_{d-,K+}(x) h \rvert}{\lVert h\rVert} \le \frac{\epsilon}{3}.
\end{align}
Let $(h_n)_{n=1}^{\infty}$ be a sequence in $[0,1]^{\infty}$ such that $x + h_n \in [0,1]^{\infty}$ for all $n$ and $\lVert h_n\rVert \to 0$. Since $\phi_{d-,K-}$ can be viewed as a function on $\R^{D'}$ for some $D'\in\N$, whose Hessian coincides with $\nabla^2\phi_{d-,K-}$, there exists a sufficiently large $N \in \N$ such that
\begin{equation} \label{eq:eps6} \frac{\lvert \nabla\phi_{d-,K-}(x+h_n) - \nabla \phi_{d-,K-}(x) - \nabla^{2} \phi_{d-,K-} (x) h_n\rvert}{\lVert h_n\rVert} < \frac{\epsilon}{3} \quad \text{for all } n \ge N.
\end{equation}

Combining~\eqref{eq:eps4},~\eqref{eq:eps5} and~\eqref{eq:eps6}, we have that
\[ \frac{\lVert \nabla \phi(x+h_n) - \nabla \phi(x) - \nabla^2 \phi (x) h_n\rVert}{\lVert h_n\rVert} < \epsilon, \]
for all $n \ge N$. Since the sequence $(h_n)_{n=1}^{\infty}$ is arbitrary, we conclude that the quotient converges to $0$ as $\lVert h\rVert \to 0$.

To prove the last assertion, it follows from~\eqref{eq:212} for each $x \in [0,1]^\infty$ that,
\begin{align}
    \lVert \nabla^2\phi(x)\rVert_{\text{op}} \le \limsup_{d \to \infty, K\to \infty} \lVert \nabla^2 \phi_{d-,K-}(x)\rVert_{\text{op}} &\le  \limsup_{d \to \infty, K\to \infty} \lVert \nabla^2 \phi_{d-,K-}(x)\rVert_F \\
    &= \limsup_{d \to \infty, K\to \infty} \lVert \phi_{d-,K-}(x)\rVert_{\ell^2\ell^1(H^2)} \\
    &\le 4\pi^2\limsup_{d \to \infty, K\to \infty} \lVert \phi_{d-,K-}\rVert_{H^{\gamma+2}} \\
    &\le  4\pi^2\lVert \phi\rVert_{H^{\gamma+2}}. 
\end{align}
\end{proof}

\section{Smoothness and strong convexity of \texorpdfstring{$\gamma$}{gamma}-smooth functions}
Let $\varphi$ be a function in the space $F^{\gamma}_{\varphi}$ as defined in \eqref{eq:Fbrenier}. Specifically, $\varphi$ is $\eta$-strongly convex. We show that $\varphi$ is $(1+4\pi^2)$-smooth, and for sufficiently large $J_0 \in \N$, the truncation of the corresponding Kantorovich potential of $\varphi$ to the first $J_0$ dyadic scales is $(1+4\pi^2)$-smooth and $\eta/2$-strongly convex.

            \begin{lemma}\label{lemma:smoothconvex}
                Given any $\varphi =  \lVert \cdot \rVert^2/2 - \phi \in F^{\gamma}_{\varphi}$, let $\bar{\varphi}_J =  \lVert \cdot \rVert^2/2 -\phi_J$ where $\phi_J$ is a truncation of $\phi$ to the first $J$ dyadic scales, that is:
                \begin{align*}
                    \phi &= \sum_{s \in \N^{\infty}_0} \sum_{\substack{l \in \Z^{\infty}_0 \\ \lfloor 2^{s_i-1}\rfloor \le l_i < 2^{s_i}}} \omega_l \psi_l, \\
                    \phi_J &= \sum_{\substack{s \in \N^{\infty}_0 \\ (1+2\alpha(\gamma))\gamma(s) < J}} \sum_{\substack{l \in \Z^{\infty}_0 \\ \lfloor 2^{s_i-1}\rfloor \le l_i < 2^{s_i}}} \omega_l \psi_l.
                \end{align*}
                The following statements hold:
                \begin{enumerate}
                    \item Both $\varphi$ and $\bar{\varphi}_J$ are $(1+4\pi^2)$-smooth functions. \label{lemma:sc1}
                    \item There exists a constant $J_0 = J_0(\phi_0,\eta)$ such that, for any $J \ge J_0$, $\bar{\varphi}_J$ is an $\eta/2$-strongly convex function. \label{lemma:sc2}
                \end{enumerate}
            \end{lemma}

\begin{proof}
Let $\varphi \in \mF$, so $\varphi =  \lVert \cdot \rVert^2/2 - \phi$ for some $\phi \in H^{\gamma+2}$ with $\lVert \phi\rVert_{H^{\gamma+2}} \le 1$. We recall from~\eqref{eq:hessbdd} that the operator norm of $\nabla^2\phi$ can be bounded by the $H^{\gamma+2}$-norm of $\phi$:     
 \begin{align}
     \lVert \nabla^2 \phi \rVert_{\text{op}}   \le4\pi^2 \lVert \phi\rVert_{H^{\gamma+2}} \le 4\pi^2,
    \end{align}
    which implies that $\lVert \nabla^2 \varphi\rVert_{\text{op}} \le 1+4\pi^2$, and so $\varphi$ is $(1+4\pi^2)$-smooth.

    To show that $\bar{\varphi}_J = \lVert \cdot\rVert^2/2 - \phi_J$ is $\eta/2$-strongly convex, we bound $\lVert \nabla^2\bar{\varphi}_J - \nabla^2\varphi_0\rVert_{\text{op}}$ in the same manner: 
    \begin{align}
        \lVert \nabla^2 \bar{\varphi}_J - \nabla^2 \varphi_{0} \rVert^2_{\text{op}} &= \lVert \nabla^2 \bar{\phi}_J - \nabla^2 \phi_{0} \rVert^2_{\text{op}} \\
        &\le 16\pi^4\lVert  \bar{\phi}_J -  \phi_{0} \rVert^2_{H^{\gamma+2}} \\
            &= \sum_{s \in \N^{\infty}_0: (1+2\alpha(\gamma))\gamma(s) \ge J} 2^{2(1+2\alpha(\gamma))\gamma(s)} \sum_{l \in \Z^{\infty}_0: \lfloor 2^{s_i-1}\rfloor \le l_i < 2^{s_i}} (\omega^0_l)^2.
        \end{align}
        Since $\lVert \phi_0\rVert_{H^{\gamma+2}} \le 1$, there exists $J_0 = J_0(\phi_0,\alpha(\gamma))$ such that the right-hand side is smaller than $\eta^2/4$ for any $J \ge J_0$. Denote by $\Lambda_{\min}(\cdot)$ the smallest eigenvalue of the matrix. By the Weyl's inequality,
        \[
            \Lambda_{\min}(\nabla^2 \varphi_{0} ) - \Lambda_{\min}(\nabla^2 \bar{\varphi}_J) \leq \Lambda_{\min}(\nabla^2 \varphi_{0}  - \nabla^2 \bar{\varphi}_J) \leq \lVert \nabla^2 \varphi_{0}  - \nabla^2 \bar{\varphi}_J \rVert_{\text{op}} \leq \frac{\eta}{2} . 
        \]
Since $\varphi_0$ is $\eta$-strongly convex, $\Lambda_{\min}(\nabla^2 \varphi_{0} ) \geq \eta$, and so $\Lambda_{\min}(\nabla^2 \bar{\varphi}_J) \geq \eta/2$; this implies that $\bar{\varphi}_J$ is $\eta/2$-strongly convex for any $J \ge J_0$.
\end{proof}

\section{Proof of the lower Bound (Theorem~\ref{thm:lower_bound})} \label{sec:fullprooflower}

Given $S \in \N$, we consider a family of functions $g_l$ indexed by the set
\[ I(S) \coloneqq \{l \in \N^d_0, 0 \le l_i \le 2^S-1\}.\]
Note that $\psi_l(x)$ is a product of cosines for any $l\in I(S)$. Denote $M \coloneqq \lvert I(S)\rvert = 2^{dS}$. With an abuse of notations, we denote $\gamma(\bm{S})$ where $\bm{S} = (S,S,...,S) \in \N^d$ by $\gamma(S)$. To construct well-separated hypotheses, we define a collection of functions indexed by elements in $I(S)$:
\[
g_0 \equiv 0, \qquad g_l(x) = (2\sqrt{2}\pi)^{-2}2^{-\gamma(S)}M^{-1/2}  \lVert l \rVert^{-1}  \psi_l(x), \qquad l\in I(S) - \{0 \}. 
\]
Here, the coefficients of $\psi_l(x)$'s are chosen so that the gradient of $g \coloneqq \sum_{l\in I(S)} g_l$ is contained in the unit ball of $H^{\gamma}([0,1]^d)$, as the following computation shows:
\begin{align}
    & (2\sqrt{2}\pi)^{-2}2^{-2\gamma(S)}M^{-1}\sum_{s\in \N^d_0: 0 \le s_i \le S} 2^{2\gamma(s)} \sum_{j=1}^d \lVert \delta_s(\partial_jg(x))\rVert^{2}_2 \\
    &\le 2^{-2\gamma(S)}M^{-1}\sum_{s\in \N^d_0: 0 \le s_i \le S} 2^{2\gamma(s)}\sum_{l\in \Z^{\infty}_0: \lfloor 2^{s_i-1} \rfloor \le \lvert  l_i\rvert < 2^{s_i}}   \frac{1}{\lVert l \rVert^2}   \sum_{j=1}^d l^2_j \\
    &\le 2^{-2\gamma(S)}M^{-1}\sum_{s\in \N^d_0: 0 \le s_i \le S} 2^{2\gamma(s)}\sum_{l\in \Z^{\infty}_0: \lfloor 2^{s_i-1} \rfloor \le \lvert  l_i\rvert < 2^{s_i}}   1 \\
    &\le M^{-1}\sum_{s\in \N^d_0: 0 \le s_i \le S} \sum_{l\in \Z^{\infty}_0: \lfloor 2^{s_i-1} \rfloor \le \lvert  l_i\rvert < 2^{s_i}}   1  \\
    &\le M^{-1} \left( 1 + 2 + \ldots + 2^{S-1} \right)^d\\
    &\le 1.
\end{align}
Their gradients and Hessian's satisfy: 
\begin{equation}\label{eq:gradhess}
    \begin{split}
        \int_{[0,1]^d} \lVert \nabla g_l(x) \rVert^2 \d x  &\asymp     2^{-2\gamma(S)}M^{-1} \lVert l\rVert^{-2} \sum_{j=1}^d l^2_j\phantom{ l^2_k \  } \asymp 2^{-2\gamma(S)}M^{-1}, \\
        \int_{[0,1]^d} \lVert  \nabla^2g_l(x)  \rVert^2_F \d x &\asymp 2^{-2\gamma(S)}M^{-1} \lVert l\rVert^{-2}   \sum_{j,k=1}^d l^2_jl^2_k   \asymp2^{-2\gamma(S)}M^{-1} \lVert l\rVert^2 .
    \end{split}
\end{equation}

Since $d \ge 4$, we have $M = \lvert I(S)\rvert = 2^{dS } \ge 8$. The Varshamov-Gilbert lemma yields binary vectors $\tau^{(0)},\ldots,\tau^{(K)}\in \{0,1\}^M$ with $K \ge 2^{M/8}$ such that $\lVert \tau^{(m)} - \tau^{(m')}\rVert^2 \ge M/8$ for all $0 \le k \not= k' \le K$. We then define Brenier potentials $\phi_k:[0,1]^d \to \R$ by
\[
\phi_m(x) = \frac{1}{2} \lVert x\rVert^2 + \sum_{l\in I(S)} \tau^{(m)}_l g_l(x), \qquad m=0,\ldots,K.
\]
From~\eqref{eq:gradhess}, we have
\begin{equation}\label{eq:separation2}
    \begin{split}
    \int_{[0,1]^d} \lVert \nabla \phi_m(x) - \nabla \phi_{m'}(x)\rVert^2 \d x &\gtrsim \sum_{l\in I(S)} \lvert \tau^{(m)}_l - \tau^{(m')}_l \rvert^2 \int_{[0,1]^d} \lVert \nabla g_l(x) \rVert^2 \d x \\
    &\gtrsim 2^{-2\gamma(S)}M^{-1}M \\
    &= 2^{-2\gamma(S)}.
    \end{split}
\end{equation}

 We now define probability distributions $P_0 = \operatorname{Uniform([0,1]^d)}$, and $Q_m = (\nabla \varphi_m)_{\#}P_0$ for $m=1,...,K$. By the definition, we have $Q_m \ll P_0$, and the Radon-Nikodym derivative $f_m = \d Q_m/\d P_0$ is given by
\[
f_m(y) = \frac{1}{\det \nabla^2\phi_m((\nabla \phi_m)^{-1} (y))}\mathbf{1}\left[ (\nabla \phi_m)^{-1}(y) \in [0,1]^d \right]. 
\]
This allows us to calculate the KL divergence between $Q_m$ and $P_0$:
\begin{align}
    D(Q_m\Vert P_0)  =  \Ep_{Y\sim Q_m} \left[ \frac{\d Q_m}{\d P_0}(Y) \right] &=  \Ep_{X\sim P_0} \left[ \frac{\d Q_m}{\d P_0}(\nabla \phi_m (X)) \right] \\
    &= -\int_{[0,1]^d} \log (\det \nabla^2\phi_m(x)) dx .
\end{align}
Let $\Lambda_k(A)$ be the $k$-th eigenvalue of a matrix $A$ and denote $\tilde{\phi}_m \coloneqq \phi_m - \lVert \cdot\rVert^2/2 = \sum_{l\in I(S)} \tau^{(m)}_l g_l(x)$. Using $\log (1+z) \ge z - z^2/2$ for all $z>0$ and $\int_{[0,1]^d} \psi_l(x) \d x = 0$ for all $l \in \Z^d_0$,
\begin{align}
    - \int_{[0,1]^d} \log (\det \nabla^2\phi_m(x)) dx &= -\sum_{k=1}^d\log ( 1 + \Lambda_k( \nabla^2\tilde{\phi}_m)) \d x \\
    &\le -\int_{[0,1]^d} \textsf{Tr}( \nabla^2\tilde{\phi}_m ) \d x + \frac{1}{2} \int_{[0,1]^d} \sum_{k=1}^d \Lambda_k( \nabla^2\tilde{\phi}_m)^2  \d x \\
    &= \frac{1}{2} \int_{[0,1]^d}\sum_{k=1}^d \Lambda_k( \nabla^2\tilde{\phi}_m)^2  \d x \\
    &\le \frac{1}{2} \int_{[0,1]^d} \lVert  \nabla^2\tilde{\phi}_m(x)  \rVert^2_F   \d x \\
    &\le \frac{1}{2}\sum_{l\in I(S)} \int_{[0,1]^d}\lVert  \nabla^2g_l(x)  \rVert^2_F  \d x.
\end{align}
The summation can be bounded using~\eqref{eq:gradhess}, leading to
\begin{align}
  - \int_{[0,1]^d} \log (\det \nabla^2\phi_m(x)) dx   &\lesssim 2^{-2\gamma(S)}M^{-1}\sum_{l\in I(S)} \lVert l\rVert^{2}_2\\
    &= 2^{-2\gamma(S)}M^{-1}\sum_{i=1}^d \Big(\sum_{j=1}^{2^S-1} j^2\Big) \cdot \Big(\sum_{j=0}^{2^S-1} 1 \Big)^{d-1} \\
    &= 2^{-2\gamma(S)}dM^{-1}\cdot \frac{1}{6} \left((2^S-1)\cdot 2^S \cdot (2^{S+1}-1) \cdot 2^{(d-1)S} \right) \\
    &\lesssim d2^{-2\gamma(S)} 2^{2S} \\
    &\le 2^{-2\gamma(S)} 2^{(2+ \log_2 d)S}.
\end{align}
Now we focus on probability distributions over $2n$ variables: $\{P_0^{\otimes n} \otimes Q_m^{\otimes n}\}_{m=1}^K$ and $P_0^{\otimes n} \otimes P_0^{\otimes n}$. Let $\theta>0$ be a sufficiently large constant so that $M/\theta \le (\log K)/9$. Choosing $S$ such that
\[ n \asymp 2^{2dS/\alpha(\gamma)}2^{(d-2- \log_2 d)S}/\theta, \quad \text{ which implies }\quad  2^{S} \asymp n^{\frac{\alpha(\gamma)}{2d +\alpha(\gamma)(d-2- \log_2 d)}}.\]
The condition $\alpha(\gamma) =\sup_{s \in \N^{\infty}_0} \sum_i s_i / \gamma(s)$ implies $\sum_i S/\alpha(\gamma) = dS/\alpha(\gamma) \le \gamma(S)$. Consequently, we obtain
\[
D\left(P_0^{\otimes n} \otimes Q_m^{\otimes n} \Vert P_0^{\otimes n}\otimes P_0^{\otimes n}\right) = nD(Q_m \Vert P_0) \le \frac{2^{2dS/\alpha(\gamma)-2\gamma(S)}2^{dS}}{\theta} \lesssim \frac{M}{\theta} \le \frac{\log K}{9}.
\]
Going back to the lower bound~\eqref{eq:separation2}, we obtain the minimax lower bound via the Fano inequality:
\begin{align}
    \inf_{\hat{T}} \sup_{P\in \mM, T_0 \in \mF^{\gamma}} \Ep \left[ \int \| \hat{T} (x) - T_0(x)\|^2 \d P(x) \right]  &\gtrsim 2^{-2\gamma(S)} \\ 
    &\geq 2^{-2dS/\alpha(\gamma)} \\
    &\asymp n^{-1}2^{(d-2- \log_2 d)S} \\
    &\asymp n^{-1}n^{\frac{\alpha(\gamma)(d-2- \log_2 d)}{2d + \alpha(\gamma)(d-2- \log_2 d)}} \\
    &= n^{-\frac{2}{2 + \alpha(\gamma) -\alpha(\gamma)(2 +\log_2 d)/d}}.
\end{align}
We obtain the desired lower bound by taking a sufficiently large $d$.
\qed

\section{Proof of the upper bound (Theorem \ref{thm:upper_bound})} \label{sec:proofe_upper}

The proof for the upper bound of \cite[Theorem 8]{hutter2021minimax} requires the estimated Brenier potential to be convex so that, which in turn implies the truncation of $\varphi_0$ to the first $J$ wavelets, say $\tilde{\varphi}_J$, is also convex; then we can apply van de Geer’s localization technique, which involves minimizing $\lVert\nabla \varphi - \nabla \varphi_0 \rVert$ over the convex series of first $J$ wavelets $\varphi$, which can now be bounded by $\lVert \nabla \tilde{\varphi}_J - \nabla \varphi_0 \rVert$. Our proof, on the other hand, uses \cite{divol2022optimal}'s approach which consists of defining the ``interpolation space'' between $\hat{f}$ and $\bar{f}$ that controls the $C^2$-norm and calculating the covering number of this space instead. An advantage of this technique is that it does not require the estimator to be convex.

We start with the localization technique presented in Section \ref{sec:upper}, then first state the stronger bound for the localized term.
The value of $\tau$ will be determined towards the end of the proof in order to obtain the final bound.
For brevity, we introduce a notation $\hat{S}$ as
\begin{align}
    \hat{S}(\varphi) := \frac1{n} \sum_{i=1}^n \varphi(X_i) + \frac1{n} \sum_{i=1}^n \varphi^{*}(Y_i),
\end{align}
so we have $\hat{\varphi}_J := \argmin_{\varphi \in \mF^*_J} \hat{S}(\varphi)$.
Denote the associated Brenier potentials $\hat\varphi_t = \lVert \cdot \rVert^2/2 - \hat\phi_t$, $\bar\varphi_J = \lVert \cdot \rVert^2/2 - \bar\phi_J$ and $\varphi_0 = \lVert \cdot \rVert^2/2 - \phi_0$. For any $\varphi \in L^2(P)$, we denote: 
\[ S_0(\varphi) \coloneqq S(\varphi) - S(\varphi_0), \qquad \hat{S}_0(\varphi) \coloneqq \hat{S}(\varphi) - \hat{S}(\varphi_0). \] 
Then, we obtain the following bound.
\begin{proposition}[Error decomposition] \label{prop:localization}
\begin{equation} 
  \lVert \nabla \hat{\phi}_t - \nabla \bar{\phi}_J \rVert^2_{L^2(P)}  \lesssim  [S(\hat{\varphi}_t) -\hat{S}(\hat{\varphi}_t)]  + [\hat{S}(\bar{\varphi}_J) - S(\bar{\varphi}_J)] + 2S_0(\bar{\varphi}_J). \label{eq:localizedbdd}
\end{equation} 
\end{proposition}

To control $S_0(\bar{\varphi}_J)$ on the right-hand side of \eqref{eq:localizedbdd}, we use the following lemmas.
The first one is for the approximation error.
\begin{lemma}[Approximation error] \label{lemma:approx_error}
    \begin{equation}\label{eq:approx}
        \lVert \nabla \bar{\phi}_J - \nabla \phi_{0} \rVert^2_{L^2(P)} \lesssim 2^{-\frac{2}{1+2\alpha(\gamma)}J}.
        \end{equation}
    \end{lemma}
Second, we develop the stability bound as \cite{hutter2021minimax,manole2021plugin,divol2022optimal}:
    \begin{proposition}[Map stability] \label{prop:stability}
        If $\varphi_0,\varphi_1$ are $\beta$-smooth functions and $P,Q$ are probability distributions such that $Q = (\nabla \varphi_0)_{\#}P$, then
        \begin{equation}
            \label{eq:stab_lower}
            \frac{1}{2\beta}\lVert \nabla \varphi_1 - \nabla \varphi_0 \rVert^2_{L^2(P)} \le S (\varphi_1) - S(\varphi_0).
        \end{equation}
        In addition, if $\varphi_1$ is $\eta$-strongly convex, then we have
        \begin{equation}
            \label{eq:stab_upper}
              S_0(\varphi_1) = S(\varphi_1) - S(\varphi_0) \le \frac1{2\eta}\lVert \nabla \varphi_1 - \nabla \varphi_0 \rVert^2_{L^2(P)}.
        \end{equation}
    \end{proposition}
    With Lemma \ref{lemma:approx_error}, we can now control $S_0(\bar{\varphi}_J)$ in Proposition \ref{prop:localization} via the stability bound~\eqref{eq:stab_upper} and the absolute continuity of $P$:
    \begin{equation}\label{eq:S0bdd}
        S_0(\bar{\varphi}_J) \lesssim \lVert \nabla \bar{\phi}_J - \nabla \phi_{0} \rVert^2_{L^2(P)} \lesssim \lVert \nabla \bar{\phi}_J - \nabla \phi_{0} \rVert^2_{L^2([0,1]^{\infty})} \lesssim  2^{-\frac{2}{1+2\alpha(\gamma)}J}.
    \end{equation}
    To bound the first two terms in Proposition \ref{prop:localization}, we introduce two function spaces:
    \begin{equation}\label{eq:mF}
        \begin{split}
            \mF_J(\tau) &\coloneqq \{ \phi_t \coloneqq t\phi + (1-t) \bar{\phi}_J : \phi \in \mF_J, \lVert \nabla \phi_t - \nabla \bar{\phi}_J \rVert_{L^2(P)}
            \leq \tau, \lVert \nabla^2\phi_t - \nabla^2 \bar{\phi}_J\rVert_{\operatorname{op}} \le \eta/4 \},  \\
             \mG_J(\tau) &\coloneqq \{(\phi - \bar{\phi}_J) \oplus (\phi^c - \bar{\phi}^c_J) : \phi \in \mF_J(\tau)\}.
        \end{split}
    \end{equation}
    Here, the condition $\lVert \nabla^2\phi_t - \nabla^2 \bar{\phi}_J\rVert \le \eta/4$ is needed for the Bousquet-Talagrand inequality \eqref{eq:Znconc} and the mixed entropy bound (Lemma~\ref{lem:mixedbound}). The condition $\lVert \nabla^2\phi_t - \nabla^2 \bar{\phi}_J\rVert_{\operatorname{op}} \le \eta/4$, combined with Lemma~\ref{lemma:smoothconvex} that $\bar{\varphi}_J = \lVert \cdot \rVert^2/2 - \bar{\phi}_J$ is $\eta/2$-strongly convex, ensures that $\varphi_t \coloneqq \lVert \cdot \rVert^2/2 -\phi_t$ is $\eta/4$-strongly convex.
                
    We also define probabilities $\Pr$ and $\operatorname{Pr}_n$ that act on any $f \oplus g \in \mG_J(\tau)$ as follows:
    \begin{equation}\label{eq:Prdef}
        \Pr(f \oplus g) = \frac{1}{2}(P(f) + Q(g)), \quad \operatorname{Pr}_n(f \oplus g) = \frac{1}{2}(P_n(f) + Q_n(g)),  
    \end{equation}
        with which we define an empirical process over $\mG_J(\tau)$:
        \begin{equation}\label{eq:Zn}
            Z_n(\tau) \coloneqq \sup_{g \in \mG_J(\tau)} \lvert (\Pr - \operatorname{Pr}_n)(g) \rvert.
        \end{equation}
 
        To bound $\lVert \nabla \hat{\phi}_t - \nabla \phi_{0} \rVert^2_{L^2(P)}$ in terms of $Z_n(\tau)$ we use Lemma~\ref{lem:bound_error} below. We then bound $Z_n(\tau)$ using a tail bound inequality in Lemma \ref{lem:bound_empirical_process} below. Then, we bound $\Ep Z_n(\tau)$ using Lemma \ref{lem:exp_empirical_process}. The proofs of these three lemmas are postponed to Section~\ref{sec:alllemmas}.

\begin{lemma}\label{lem:bound_error}
The following inequality holds:
     \begin{equation}\label{eq:estbdd}
       \lVert \nabla \hat{\phi}_t - \nabla \phi_{0} \rVert^2_{L^2(P)}  \leq   C_3(Z_n(\tau) + 2^{-\frac{2}{1+2\alpha(\gamma)}J}) + c_3\tau^2,
       \end{equation} 
       for some constants $C_3 > 0$ and $c_3 \in (0,1)$.
\end{lemma}

\begin{lemma}\label{lem:bound_empirical_process}
        Denote $\tilde\tau \coloneqq \tau + 2^{-\frac{2}{1+2\alpha(\gamma)}J}$. With probability greater than $1-e^{-u}$, we have
\begin{equation} \label{eq:concentrate}
    Z_n(\tau) \lesssim 2\Ep Z_n(\tau) + \tilde\tau\sqrt{\frac{2u}{n}} +\frac{u}{n}.
\end{equation}
\end{lemma}
    
\begin{lemma} \label{lem:exp_empirical_process} Denote $\tilde\tau \coloneqq \tau + 2^{-\frac{2}{1+2\alpha(\gamma)}J}$. The following inequality holds:  
    \begin{align} \label{eq:EZn_bound}
 \Ep Z_n(\tau) &\lesssim  \frac{\tilde\tau 2^{\frac{\alpha(\gamma)}{2(1+2\alpha(\gamma))} J }\sqrt{J \log (1+C_4/\tilde\tau)}}{\sqrt{n}} + \frac{J 2^{\frac{\alpha(\gamma)}{1+2\alpha(\gamma)} J} }{n}. 
\end{align}
for some constant $C_4>0$.
\end{lemma}

The bounds~\eqref{eq:concentrate} and~\eqref{eq:EZn_bound} yield some constants $C_5,C_6>0$ such that the following estimate holds with probability at least $1-e^{-u}$:
\begin{align}
    &C_3( Z_n(\tau) + 2^{-\frac{2}{1+2\alpha(\gamma)}J}) \label{eq:C5}\\
    &\le C_5 \biggl(\Ep Z_n(\tau) + \tilde\tau\sqrt{\frac{u}{n}} +\frac{u}{n}\biggr) + C_3 2^{-\frac{2}{1+2\alpha(\gamma)}J}   \\
    &\le C_6 \biggl(\frac{\tilde\tau 2^{\frac{\alpha(\gamma)}{2(1+2\alpha(\gamma))} J}\sqrt{J \log (1+C_4/\tilde\tau)}}{\sqrt{n}} + \frac{J 2^{\frac{\alpha(\gamma)}{1+2\alpha(\gamma)} J} }{n} \biggr)  + C_5 \biggl( \tilde\tau\sqrt{ \frac{u}{n}} +\frac{u}{n}\biggr)  + C_3 2^{-\frac{2}{1+2\alpha(\gamma)}J}  \\
      &=: A_1.
\end{align}
Recall the definitions $\hat{\phi}_t = t \hat{\phi}_J + (1-t)\bar{\phi}_J$ with $t=\tau / (\tau + \lVert \nabla \hat{\phi}_J - \nabla \bar{\phi}_J \rVert_{L^2(P)})$ and the constant $c_3$ in \eqref{eq:estbdd}. Suppose that the inequality $\lVert \nabla \hat{\phi}_J - \nabla \bar{\phi}_{J}\rVert^2_{L^2(P)} > c \tau^2$ holds for some constant $c >0$. Then, it implies that 
\begin{align*} 
\lVert \nabla \hat{\phi}_t - \nabla \bar{\phi}_{J}\rVert^2_{L^2(P)} = t^2\lVert \nabla \hat{\phi}_J - \nabla \bar{\phi}_{J}\rVert^2_{L^2(P)}  > \frac{c}{(c^{1/2}+1)^2}\tau^2. 
\end{align*} 
As $c_3 \in (0,1)$, we can choose $c$ sufficiently large so that $c (c^{1/2}+1)^{-2} > c_3$. By bounding the leftmost term in the above display using~\eqref{eq:estbdd}, we obtain the following bounds with $\tilde{c} \coloneqq c (c^{1/2}+1)^{-2} - c_3$:
\[ C_3(Z_n(\tau) + 2^{-\frac{2}{1+2\alpha(\gamma)}J}) + c_3\tau^2 > c(c^{1/2}+1)^{-2}\tau^2 \quad \Longrightarrow \quad C_3(Z_n(\tau) + 2^{-\frac{2}{1+2\alpha(\gamma)}J}) > \tilde{c}\tau^2. \] 
To ensure that $\tilde{c} \tau^2  \ge  A_1$, we choose
\begin{equation} \label{eq:taudef}\tau  = \frac{1}{\tilde{c}^{1/2}} \left\{ (C_6 + 1) \frac{2^{\frac{\alpha(\gamma)}{2(1+2\alpha(\gamma))}J}\sqrt{J \log (1+C_4n)}}{\sqrt{n}} + (C_5 + 1)\sqrt{\frac{u}{n}} +  (C_3 + 1)2^{-\frac{1}{1+2\alpha(\gamma)}J} \right\}, \end{equation}
with which the following chain of inequalities holds:
\begin{align}
   & \Pr \Bigl( \lVert \nabla \hat{\phi}_J - \nabla \phi_0\rVert^2_{L^2(P)}  > c\tau^2 + \lVert \nabla \bar{\phi}_J - \nabla \phi_0\rVert^2_{L^2(P)} \Bigr) \\
    &\le \Pr \left( \lVert \nabla \hat{\phi}_J - \nabla \bar{\phi}_J\rVert^2_{L^2(P)}  > c\tau^2\right) \\
    &\le \Pr \left( C_3(Z_n(\tau) + 2^{-\frac{2}{1+2\alpha(\gamma)}J})  > \tilde{c}\tau^2 \right) \\
    &\le \Pr \left( C_3(Z_n(\tau) + 2^{-\frac{2}{1+2\alpha(\gamma)}J}) > A_1\right) \\
    &\le e^{-u}.
\end{align}
In other words, with probability at least $1-e^{-u}$,
\begin{equation*}
     \lVert  \nabla \hat{\phi}_J - \nabla \phi_0\rVert^2_{L^2(P)}  \lesssim \tau^2+ \lVert \nabla \bar{\phi}_J - \nabla \phi_0\rVert^2_{L^2(P)} \lesssim \underbrace{\frac{J2^{\frac{\alpha(\gamma)}{1+2\alpha(\gamma)}J} \log n}{n} + 2^{-\frac{2}{1+2\alpha(\gamma)}J}}_{A_2} + \frac{u}{n} ,
\end{equation*}
where we again use~\eqref{eq:approx} to bound the approximation error $\lVert \nabla \bar{\phi}_J - \nabla \phi_0\rVert^2_{L^2(P)}$.

We can now compute the expected estimation error by integrating the tail probability, and applying a change of variables:
\begin{align}
    &\Ep \lVert  \nabla \hat{\phi}_J - \nabla \phi_0\rVert^2_{L^2(P)} \\
    &= \int_0^{A_2} \Pr( \lVert  \nabla \hat{\phi}_J - \nabla \phi_0\rVert^2_{L^2(P)} > r) \ \d r + \int_{A_2}^{\infty} \Pr( \lVert  \nabla \hat{\phi}_J - \nabla \phi_0\rVert^2_{L^2(P)} > r) \ \d r\\
    &\le A_2 + \int_0^{\infty} \frac{1}{n}\Pr( \lVert  \nabla \hat{\phi}_J - \nabla \phi_0\rVert^2_{L^2(P)} > A_2+u/n) \, \d u\\
    &\le A_2 + \int_0^{\infty} \frac{1}{n}\Pr( \lVert  \nabla \hat{\phi}_J - \nabla \phi_0\rVert^2_{L^2(P)} \gtrsim \tau^2 + \lVert \nabla \bar{\phi}_J - \nabla \phi_0\rVert^2_{L^2(P)} ) \, \d u \\
    &\le A_2 + \int_0^{\infty} \frac{1}{n}e^{-u} \ \d u \\
    &= A_2 + \frac{1}{n} \\
    &= \frac{J 2^{\frac{\alpha(\gamma)}{1+2\alpha(\gamma)} J} \log n}{n} + 2^{-\frac{2}{1+2\alpha(\gamma)}J} + \frac{1}{n}. \label{eq:almostdone}
\end{align}
Optimizing for $J$ yields $J \approx \frac{1+2\alpha(\gamma)}{2+\alpha(\gamma)}\log_2 n$, leading us to the desired bound for mixed-smooth potentials:
\[
\Ep \lVert  \nabla \hat{\varphi}_J - \nabla \varphi_0\rVert^2_{L^2(P)} = \Ep \lVert  \nabla \hat{\phi}_J - \nabla \phi_0\rVert^2_{L^2(P)} \lesssim n^{-\frac{2}{2 + \alpha(\gamma)}}(\log n)^2.
\]
\qed

\section{Proofs of lemmas for the upper bound} \label{sec:alllemmas}

    \begin{proof}[Proof of Proposition \ref{prop:stability}]
        We start with~\eqref{eq:stab_lower}. Since $Q = (\nabla \varphi_0)_{\#}P$, we can write it as
        \begin{equation}\label{eq:sphi}
            S(\varphi_1) = P \varphi_1 + Q \varphi_1^{*} = \int_{[0,1]^{\infty}} \varphi_1(x) + \varphi_1^{*}(\nabla \varphi_0(x)) \ \d P(x).
        \end{equation}
        We use the following characterization of $\beta$-smooth functions, (see e.g.~\cite[Theorem 18.15]{Bauschke2011}): for any $x,y\in [0,1]^{\infty}$,
        \begin{equation} \label{eq:smoothineq}
            \varphi_1(y) \le \varphi_1(x) + \langle y-x, \nabla \varphi_1(x) \rangle + \frac{\beta}{2} \lVert x-y \rVert^2.
        \end{equation}
        We use this to find a lower bound of $\varphi^{*}_0(\nabla \varphi_1(x))$ as follows:
        \begin{align}
            \varphi_1^{*}(\nabla \varphi_0(x)) &= \sup_{y \in [0,1]^{\infty}} \left\{\langle y, \nabla \varphi_0(x) \rangle - \varphi_1(y) \right\} \\
            &\ge \sup_{y \in [0,1]^{\infty}} \left\{ \langle y, \nabla \varphi_0(x) \rangle - \langle y-x, \nabla \varphi_1(x) \rangle - \frac{\beta}{2} \lVert x-y \rVert^2\right\} - \varphi_1(x). \\
            &= \sup_{y \in [0,1]^{\infty}} \left\{ \langle y, \nabla \varphi_0(x) - \nabla \varphi_1(x) \rangle - \frac{\beta}{2} \lVert x-y \rVert^2\right\} +\langle x, \varphi_1(x) \rangle - \varphi_1(x). \\
        \end{align}
        The supremum is attained at $y= x - \frac{1}{\beta}(\nabla \varphi_0(x) - \nabla\varphi_1(x))$, leading to
        \begin{align}
            &\varphi_1^{*}(\nabla \varphi_0(x))\\
            &\ge \langle x, \nabla \varphi_0(x) - \nabla \varphi_1(x) \rangle + \frac{1}{2\beta} \lVert \nabla \varphi_0(x) - \nabla \varphi_1(x) \rVert^2 +\langle x, \nabla \varphi_1(x) \rangle - \varphi_1(x) \\
            &= \frac{1}{2\beta} \lVert \nabla \varphi_0(x) - \nabla \varphi_1(x) \rVert^2 +\langle x, \nabla \varphi_0(x) \rangle - \varphi_1(x).
        \end{align}
     Having this bound, we continue from~\eqref{eq:sphi} and obtain
        \begin{align}
            S(\varphi_1) &= \int_{[0,1]^{\infty}} \varphi_1(x) + \varphi_1^{*}(\nabla \varphi_0(x)) \ \d P(x) \\
            &\ge \frac{1}{2\beta} \lVert \nabla \varphi_0(x) - \nabla \varphi_1(x) \rVert^2_{L^2(P)} +\int_{[0,1]^{\infty}} \langle x, \nabla \varphi_0(x) \rangle \ \d P(x) \\
            &= \frac{1}{2\beta} \lVert \nabla \varphi_0(x) - \nabla \varphi_1(x) \rVert^2_{L^2(P)} + S (\varphi_0),
        \end{align}
        where the last line follows from the fact that $(\operatorname{id},\nabla \varphi_0)_{\#}P$ is the optimal coupling for the optimal transport problem. The proof of~\eqref{eq:stab_upper} follows in a similar manner since the $\eta$-strong convexity of $\varphi$ is equivalent to the following reverse inequality of~\eqref{eq:smoothineq}:
        \begin{equation*} 
            \varphi_1(y) \ge \varphi_1(x) + \langle y-x, \nabla \varphi_1(x) \rangle + \frac{\eta}{2} \lVert x-y \rVert^2.
        \end{equation*}
    \end{proof}

\begin{proof}[Proof of Proposition \ref{prop:localization}]
    
        Since $\hat{\phi}_t,\bar{\phi}_J \in H^{\gamma+2}([0,1]^{\infty})$, they are $\beta$-smooth for some $\beta>0$; this allows us to apply the stability bound:
        \begin{align}
            \lVert \nabla \hat{\phi}_t - \nabla \bar{\phi}_J \rVert^2_{L^2(P)} 
            &\lesssim \lVert  \nabla \hat{\phi}_t - \nabla \phi_0\rVert_{L^{2}(P)}^2 + \lVert  \nabla\bar{\phi}_{J} - \nabla\phi_0\rVert_{L^{2}(P)}^2 \\
            &\lesssim  S_0(\hat{\phi}_t) + S_0(\bar{\phi}_J)   \\
            &=  S_0(\hat{\phi}_t) - S_0(\bar{\phi}_J) + 2S_0(\bar{\phi}_J)   \\
            &=  S(\hat{\phi}_t) - S(\bar{\phi}_J) + 2S_0(\bar{\phi}_J). \label{eq:SSS} 
        \end{align}
        By the convexity of $\hat{S}$, it holds that
        \begin{align}
            \hat{S} (\hat{\phi}_t)  \le t \hat{S} (\hat{\phi}_J) + (1-t) \hat{S}  (\bar{\phi}_J) \le \hat{S} (\bar{\phi}_J).
        \end{align}
        This allows us to upper bound~\eqref{eq:SSS} by empirical processes:
        \begin{equation} 
          \lVert \nabla \hat{\phi}_t - \nabla \phi_{0} \rVert^2_{L^2(P)}  \lesssim  [S(\hat{\phi}_t) -\hat{S}(\hat{\phi}_t)]  + [\hat{S}(\bar{\phi}_J) - S(\bar{\phi}_J)] + 2S_0(\bar{\phi}_J).   \label{eq:phiZn}
        \end{equation}
\end{proof}

    \begin{proof}[Proof of Lemma \ref{lemma:approx_error}]
    By writing $\nabla \bar{\phi}_J(x) - \nabla \phi_{0}(x)$ as Fourier series, it follows from the proof for the Frech\'et gradient bound \eqref{eq:l2l1_bdd} that
    \begin{align}
        \int _{[0,1]^{\infty}}\lVert \nabla \bar{\phi}_J(x) - &\nabla \phi_{0}(x) \rVert^2 \d P(x) \\
        &{}\le \lVert \nabla \bar{\phi}_J - \nabla \phi_{0} \rVert^2_{L^\infty} \\
        &\asymp{} \lVert \nabla \bar{\phi}_J - \nabla \phi_{0} \rVert^2_{\ell^2\ell^1(H^1)}  \\     
        &\lesssim{}  \lVert \phi_0\rVert^2_{H^{\gamma+2}} \sum_{s\in\N^{\infty}_0: (1+2\alpha(\gamma))\gamma(s)  \ge J} 2^{3 \sum_i s_i - 2(\gamma(s)+2\alpha(\gamma))\gamma(s) } \\
        &\le{}  \lVert \phi_0\rVert^2_{H^{\gamma+2}} \sup_{s\in\N^{\infty}_0: (1+2\alpha(\gamma))\gamma(s)  \ge J} 2^{-2\gamma(s)}\sum_{s\in\N^{\infty}_0} 2^{-\sum_i s_i + (4\sum_i s_i - 4\alpha(\gamma)\gamma(s))} \\
        &\le{} \lVert \phi_0\rVert^2_{H^{\gamma+2}} 2^{-\frac{2}{1+2\alpha(\gamma)}J}\sum_{s\in\N^{\infty}_0} 2^{-\sum_i s_i} \\
        &\le{} 2^{-\frac{2}{1+2\alpha(\gamma)}J}\lVert \phi_0\rVert^2_{H^{\gamma+2}}. \label{eq:tail_bdd}
    \end{align}
    \end{proof}

\begin{proof}[Proof of Lemma \ref{lem:bound_error}]
           We will bound $\lVert \nabla\hat{\phi}_t - \nabla \bar{\phi}_J\rVert_{L^{2}(P)}$ in two cases of $t$; if necessary, we decrease the value of $\eta$ so that $\eta < 32\pi^2$:
        \begin{enumerate}
        \item $t \leq \eta/32\pi^2$. With $t \coloneqq \tau / (\tau + \lVert \nabla \hat{\phi}_J - \nabla \bar{\phi}_J \rVert_{L^2(P)})$, we have
        \begin{align*}
            \lVert \nabla \hat{\phi}_t - \nabla \bar{\phi}_J \rVert_{L^2(P)} 
            = t\lVert \nabla \hat{\phi}_J - \nabla \bar{\phi}_J \rVert_{L^2(P)}  = \frac{\tau \lVert  \nabla\hat{\phi}_J - \nabla\bar{\phi}_{J}\rVert_{L^{2}(P)}}{\tau + \lVert \nabla \hat{\phi}_J - \nabla \bar{\phi}_J \rVert_{L^2(P)}}  \leq \tau.
        \end{align*}
        And it follows from Lemma~\ref{lemma:smoothconvex} that the matrix operator norms of both $\nabla^2\hat{\phi}_t$ and $\nabla^2 \bar{\phi}_J$ are bounded by $4\pi^2$. Consequently,
        \begin{equation*}
            \lVert \nabla^2\hat{\phi}_t - \nabla^2 \bar{\phi}_J\rVert_{\text{op}} = t \lVert \nabla^2 \hat{\phi}_J - \nabla^2 \bar{\phi}_{J}\rVert_{\text{op}}  \le t \left( \lVert \nabla^2 \hat{\phi}_J  \rVert_{\text{op}} + \lVert  \nabla^2 \bar{\phi}_{J}\rVert_{\text{op}}  \right) \le 8\pi^2 t \le \frac{\eta}{4}.
        \end{equation*}
         The above two displays imply $\hat{\phi}_t \in \mF_J(\tau)$.
         Then, the bounds \eqref{eq:localizedbdd} and \eqref{eq:S0bdd} imply
        \begin{equation}\label{eq:smallt}
            \lVert \nabla \hat{\phi}_t - \nabla \phi_{0} \rVert^2_{L^2(P)} \lesssim 2Z_n(\tau) + 2^{-\frac{2}{1+2\alpha(\gamma)}J}.
        \end{equation}
        \item $t > \eta/32\pi^2$. Then,
        \begin{equation} \label{eq:larget} \lVert \nabla\hat{\phi}_t - \nabla \bar{\phi}_J\rVert_{L^{2}(P)} = \frac{\tau\lVert  \nabla\hat{\phi}_J - \nabla\bar{\phi}_{J}\rVert_{L^{2}(P)} }{\tau + \lVert  \nabla\hat{\phi}_J - \nabla\bar{\phi}_{J}\rVert_{L^{2}(P)}} = \tau(1-t) < \left(1-\frac{\eta}{32\pi^2}\right)\tau.
        \end{equation}
        \end{enumerate}
       Combining \eqref{eq:smallt} and \eqref{eq:larget}, we obtain \eqref{eq:estbdd}.
\end{proof}

\begin{proof}[Proof of Lemma \ref{lem:bound_empirical_process}]

        To bound $Z_n(\tau)$ with high probability, we use a corollary of the Bousquet-Talagrand inequality (\cite[Lemma 2.15.10]{van1996weak}): for any $u>0$,
        \begin{equation}\label{eq:Znconc}
           \Pr\Big(Z_n(\tau) \ge 2\Ep Z_n(\tau) + \sigma\sqrt{\frac{2u}{n}} +2M\frac{u}{n}\Big) \le e^{-u},
        \end{equation}
        where
        \begin{align}
           \sigma =\sup_{g \in \mG_J(\tau)} \lVert g \rVert_{L^2(\Pr)}, \quad M = \sup_{g \in \mG_J(\tau)} \lVert g\rVert_{L^{\infty}}.
          \end{align}
          We first find an upper bound for $\sigma$. For any $\phi \in \mF_J(\tau)$, by the definition of $\mF_J(\tau)$ and Poincar\'e inequality, we have
          \begin{align} 
           \lVert \phi - \bar{\phi}_J \rVert_{L^{2}(P)}  \le \lVert \nabla \phi - \nabla \bar{\phi}_J \rVert_{L^{2}(P)} \le \tau. \label{eq:l2bdd}
         \end{align}
    To bound $\lVert \phi^c- \bar{\phi}^c_J \rVert^2_{L^2(Q)}$, 
    we first apply the triangle inequality.
        \[
            \lVert \phi^c - \bar{\phi}^c_J \rVert^2_{L^2(Q)} \le \lVert \phi^c - \phi^c_0 \rVert^2_{L^2(Q)} + \lVert  \bar{\phi}^c_J - \phi^c_0 \rVert^2_{L^2(Q)}.
        \]
        Since the Brenier potentials $\varphi \coloneqq \frac{1}{2}(\lVert \cdot \rVert^2 - \phi)$ and $\varphi_0 \coloneqq \frac{1}{2}(\lVert \cdot \rVert^2 -   \phi_0)$ are strongly convex, the maps $T\coloneqq \nabla \varphi$ and $T_0 \coloneqq \nabla \varphi_0$ are invertible. From $\varphi^{*} \coloneqq \frac{1}{2}(\lVert \cdot \rVert^2 -   \phi^c)$ and $\varphi^{*}_0 \coloneqq \frac{1}{2}(\lVert \cdot \rVert^2 -   \phi^c_0)$, it follows that
        \begin{align}
            &\int   \lvert \phi^c(y) - \phi^c_0(y) \rvert^2 \ \d Q(y) \\
            &= 2\int   \lvert \varphi^{*}(y) - \varphi^{*}_0(y) \rvert^2 \ \d Q(y) \\
            &= 2\int \bigl\lvert \langle y, T^{-1}(y)  \rangle - \varphi \circ T^{-1}(y)  - \Bigl( \langle y, T^{-1}_0(y)  \rangle - \varphi_0 \circ T^{-1}_0(y) \Bigr) \bigr\rvert^2 \ \d Q(y) \\
            &= 2\int \bigl\lvert \langle T_0(x), T^{-1}\circ T_0(x)  \rangle - \varphi \circ T^{-1}\circ T_0(x) - \Bigl( \langle T_0(x), x  \rangle - \varphi_0 (x) \Bigr) \bigr\rvert^2 \ \d P(x) \\
            &\le 6\int \bigl\lvert \langle T_0(x), T^{-1}\circ T_0(x)  \rangle - \langle T_0(x), x  \rangle \bigr\rvert^2  + \bigl\lvert \varphi(x) - \varphi \circ T^{-1}\circ T_0(x) \bigr\rvert^2  + \bigl\lvert  \varphi_0 (x)  - \varphi(x) \bigr\rvert^2 \ \d P(x) \\
            &\le 6 \bigl( \lVert T_0\rVert^2_{L^2(P)} \lVert T^{-1}\circ T_0 - \operatorname{Id}  \rVert^2_{L^2(P)} + \lVert \varphi - \varphi \circ T^{-1}\circ T_0 \rVert^2_{L^2(P)}  + \lVert  \varphi_0   - \varphi \rVert^2_{L^2(P)} \bigr).
        \end{align}
By the $\eta/2$-strong convexity of $\varphi$, we have
\begin{align*}
\lVert T^{-1}\circ T_0 - \operatorname{Id}  \rVert^2_{L^2(P)}  = \lVert T^{-1}\circ T_0 - T^{-1}\circ T  \rVert^2_{L^2(P)} 
\le \frac{4}{\eta^2} \lVert T_0 - T  \rVert^2_{L^2(P)}. 
\end{align*}
By the gradient norm bound \eqref{eq:gradphibdd}, we have $\sup_{x \in [0,1]^\infty}\lVert T(x)\rVert \lesssim \lVert \varphi \rVert_{H^\gamma} \lesssim 1$. Combining this with the mean-value theorem, we obtain:
\begin{align*}
    \lVert \varphi - \varphi \circ T^{-1}\circ T_0 \rVert^2_{L^2(P)}   &\lesssim \sup_{x \in [0,1]^\infty} \lVert T(x)\rVert^2 \lVert \operatorname{Id} -  T^{-1}\circ T_0\rVert^2_{L^2(P)}
    \le \frac{4}{\eta^2} \lVert T - T_0\rVert^2_{L^2(P)}. 
\end{align*}
By the Poincar\'e inequality,
\[
 \lVert  \varphi_0   - \varphi \rVert^2_{L^2(P)} \le  \lVert T_0 - T\rVert^2_{L^2(P)} = \lVert \nabla \phi_0 - \nabla \phi \rVert^2_{L^2(P)}. 
\]
    Combining the previous three displays, we have 
    \begin{align*}
    \lVert \phi^c - \phi^c_0 \rVert^2_{L^2(Q)} \lesssim \lVert \nabla \phi  - \nabla \phi_0 \rVert^2_{L^2(P)} 
    &\lesssim \lVert \nabla \phi  - \nabla \bar\phi_J\rVert^2_{L^2(P)} + \lVert \nabla \bar\phi_J  - \nabla \phi_0\rVert^2_{L^2(P)} \\
    &\lesssim \tau^2 + 2^{-\frac{2}{1+2\alpha(\gamma)}J},
    \end{align*} 
    where the final inequality follows from $\phi \in \mF_J(\tau)$ and the approximation error bound (Lemma~\ref{lemma:approx_error}). As $\bar{\phi}_J$ is also $\eta/2$- strongly convex, the same argument yields $\lVert \bar{\phi}^c_J - \phi^c_0 \rVert^2_{L^2(Q)} \lesssim 2^{-\frac{2}{1+2\alpha(\gamma)}J}$; these two bounds result in:
    \begin{equation}
        \label{eq:conjlinfbdd}
        \lVert \phi^c - \bar{\phi}^c_J \rVert^2_{L^2(Q)} \lesssim \tau^2 + 2^{-\frac{2}{1+2\alpha(\gamma)}J}.
    \end{equation}
    We can now conclude that
\begin{equation} \label{eq:sigmabd}
    \sigma =\sup_{g \in \mG_J(\tau)} \lVert g \rVert_{L^2(\Pr)} \lesssim \tau + 2^{-\frac{2}{1+2\alpha(\gamma)}J}.
\end{equation}
Now we find an upper bound for $M$.

By writing both $\phi$ and $\bar\phi_J$ as Fourier series and applying the Cauchy-Schwarz inequality, we have $\lVert \phi \rVert_{L^\infty} \lesssim \lVert \phi \rVert_{H^{\gamma+2}}$ and $\lVert \bar\phi_J \rVert_{L^\infty} \lesssim \lVert \bar\phi_J \rVert_{H^{\gamma+2}}$. Consequently, 
\begin{equation}\label{eq:unif_bound}
    \begin{split}
    \lvert \phi(x) - \bar{\phi}_J(x) \rvert \le \lVert \phi \rVert_{L^\infty} + \lVert \bar{\phi}_J \rVert_{L^\infty} 
    &\lesssim \lVert \phi \rVert_{H^{\gamma+2}} + \lVert \bar{\phi}_J \rVert_{H^{\gamma+2}} \\
    &\le \lVert \phi \rVert_{H^{\gamma+2}} + \lVert \phi_0 \rVert_{H^{\gamma+2}} \\
    &\lesssim 1.
    \end{split}
\end{equation}

To bound $\phi^c - \bar{\phi}^c_J$, we first recall the relationships between Kantorovich's and Brenier's potentials: $\phi^c = \lVert \cdot\rVert^2 - 2\varphi^{*}$ and $\varphi^*(y) = \left\langle x_y,y \right\rangle - \varphi(x_y)$ where $x_y = (\nabla \varphi)^{-1}(y)$, we upper bound the conjugate difference as
\begin{align}
    \phi^c(y) - \bar{\phi}^c_J(y) &= 2(\varphi^{*}(y) - \bar{\varphi}^{*}_J ) \\
    &= 2\Bigl( \left\langle x_y, y \right\rangle - \varphi(x_y) - \bar{\varphi}^c_J(y)\Bigr)  \\
    &\le 2\Bigl( \left\langle x_y, y \right\rangle - \varphi(x_y) - \bigl\{\left\langle x_y, y \right\rangle - \bar{\varphi}_J(x_y)\bigr\}\Bigr)  \\
    &= 2(\bar{\varphi}_J(x_y) - \varphi(x_y)). \label{eq:varphiinf}
\end{align}
Similarly, we have $\bar{\phi}^c_J(y) - \phi^c(y) \le 2(\varphi(\bar{x}_y) - \bar{\varphi}_J(\bar{x}_y))$ where $\bar{x}_y = (\nabla \bar{\varphi}_J)^{-1}(y)$. Consequently, we obtain
\begin{equation} \label{eq:ubound}
    \begin{split}
        \lvert \phi^c(y) - \bar{\phi}^c_J(y) \rvert &\le 2(\bar{\varphi}_J(x_y) - \varphi(x_y))  + 2(\varphi(\bar{x}_y) - \bar{\varphi}_J(\bar{x}_y)) \\
        &\le 4 \lVert \varphi - \bar{\varphi}_J\rVert_{L^{\infty}} \\
        &=  2 \lVert \phi - \bar{\phi}_J\rVert_{L^{\infty}} \\
        &\lesssim 1.
\end{split}
\end{equation}
This gives us a bound for $M$ as
\begin{equation} \label{eq:linfbdd}
    M = \sup_{g \in \mG_J(\tau)} \lVert g\rVert_{L^{\infty}} \le \sup_{\phi \in \mF_J(\tau)} (\lVert \phi - \bar{\phi}_J\rVert_{L^{\infty}} + \lVert \phi^c - \bar{\phi}^c_J\rVert_{L^{\infty}}) \lesssim 1.
\end{equation}
With these $L^2$- and $L^{\infty}$-bounds over $\mG_J(\tau)$ in place, we go back to the Bousquet-Talagrand inequality~\eqref{eq:Znconc}, which yields with probability at least $1-e^{-u}$,
\begin{equation} 
    Z_n(\tau) \lesssim 2\Ep Z_n(\tau) + \tau\sqrt{\frac{2u}{n}} +\frac{u}{n},
\end{equation}
as desired.
\end{proof}

\begin{proof}[Proof of Lemma \ref{lem:exp_empirical_process}]

With the function spaces $\mF_J(\tau),\mG_J(\tau)$ defined as in \eqref{eq:mF}, and the probability distributions $\Pr,\operatorname{Pr}_n$ defined as in \eqref{eq:Prdef}, we use the following mixed entropy bound:

\begin{lemma}[{Mixed Entropy Bound~\cite[Theorem 2.14.21]{van1996weak}}] \label{lem:mixedbound}
    Let $\mF$ be a separable class of measurable functions with $\int f^2 \ \d \mathbb{P} \le \sigma^2$ for any $f \in \mF$ and $\sup_{f \in \mF} \lVert f\rVert_{\infty} \le M$, then
    \begin{equation}
        \label{eq:mixedEnt}
        \begin{split}
        \Ep \left[\sup_{f \in \mF} \lvert (\mathbb{P} - \mathbb{P}_n)(f) \rvert\right] &\lesssim \frac1{\sqrt{n}}\int_0^{\sigma} \sqrt{1 + \log N(\epsilon, \mF, L^2(\mathbb{P}))} \ \d \epsilon \\
        &\quad + \frac{1}{n}  \int_0^{2M} \bigl( 1 + \log N(\epsilon, \mF, L^{\infty}) \bigr) \ \d \epsilon. 
        \end{split}
    \end{equation}
\end{lemma}

We apply this lemma with $\mathbb{P} =\Pr, \mathbb{P}_n = \operatorname{Pr}_n$ and $\mF = \mG_J(\tau)$. From~\eqref{eq:sigmabd} and~\eqref{eq:linfbdd}, there exist constants $C_1$ and $C_2$ such that:
\[
\sigma \le C_1(\tau + 2^{-\frac{2}{1+2\alpha(\gamma)}J})= C_1\tilde\tau, \quad \text{ and } \quad M \le C_2.
\]
Now we consider the covering number of $\mG_J(\tau)$. Trivially, we have 
\begin{equation}\label{eq:cover1}
N(\epsilon, \mG_J(\tau), L^2(\Pr))  \le N(\epsilon, \mG_J(\tau), L^{\infty}).
\end{equation}
Thus, it suffices to consider the $L^{\infty}$-covering number of $\mG_J(\tau)$. Consider any $\phi_t,\phi'_{t'} \in \mF_J(\tau)$. With the same argument as that of~\eqref{eq:ubound}, we have
\[
\lVert \phi^c_t - \phi'^c_{t'}\rVert_{L^{\infty}} \le 2 \lVert \phi_t - \phi'_{t'}\rVert_{L^{\infty}}.
\]
Consequently, $\lVert \phi_t  \oplus \phi^c_t - \phi'_{t'}  \oplus \phi'_{t'}\rVert_{L^{\infty}} \leq 3 \lVert \phi_t - \phi'_{t'}\rVert_{L^{\infty}}$ and we have
\begin{equation}\label{eq:cover2}
N(\epsilon, \mG_J(\tau), L^{\infty}) \le N(\epsilon/3, \mF_J(\tau), L^{\infty}).
\end{equation}
Using the Cauchy-Schwarz inequality, we can show that the condition $\lVert \phi \rVert_{H^{\gamma+2}} \le 1$ implies $\lVert \phi\rVert_{L^{\infty}} \le 1$; thus the function space $\mF_J(\tau)$ is uniformly bounded by $1$. Consequently, writing $\phi_t= t \phi + (1-t) \bar{\phi}_J$ and $\phi'_{t'} = t' \phi' + (1-t') \bar{\phi}_J$ where $\phi,\phi' \in \mF_J$, we have the following bound
\[
 \lVert \phi_t - \phi'_{t'}\rVert_{\infty} \le \lVert \phi - \phi'\rVert_{\infty} + \lvert t - t'\rvert (\lVert  \phi\rVert + \lVert \phi'\rVert) \le \lVert \phi - \phi'\rVert_{\infty} + 2\lvert t - t'\rvert.
\]
Thus, letting $\mE(\mF_J)$ be an $\epsilon/6$-covering of $\mF_J$ and $\mE([0,1]) \coloneqq \{ k\epsilon/12: k = 0,1,\ldots, \lfloor 12/\epsilon \rfloor\}$, the set $\{ t\phi + (1-t)\phi_J : \phi \in \mE(\mF_J), t \in \mE([0,1])\}$ is an ($\epsilon/3$)-covering of $\mF_J$. Therefore, 
\begin{equation}\label{eq:cover3}
    \log N(\epsilon/3, \mF_J(\tau), L^{\infty}) \le \log N(\epsilon/6, \mF_J, L^{\infty}) + \log (12/\epsilon).
\end{equation}

We now aim to control the covering number of $\mF_J$. Let $\phi,\phi' \in \mF_J$. Let $\omega_l$ and $\omega'_l$ be the Fourier coefficients of $\phi$ and $\phi'$, respectively. By the Cauchy-Schwarz inequality,
\begin{align}
    &\lvert \phi(x) - \phi'(x) \rvert \\
    &\le  \sum_{s\in \N^{\infty}_0, (1+2\alpha(\gamma))\gamma(s)  \le J} \sum_{\lfloor 2^{s_i-1} \rfloor \leq \lvert l_i \rvert < 2^{s_i}} \lvert (\omega_l - \omega'_l)\psi_l(x) \rvert \notag \\
    &\le  \sum_{s\in \N^{\infty}_0,(1+2\alpha(\gamma))\gamma(s)  s_i \le J}\biggl( \sum_{\lfloor 2^{s_i-1} \rfloor \leq \lvert l_i \rvert < 2^{s_i}} (\omega_l - \omega'_l)^2 \biggr)^{1/2} \biggl(\sum_{\lfloor 2^{s_i-1} \rfloor \leq \lvert l_i \rvert < 2^{s_i}} \psi^2_l(x) \biggr)^{1/2} \notag \\
    &\asymp  \sum_{s\in \N^{\infty}_0,(1+2\alpha(\gamma))\gamma(s)  \le J} \Biggl\{ 2^{\sum_i (s_i - 1)^+/2}\biggl(\sum_{\lfloor 2^{s_i-1} \rfloor \leq \lvert l_i \rvert < 2^{s_i}} (\omega_l - \omega'_l)^2 \biggr)^{1/2}  \Biggr\} \notag \\
    &\le  \biggl( \sum_{s \in \N^{\infty}_0, (1+2\alpha(\gamma))\gamma(s) \le J} 2^{\sum_i (s_i - 1)^+} \biggr) \biggl(\sup_{s \in \N^{\infty}_0, \gamma(s) \le J}     \lVert \delta_s(\phi) - \delta_s(\phi')\rVert_{L^2([0,1]^{\infty})}\biggr). \label{eq:phiphip} 
\end{align}
We can bound the summation in \eqref{eq:phiphip} with the following useful lemma. The proof is provided in Appendix~\ref{sec:misclemmas}.
    \begin{lemma}\label{lemma:ddbdd}
      Let $\gamma$ be either $\gamma^{a,1}$ or $\gamma^{a,\infty}$. Let $a = (a_i)_{i=1}^{\infty}$ be a sequence that satisfies $a_i = \Omega(a^q)$ for some $q > 0$. Then, the following inequality holds for any $J \in \N$:
      \begin{align}
          \sum_{s \in \N^{\infty}_0: (1+2\alpha(\gamma))\gamma(s) < J} 2^{\sum_i (s_i - 1)^+} \lesssim 2^{\frac{ \alpha(\gamma)}{1+2\alpha(\gamma)}J}.  \label{eq:dysmall} 
     \end{align}
\end{lemma}

With this lemma, we continue the bound~\eqref{eq:phiphip}.
\begin{equation*}
    \lvert \phi(x) - \phi'(x) \rvert \lesssim 2^{\frac{\alpha(\gamma)}{1+2\alpha(\gamma)}J} \biggl(\sup_{s \in \N^{\infty}_0, (1+2\alpha(\gamma))\gamma(s) \le J}     \lVert \delta_s(\phi) - \delta_s(\phi')\rVert_{L^2([0,1]^{\infty})}\biggr).
\end{equation*}
What we have proved so far is that
\begin{equation}\label{eq:cover4}
 N(\epsilon/6, \mF_J, L^{\infty}) \le  N(\epsilon/6, \mF_J, \ell^{\infty}_s(2^{\frac{\alpha(\gamma)}{1+2\alpha(\gamma)}J}\lVert \delta_s(\cdot)\rVert_{L^2([0,1]^{\infty})})).
\end{equation}
From this, we have the following bound for the latter covering number. The proof is provided in Appendix~\ref{sec:misclemmas}.
\begin{lemma} \label{lem:covering_number}
    Assume that $\gamma = \gamma^{a,1}$ or $\gamma = \gamma^{a,\infty}$. Let $\alpha(\gamma) = \sup_{s \in \Z^{\infty}_0} \sum_i s_i/\gamma(s)$. Then,
    \begin{equation}\label{eq:cover5}
     \log  N(\epsilon, \mF_J, \ell^{\infty}_s(2^{\frac{\alpha(\gamma)}{1+2\alpha(\gamma)}J}\lVert \delta_s(\cdot)\rVert_{L^2([0,1]^{\infty})})) \lesssim J 2^{\frac{\alpha(\gamma)}{1+2\alpha(\gamma)} J} + 2^{\frac{\alpha(\gamma)}{1+2\alpha(\gamma)} J} \log_2 \left(\frac{1}{\epsilon} \right).
    \end{equation}
\end{lemma}
We are now ready to apply Lemma~\ref{lem:mixedbound}. First, we have
\begin{align}
    \Ep Z_n(\tau) &= \Ep \biggl[ \sup_{\phi \in \mG_J(\tau)} \lvert (\Pr - \operatorname{Pr}_n)(\phi) \rvert \biggr]  \\
    &\lesssim \frac{1}{\sqrt{n}}\underbrace{\int_0^{C_1 \tilde\tau} \sqrt{1 + \log N(\epsilon, \mG_J(\tau), L^{\infty})} \ \d \epsilon}_{I_1} + \frac{1}{n}  \underbrace{\int_0^{2C_2}  1 + \log N(\epsilon, \mG_J(\tau), L^{\infty})  \ \d \epsilon}_{I_2}.
\end{align}
Combining~\eqref{eq:cover1},~\eqref{eq:cover2},~\eqref{eq:cover3},~\eqref{eq:cover4} and~\eqref{eq:cover5}, we obtain
\begin{align}
    I_1 &\lesssim \int_0^{C_1 \tau} \sqrt{1 + N(\epsilon/6, \mF_J, L^{\infty}) + \log (1/\epsilon)} \ \d \epsilon \\
    &\lesssim \int_0^{C_1 \tilde\tau} \sqrt{1 + J 2^{\frac{\alpha(\gamma)}{1+2\alpha(\gamma)} J} + 2^{\frac{\alpha(\gamma)}{1+2\alpha(\gamma)} J} \log (1/\epsilon) + \log (1/\epsilon)} \ \d \epsilon \\
    &\lesssim \tilde\tau 2^{\frac{\alpha(\gamma)}{2(1+2\alpha(\gamma))} J}\sqrt{J \log (1+C_4/\tilde\tau)},
\end{align}
for some constant $C_4>0$, and similarly,
\begin{align}
   I_2  &\le \int_0^{2C_2}  1 + N(\epsilon/6, \mF_J, L^{\infty}) + \log (1/\epsilon) \ \d \epsilon \lesssim J 2^{\frac{\alpha(\gamma)}{1+2\alpha(\gamma)} J} .
\end{align}
Combining the bounds for $I_1$ and $I_2$, we obtain the final bound:
\begin{equation*}
    \Ep Z_n(\tau)  \lesssim \frac{\tilde\tau 2^{\frac{\alpha(\gamma)}{2(1+2\alpha(\gamma))} J}\sqrt{J \log (1+C_4/\tilde\tau)}}{\sqrt{n}} + \frac{J 2^{\frac{\alpha(\gamma)}{1+2\alpha(\gamma)} J} }{n}.
\end{equation*}
\end{proof}

\section{Proof of upper bound for the neural estimator}\label{sec:nn_proof}

    Here, we provide the proof of Theorem~\ref{thm:upper_bound_nn}. Let $\bar\varphi_{nn} \coloneqq \argmin_{\varphi \in \widetilde\mF_{J}} S(\varphi)$ and $\bar\phi_{nn}$ be the associated Kantorovich potentials. We introduce two function spaces analogous to those in \eqref{eq:mF}:
    \begin{equation*}
        \begin{split}
            \widetilde\mF_{J}(\tau) &\coloneqq \{ \phi_t \coloneqq t\phi + (1-t) \bar{\phi}_J : \phi \in \widetilde\mF_{J}, \lVert \nabla \phi_t - \nabla \bar{\phi}_{J} \rVert_{L^2(P)}
            \leq \tau, \lVert \nabla^2\phi_t - \nabla^2 \bar{\phi}_J\rVert_{\operatorname{op}} \le \eta/4 \},  \\
             \widetilde\mG_{J}(\tau) &\coloneqq \{(\phi - \tilde{\phi}_J) \oplus (\phi^c - \tilde{\phi}^c_J) : \phi \in \widetilde\mF_{J}(\tau)\}.
        \end{split}
    \end{equation*}
    Define probabilities $\Pr$ and $\operatorname{Pr}_n$ that act on any $f \oplus g \in \tilde\mG_J(\tau)$ as follows:
    \begin{equation}
        \Pr(f \oplus g) = \frac{1}{2}(P(f) + Q(g)), \quad\mbox{and}\quad \operatorname{Pr}_n(f \oplus g) = \frac{1}{2}(P_n(f) + Q_n(g)).
    \end{equation}
     We then proceed in the same manner as with the plug-in estimator. First, We introduce a localization $\hat{\varphi}_t = t\hat{\varphi}_{nn} + (1-t)\bar{\varphi}_J$, and $\hat\phi_t$ the associated Kantorovich potential, whose error can be decomposed as follows: 
    \begin{equation}
        \lVert \nabla \hat{\phi}_t - \nabla \phi_{0} \rVert^2_{L^2(P)} \lesssim \lVert \nabla \hat{\phi}_t - \nabla \bar\phi_{nn} \rVert^2_{L^2(P)} + \lVert \nabla \bar{\phi}_{nn} - \nabla \bar\phi_J \rVert^2_{L^2(P)} + \lVert \nabla \bar\phi_J - \nabla \phi_{0} \rVert^2_{L^2(P)}.
    \end{equation}
    Using Proposition~\ref{prop:localization}, we decompose the first term on the RHS:
    \begin{equation}
    \label{eq:stab_lower_nn}
     \lVert \nabla \hat{\phi}_t - \nabla \bar{\phi}_{nn} \rVert^2_{L^2(P)}  \lesssim  [S(\hat{\varphi}_t) -\hat{S}(\hat{\varphi}_t)]  + [\hat{S}(\bar{\varphi}_{nn}) - S(\bar{\varphi}_{nn})] + 2S_0(\bar{\varphi}_{nn}).
    \end{equation}
    The first two terms on the RHS are bounded by the following supremum of the empirical process:
    \begin{equation*}
        \widetilde{Z}_n \coloneqq  \sup_{g \in \widetilde\mG_J(\tau)} \lvert (\Pr - \operatorname{Pr}_n)(g) \rvert.
    \end{equation*}
      For the last term, we use the stability bound (Proposition~\ref{prop:stability}) and Fourier series tail bound \eqref{eq:S0bdd} to obtain:
      \begin{align*} 
      S_0(\bar{\varphi}_{nn}) &\lesssim \lVert \nabla \bar\phi_{nn} - \nabla \bar\phi_{J} \lVert^2_{L^2(P)} +  \lVert \nabla \bar\phi_{J} - \nabla \phi_0 \lVert^2_{L^2(P)} \\
      &\lesssim \lVert \nabla \bar\phi_{nn} - \nabla \bar\phi_{J} \lVert^2_{L^2(P)} + 2^{-\frac{2}{1+2\alpha(\gamma)}J}. 
      \end{align*} 
      From here, we strictly follow the proof of Lemma~\ref{lem:bound_error} (Section~\ref{sec:alllemmas}) to obtain the following bound:
        \begin{equation}\label{eq:estbddtilde}
       \lVert \nabla \hat{\phi}_t - \nabla \phi_{0} \rVert^2_{L^2(P)}  \lesssim   \widetilde{Z}_n + \lVert \nabla \bar\phi_{nn} - \nabla \bar\phi_J \rVert^2_{L^2(P)} + 2^{-\frac{2}{1+2\alpha(\gamma)}J} + \tau^2.
       \end{equation} 
    For $\widetilde{Z}_n$, we have the following bound from Lemma~\ref{lem:bound_empirical_process} with probability greater than $1-e^{-u}$:
    \begin{equation} \label{eq:concentratetilde}
    \widetilde{Z}_n \lesssim 2\Ep \widetilde{Z}_n + \tau\sqrt{\frac{2u}{n}} +\frac{u}{n}.
    \end{equation}
    To bound $\Ep \widetilde{Z}_n$, we again appeal to the mixed entropy bound \eqref{eq:mixedEnt}, which was used to prove Lemma~\ref{lem:exp_empirical_process}: For any class of separable and measurable functions $\mF$,
    \begin{equation}\label{eq:EZn_nn}
        \begin{split}
        &\Ep \left[\sup_{g \in \tilde\mG_J(\tau)} \lvert (\Pr - \operatorname{Pr}_n)(f) \rvert\right]  \\
        &\lesssim \frac1{\sqrt{n}}\int_0^{\sigma} \sqrt{1 + \log N(\epsilon, \tilde\mG_J(\tau), L^2(\Pr))} \, \d \epsilon  + \frac{1}{n}  \int_0^{2M} \bigl( 1 + \log N(\epsilon, \tilde\mG_J(\tau), L^{\infty}) \bigr) \, \d \epsilon, 
        \end{split}
    \end{equation}
    where $\sigma^2 = \sup_{g \in \tilde\mG_J(\tau)}\int g^2 \, \d \mathbb{P}$ and $M = \sup_{g \in \tilde\mG_J(\tau)} \lVert g\rVert_{\infty}$.

    From here, we closely follow the proof of Lemma~\ref{lem:exp_empirical_process} (Section~\ref{sec:alllemmas}). Specifically, the proofs for $\sigma \lesssim \tau + 2^{-\frac{2}{1+2\alpha(\gamma)}J}$ and $M \lesssim 1$ remain mostly the same; the only difference is that the bound for $\sup_{\phi \in \widetilde{F}_J}\lVert \phi \rVert_{L^\infty}$ follows directly from the fact that the class $\widetilde{\mF}_J$ is uniformly bounded due to the constraints on the networks' parameters. Applying the covering number bounds that we proved for the plug-in estimator \eqref{eq:cover3} and \eqref{eq:cover4}, yield:
    \begin{equation}\label{eq:cover_bounds_nn}
        \begin{split}
        \log N(\varepsilon, \widetilde\mG_J(\tau), L^2(\Pr)) &\lesssim \log N(\varepsilon, \widetilde\mG_J(\tau), L^\infty)  \\
        &\lesssim \log N(\varepsilon, \widetilde\mF_J(\tau), L^\infty)  \\
        &\lesssim \log N(\epsilon/6, \widetilde\mF_J, L^{\infty}) + \log(12/\varepsilon),
        \end{split}
    \end{equation}
    which can be bounded further with the help of the following lemma:
    \begin{lemma}[{Covering number bound for neural networks~\cite[Lemma 3]{Suzuki2018}}] \label{lem:nncov} The following bound for the covering number of $\widetilde{F}(W,L,R,B)$ holds:
        \begin{equation}\label{eq:nncov}
            \log N(\varepsilon, \widetilde{F}(W,L,R,B), L^\infty) \le 2 R L \log(\varepsilon^{-1} L(W+1)\max\{B , 1\} ). 
        \end{equation}
    \end{lemma}
    We will need to balance this bound with the following bound for the approximation error of $\bar\phi_J$ by $\bar\phi_{nn}$:
    \begin{lemma}[{Approximation error for neural networks~\cite[Proof of Theorem 7]{okumoto2021learnability}}] \label{lem:nn}
        Define the following notions of finite-dimensionality\footnote{To achieve dimension-independent upper bounds, we have modified the definition of $G$ from \cite{okumoto2021learnability}, where $G \coloneqq \sum_{s \in \N^\infty_0: \gamma(s) < J} 2^{\sum_i s_i}$. The lemma's statement remains valid under this new definition.}:
        \begin{align*}
            d_{\max} &\coloneqq \max \{i \in \N : \exists s \in \N^\infty_0 - \{\bm{0}\}, \gamma(s) < J \} \\
            s_{\max} &\coloneqq \max_{s\in \N^\infty_0: \gamma(s) < J} \max_{i \in \N} s_i, 
            \qquad G \coloneqq \sum_{s \in \N^\infty_0: \gamma(s) < J} 2^{\sum_i (s_i - 1)^+},
        \end{align*}
        where $(s_i - 1)^+ = \max\{s_i-1, 0\}$. For some constants $K,K' >0$, we specify the following neural network configuration:
        \begin{equation}\label{eq:WLRB}
        \begin{split}
            W &= 21 d_{\max} G \\
            L &= 2K \max \{d^2_{\max}, J^2, (\log G)^2, \log s_{\max}  \} \\
            R &= 1764 K d^2_{\max} \max \{d^2_{\max}, J^2, (\log G)^2, \log s_{\max} \} G \\
            B &= 2^{d_{\max}/2}K',
        \end{split}
        \end{equation}
        Consider any $f \in H^{\gamma}([0,1]^{\infty})$. Given $\tilde{J} >0$, denote $f_{\tilde{J}} =\coloneqq \sum_{s\in \N^\infty_0:\gamma(s) < \tilde{J}} \delta_s(f)$. There exists a neural network $f_{nn} \in \widetilde{F}(W,L,R,B)$ only depending on the first $d_{\max}$ coordinates of its input such that
        \begin{equation}\label{eq:nnapprox}
            \lVert f_{nn} - f_{\tilde{J}} \rVert_{L^{\infty}([0,1]^{d_{\max}})} \lesssim 2^{-\tilde{J}}.
        \end{equation}
    \end{lemma}
\begin{proof}[Proof for Mixed Smoothness]
    In this case, we assume that $\phi_0 \in H^{\gamma+2}([0,1]^{\infty})$ where $\gamma(s) = \gamma^{a,1}(s) = \sum_{i=1}^\infty a_i s_i$. Recall that for the mixed smoothness, we have $\alpha(\gamma) = 1/a_1$. For notational convenience, we denote
    \[a'_i = (1+2\alpha(\gamma))a_i. \quad \text{In particular, } a'_1 = (1+2\alpha(\gamma))a_1 = \frac{1+2\alpha(\gamma)}{\alpha(\gamma)}.\]
    Then, $H^{\gamma+2}([0,1]^{\infty}) = H^{\gamma'}([0,1]^{\infty})$ where $\gamma'(s) = \sum_{i=1}^\infty a'_i s_i$. With the assumption $a_i = \Omega(i^q)$ the function class parameters in Lemma~\ref{lem:nn} are as follows:
    \begin{align*}
        d_{\max} &\asymp \left\{J/(1+2\alpha(\gamma)) \right\}^{1/q} \\
        s_{\max} &\asymp  J/((1+2\alpha(\gamma))a_1) = J/a'_1.
    \end{align*}
    And by Lemma~\ref{lemma:ddbdd}, we have
    \[ G \lesssim 2^{\frac{\alpha(\gamma)}{1+2\alpha(\gamma)}J} = 2^{J/a'_1}.\]
    With these choices of $d_{\max}, s_{\max}$ and $G$, we derive neural network parameters via \eqref{eq:WLRB}: 
    \begin{align*}
        W \asymp J^{1/q} 2^{J/a'_1},\quad
        L \asymp J^{2+2/q},\quad
        R \asymp J^{2 + 4/q}2^{J/a'_1},\quad
        B \asymp 2^{J^{1/q}/2}.
    \end{align*}
    Then, from the fact that the Fourier series of $\bar\phi_J$ is truncated at dyadic scales $s$ up to $(1+2\alpha(\gamma))\gamma(s) < J$, we appeal to the Fourier truncation error bound \eqref{eq:approx} and the neural network approximation error bound \eqref{eq:nnapprox} with $\tilde{J} = J/(1+2\alpha(\gamma))$ to obtain:
    \begin{align}
        \lVert \nabla \bar\phi_{nn} - \nabla \phi_0 \rVert^2_{L^2(P)} 
        &\lesssim \lVert \nabla \bar\phi_{nn} - \nabla \bar\phi_J \rVert^2_{L^2(P)} + \lVert \nabla \bar\phi_J - \nabla \phi_0 \rVert^2_{L^2(P)} \\
        &\lesssim \lVert \nabla \bar\phi_{nn} - \nabla \bar\phi_J \rVert^2_{L^\infty([0,1]^{d_{\max}})} + 2^{-\frac{2}{1+2\alpha(\gamma)}J} \\
        &\lesssim 2^{-\frac{2}{1+2\alpha(\gamma)}J}.
    \end{align}
    In addition, plugging our choices of $W,L,R,B$ in \eqref{eq:nncov} yields the following bound. Notably, $G$ contributes the most to the bound through $R$.
    \begin{equation}\label{eq:nncov1}
        \begin{split}
            \log N(\varepsilon, \widetilde{F}(W,L,R,B), L^\infty) &\lesssim J^{2 + 7/q} 2^{J/a'_1} \log(J) \\
            &= J^{2 + 7/q} 2^{\frac{\alpha(\gamma)}{1+2\alpha(\gamma)}J} \log(J). 
            \end{split}
        \end{equation}
    This provides an upper bound for $\widetilde{Z}_n$ in \eqref{eq:estbddtilde} through the mixed entropy bound \eqref{eq:EZn_nn} and the covering number bound \eqref{eq:cover_bounds_nn}. From here, we proceed with the same proof for the plug-in estimator, starting from \eqref{eq:C5} and concluding at \eqref{eq:almostdone}; the only difference is that we now have an upper bound for the covering number \eqref{eq:nncov1} instead of \eqref{eq:cover5}. As a result, we obtain:
    \begin{equation}\label{eq:finalnn}
         \Ep \lVert  \nabla \hat{\phi}_{nn} - \nabla \phi_0\rVert^2_{L^2(P)}  \lesssim \frac{J^{2+7/q} 2^{\frac{\alpha(\gamma)}{1+2\alpha(\gamma)} J} \log n}{n} + 2^{-\frac{2}{1+2\alpha(\gamma)}J} + \frac{1}{n}.
    \end{equation}
    As with the plug-in estimator, we choose $J \approx \frac{1+2\alpha(\gamma)}{2+\alpha(\gamma)}\log_2 n$ in order to obtain the near-minimax-optimal rate: 
\[
\Ep \lVert  \nabla \hat{\varphi}_J - \nabla \varphi_0\rVert^2_{L^2(P)} = \Ep \lVert  \nabla \hat{\phi}_J - \nabla \phi_0\rVert^2_{L^2(P)} \lesssim n^{-\frac{2}{2 + \alpha(\gamma)}}(\log n)^2.
\]
\end{proof}

\begin{proof}[Proof for Anisotropic Smoothness]
    In this case, $\gamma(s) = \gamma^{a,\infty}(s) = \max_{i \in \N} a_i s_i$. We apply the same argument as in the mixed smoothness case. Specifically, since $d_{\max}$ and $s_{\max}$ are independent of our choice of $\gamma$, their values remain the same. In addition, using Lemma~\ref{lemma:ddbdd} again but with $\gamma = \gamma^{a,\infty}$, we obtain an analogous bound $G \lesssim 2^{\frac{\alpha(\gamma)}{1+2\alpha(\gamma)}J}$. As the neural network configuration depends solely on these three parameters, applying the same proof argument as the mixed smoothness case yields the final bound~\eqref{eq:finalnn} with $\gamma = \gamma^{a,\infty}$.

\end{proof}

\section{Proofs of Section \ref{sec:sobolevellipsoid}}\label{sec:ellipsoid_proof}

\begin{proof}[Proof of Theorem \ref{thm:lower_rate_log}]

Given $\delta>0$, it is known that the $\delta$-packing number $M$ of $\Theta^{\infty}(b)$ satisfies $\log M \asymp \delta^{-1/b}$, and with $d = \lceil \delta ^{-1/b} \rceil$, we can find a $\delta$-packing set that lives inside the truncated ellipsoid:
\begin{equation}\label{eq:trunc} \Theta^d(b) = \left\{ x\in \Theta^{\infty}(b) : x_i=0 \quad \forall  i \ge d+1 \right\}.
\end{equation}
For the rest of the proof, we shall identify $\Theta^d(b)$ as a subset of $\R^d$. Thus, there exists a set $\{x_i\}_{i=1}^{M} \subset \Theta^d(b) \subset \R^d$ such that $\lVert x_i - x_j\rVert > \delta$ for all $i\not= j$. 
We also set 
\begin{equation} \label{eq:defepsilond}
    \epsilon_n \coloneqq (1 / \log n)^{b} \qquad \text{and} \qquad d \coloneqq \lceil2\pi e \epsilon_n^{-1/b}  \rceil, 
\end{equation} 
which makes $\lVert x - \iota_d(x)\rVert^2 \le  \sum_{j=d+1}^{\infty}x^2_j \le d^{-2b}\sum_{j=d+1}^{\infty}j^{2b}x^2_i \le \epsilon^2_n$.

Let $g:\R^d \to \R$ be a symmetric smooth bump function with a compact support in $B^d(0,1)$. From this, we define a family of functions $g_i:\Theta^d(b) \to \R$ as
\begin{equation}\label{eq:defg1}
g_0 \equiv 0, \qquad g_i(x) = \frac{\kappa \delta^{-d/2+2}}{M^{1/2}d}g\left(\frac{x-x_i}{\delta}\right), \qquad i=1,\ldots,M,
\end{equation}
where $\kappa=\kappa(\eta,\beta)$ is a small constant to be chosen later. By the choices of $x_i$'s, the supports of these functions are pairwise disjoint, which allows us to calculate its gradient and Hessian:
\begin{align}
 \left\lVert \lVert \nabla g_i \rVert \right\rVert^2_{L^2([0,1]^d)} &\asymp \frac{\kappa^2\delta^2}{Md} \label{eq:ggrad} \\
 \left\lVert \lVert  \nabla^2g_i  \rVert_F \right\rVert^2_{L^2([0,1]^d)} &\asymp \frac{\kappa^2}{M}. \label{eq:ghess}
\end{align}
 For $M \ge 8$, the Varshamov-Gilbert lemma yields binary vectors $\tau^{(0)},\ldots,\tau^{(K)}\in [0,1]^M$ with $K \ge 2^{M/8}$ such that $\lVert \tau^{(k)} - \tau^{(k')}\rVert^2 \ge M/8$ for all $0 \le k \not= k' \le K$. We then define Brenier potentials $\varphi_k:\Theta^d(b) \to \R$ by
\begin{equation} \label{eq:philb}
\varphi_k(x) = \frac{1}{2} \lVert x\rVert^2 + \sum_{i=1}^M \tau^{(k)}_i g_i(x), \qquad k=0,\ldots,K.
\end{equation}
From the definition of $g_i$ in~\eqref{eq:defg1}, we can choose $\kappa$ sufficiently small so that $\varphi_k \in C^2(\R^d)$ and $\varphi_k$ is $\eta$-strongly convex and $\beta$-smooth. From this point on, $\kappa$ will play no further role in our analysis and it will be included in the hidden constants implied by the $\lesssim$ and $\gtrsim$ signs.

We now go back to the finite-dimensional Sobolev ellipsoid $\Theta^d(b)$. Define $P_0 = \operatorname{Unif}(\Theta^d(b))$ and $Q_k = (\nabla \varphi_k)_{\#} P_0$ for $k=1,\ldots,K$. By the Brenier Theorem, $T_k = \nabla \varphi_k$ is the optimal transport map from $P_0$ to $Q_k$.

Define $d(T_k,T_{k'}) = \int_{\mX} \lVert T_k - T_{k'}\rVert^2 dP_0$. Using~\eqref{eq:ggrad}, we obtain a separation lower bound for any $k\not= k'$
\begin{align}
    d(T_k, T_{k'}) = \int_{\Theta^d(b)} \lVert \nabla \varphi_k  - \nabla \varphi_{k'}  \rVert^2 dP_0 \gtrsim \frac{\delta^2}{d}.
\end{align}

We now find an upper bound for the KL divergence. By the definition, we have $Q_k \ll P_0$, and the Radon-Nikodym derivative $f_k = dQ_k/dP_0$ is given by
\[
f_k(y) = \frac{1}{\det \nabla^2\varphi_k((\nabla \varphi_k)^{-1} (y))}\mathbf{1}\left[ (\nabla \varphi_k)^{-1}(y) \in \Theta^d(b) \right] 
\]
We then apply useful inequality from the proof of~\cite[Theorem 6]{hutter2021minimax}:
\begin{align}
    D(Q_k\Vert P_0)  \leq \frac1{2} \sum^M_{i=1} \left\lVert \lVert  \nabla^2g_i  \rVert_F \right\rVert^2_{L^2([0,1]^d)},
\end{align}
which, combined with the Hessian upper bound~\eqref{eq:ghess}, yields an upper bound for the KL divergence:
\begin{equation}
    D(Q_k\Vert P_0) \gtrsim 1.
\end{equation}
We consider the probability distributions over a sample of $2n$ points: $\{P_0^{\otimes n} \otimes Q_k^{\otimes n}\}_{k=1}^K$ and $P_0^{\otimes n} \otimes P_0^{\otimes n}$. To apply the Fano's method, $\delta$ must be chosen to satisfy the following bound (here, $C$ is some universal constant):
\[
D\left(P_0^{\otimes n} \otimes Q_k^{\otimes n} \Vert P_0^{\otimes n}\otimes P_0^{\otimes n}\right) = nD(Q_k \Vert P_0) \lesssim n \lesssim e^{C\delta^{-1/b}} \asymp M \asymp \log K.
\]
This leads us to choosing $\delta \asymp (\log n)^{-b}$. The relation $d \asymp \delta^{-1/b} \asymp \log n$ results in
\begin{align}
    \inf_{\overline{T}} \sup_{P,Q \in \mP(\mX)} \Ep \left[ \int \|\overline{T} (x) - T_0(x)\|^2 dP(x) \right]  \gtrsim \frac{\delta^2}{d} \asymp \frac{1}{(\log n)^{2b+1}}.
\end{align}
\end{proof}

\section{Upper bound of the minimax estimation error with the $C^1$-class} \label{sec:Ck-upper}

We derive an upper bound of the minimax risk by introducing two estimators that could achieve a near-minimax optimal rate of $O((\log n)^{2b+1})$.

First, we estimate an OT plan based on samples truncated at finite dimension. Given $d\in \N$, we introduce the $d$-dimensional Sobolev ellipsoid:
\begin{align*}
\Theta^{d}(b) := \left\{ \bm{\theta} \in [0,1]^d : \sum_{j=1}^{d} j^{2b}\theta_j^2 < 1 \right\}.
\end{align*}
Our proof argument will require moving back and forth between $\Theta^\infty(b)$ and $\Theta^d(b)$. To this end, we denote by $\iota_d:\Theta^\infty(b) \to \Theta^d(b)$ the standard $d$-dimensional projection i.e. $\iota_d(x_1,\ldots,x_d,\ldots) = (x_1,\ldots,x_d)$ and by $\iota^{-1}_d :\Theta^d(b) \to \Theta^\infty_b$ its pseudo-inverse: $\iota^{-1}_d(x_1,\ldots,x_d) = (x_1,\ldots,x_d,0,0,\ldots)$.

We follow the plug-in approach of  \cite{manole2021plugin}: Given two samples $X_1,\ldots,X_n \sim P$ and $Y_1,\ldots,Y_n \sim Q$, we denote the empirical distributions ${P}^d_n = \frac{1}{n}\sum_{i=1}^n \delta_{\iota_d(X_i)}$ and ${Q}^d_n = \frac{1}{n}\sum_{i=1}^n \delta_{\iota_d(Y_i)}$. 
In preparation for the estimation, we derive an optimal transport plan from ${P}^d_n$ to ${Q}^d_n$ by solving the following minimization problem:
\begin{align*}
    \label{eq:OTprob}
    \hat{\pi} \in \argmin_{\pi \in \Pi_n} \sum_{i,j=1}^n \pi_{ij} \lVert \iota_d(X_i) - \iota_d(Y_j)\rVert^2, 
\end{align*}
where $\Pi_n \coloneqq \{(\pi_{ij} )_{i,j=1}^{n} \in [0,1]^{n \times n}:\sum_{i=1}^n \pi_{ij} = \sum_{j=1}^n \pi_{ij} = \frac{1}{n} \}$. We then define a \emph{nearest neighbor} estimator: For each $x \in \Theta^{d}(b)$, we consider the set of $X_i$'s whose truncations are closest to $x$ as
$
\textsf{NN}(x) \coloneqq \{i \in [n]: \lVert x - \iota_d(X_i)\rVert \le \lVert x - \iota_d(X_j)\rVert, \forall i \not= j \}$.
We define the nearest-neighbor estimator as
\begin{equation}
    \label{eq:NNest}
     \hat{T}^{\textsf{NN}}_{n,d}(x) \coloneqq \sum_{i \in \textsf{NN}(x)} \sum_{j=1}^n n \hat{\pi}_{ij} \iota_d(Y_j).
 \end{equation}

To analyze this estimator, we introduce some notation.
Let $\lambda^d$ be the $d$-dimensional Lebesgue measure and define the push-forward measure $P^d$ on any $\lambda^d$-measurable set $A$ as
    \begin{equation*}
        P^d(A) = (\iota_d)_{\#}P(A) = P(\iota^{-1}_d(A)),
    \end{equation*}

With this setup, we state some assumptions on $P$ and $Q$, on their supports $\Omega_P , \Omega_Q \subseteq \Theta^{\infty}(b)$.
\begin{assumption}\label{assum:e1}
   Both $P$ and $Q$ satisfy the assumptions in the Brenier theorem for infinite-dimensional spaces (Theorem~\ref{thm:brenier_infinite}). In addition, $Q$ is regular. For all $d\in \N$, $P^d \ll \lambda^d$ with density $p_d:\Theta^d(b) \to [0,\infty)$; and for all $x \in \Theta^d(b)$, $L^{-1}d^{-\ell} < p_d(x) < Ld^{u}$ for some $\ell,u \ge 0$ and some $L>0$, all independent of $d$. 
\end{assumption}
This assumption guarantees existence of the optimal transport maps $T_0$ from $P \to Q$, $T_{0,d}$ from $P^d \to Q^d$ and $T_d$ from $P \to Q^d$. We then impose additional regularities on these maps:
\begin{assumption} \label{assum:e15}
There exists a uniform constant $L \ge 1$ such that the transport maps $T_0$, $T_{0,d}$ and $T_d$ are bi-Lipschitz with constants $1/L$ and $L$.
\end{assumption}
This assumption ensures that the maps are well-behaved as we take the limit $d\to\infty$.

With these assumptions, we obtain a near-optimal upper bound of the aforementioned estimators in the following theorem. The proof is deferred to Appendix~\ref{sec:ellipsoid_proof}.

\begin{theorem} \label{thm:upper_log_1}
  Let $b > 1$ and $T_0:\Theta^{\infty}(b) \to \Theta^{\infty}(b)$ be the optimal transport map from $P$ to $Q$. If Assumption~\ref{assum:e1}, and \ref{assum:e15} hold with $d=\lceil \log n/(1+b+\ell+u)\log \log n) \rceil$, then we have the following as $n \to \infty$: 
    \begin{equation}
        \label{eq:NNLSub}
                       \Ep_{(X_{1:n}, Y_{1:n})} \left[ \lVert \iota^{-1}_d \circ \hat{T}^{\textsf{NN}}_{n,d} \circ \iota_d - T_0 \rVert^2_{L^2(P)}\right] \lesssim \left(\frac{\log \log n}{\log n}\right)^{2b-1}.
    \end{equation}
\end{theorem}

\begin{proof}
     For notational convenience, we denote $X_{i|d} = \iota_d(X_i)$ and $Y_{i|d} = \iota_d(Y_i)$. Let $V_i,\ldots,V_n$ be the Voronoi partition in $\Theta^d(b)$ based on $X_{1|d},\ldots,X_{n|d}$:
    \begin{equation*}
        V_i \coloneqq \{x \in \Theta^d(b) : \lVert x - X_{i|d}\rVert \le \lVert x - X_{j|d}\rVert, \quad \forall j\not= i\}, \qquad i=1,\ldots,n.
    \end{equation*}
    We shall use the following result obtained in the proofs of~\cite[Proposition 14 and 15]{manole2021plugin}:
\begin{lemma}
    Under Assumption~\ref{assum:e1}, we have the following bound for any $m \in [0,1]$:
\begin{equation}\label{eq:finiteUpper}
    \begin{split}
        \Ep \lVert \hat{T}^{\textsf{NN}}_{n,d} - T_{0,d}\rVert^2_{L^2(P^d)} &\lesssim L^2\Ep \Bigl[ \max_{1 \le i \le n} \rho_i^2 \Bigr] + n \Pr(\max_{1 \le i \le n} P^d(V_i) > m) \\
        &\qquad + mn L^2\left(W^2_2(P^d_n,P^d) + W^2_2(Q^d_n,Q^d)\right).
        \end{split}
    \end{equation}
\end{lemma}
    To control the right-hand side, we first consider $\max_{1 \le i \le n}\rho_i$ and $\max_{1 \le i \le n} P(V_i)$. Denote the empirical measure $P^d_n$ by $P^d_n(A) \coloneqq \frac{1}{n}\sum_{i=1}^n \mathbf{1}[X_{i|d} \in A]$ and $P^d = (\iota_d)_{\#}P$. Let $x_i \in \argmax_{x \in V_i} \lVert X_{i|d} -  x\rVert$ and $\rho_i = \lVert X_{i|d} -  x_i\rVert$. 
    Denote by $B^d(x_i,\rho_i)$ the $d$-dimensional open ball centered at $x_i$ with radius $\rho_i$. We then have $P^d_n(B^d(x_i,\rho_i)) = 0$ for all $i$, which allows us to use the following Vapnik-Chervonenkis inequality:
    \begin{lemma}[{\cite[Theorem 5.1]{Bousquet2004},~\cite[Lemma 16]{Chaudhuri2010}}]\label{lemma:VCineq}
        Let $\mB$ be the set of balls in $\R^d$. There exists a universal constant $C> 0$ such that for any $\delta > 0$, with probability at least $1-\delta$, we have for any ball $B \in \mB$,
        \begin{equation}
            \text{If }  \ \ P^d(B) \geq \frac{C}{n}\left[ d\log n + \log \left(\frac{1}{\delta} \right)\right]  \ \ \text{ then } \ \ P^d_n(B) > 0.
        \end{equation}
    \end{lemma} 
    It follows from this lemma that with probability at least $1-\delta$,
    \begin{equation}
        \label{eq:maxPd}
        \max_{1 \le i \le n} P^d(B^d(x_i,\rho_i)) \le \frac{C}{n}\left[ d \log n + \log \left(\frac{1}{\delta} \right) \right],
    \end{equation}
for some universal constant $C$. From Assumption~\ref{assum:e1} and the fact that $X_{i|d} \notin B^d(x_i,\rho_i)$ (as $B^d(x_i,\rho_i)$ is an open ball), we obtain the following inequalities; here, we use the maximal inequality \eqref{eq:maxPd} in the final step, which holds with probability at least $1-\delta$:
    \begin{align*}
        \max_{1 \le i \le n} \sup_{x_i \in V_i} \rho_i 
        &= \max_{1 \le i \le n} \sup_{x_i \in V_i} \left[\frac{1}{\lambda^d(B^d(0,1))} \lambda^d(B^d(x_i,\rho_i))\right]^{1/d}   \\
        &\le  L^{1/d} d^{\ell/d} \max_{1 \le i \le n} \sup_{x_i \in V_i} \left[\frac{1}{\lambda^d(B^d(0,1))} P^d(B^d(x_i,\rho_i) )\right]^{1/d}    \\
        &\le \max_{1 \le i \le n} \sup_{x_i \in V_i} d^{1/2+1/(2d)+\ell/d}  \left(P^d(B^d(x_i,\rho_i)\right)^{1/d}    \\
        &\le  d^{1/2+1/(2d)+\ell/d}  n^{-1/d}(d \log n + \log(1/\delta))^{1/d} ,
    \end{align*}

    With our choice of $d=\lceil \log n/(1+b+\ell+u)\log \log n) \rceil$ in~\eqref{eq:defepsilond}, we obtain for a sufficiently large $n$:
    \begin{align}
        \Ep \left[ \max_{1 \le i \le n} \rho_i^2  \right] &\lesssim d^{1+3/d+2\ell/d} \left(\frac{\log n}{n} \right)^{2/d} \\
        &\lesssim  \frac1{(\log n)^{2b+\ell+u+1}}.
    \end{align}

    We now derive a bound for $\Pr(\max_{1 \le i \le n} P^d(V_i) > m)$.
    By the absolute continuity of $P^d$ (Assumption~\ref{assum:e1}), the following inequalities hold :
    \begin{align}
        \max_{1 \le i \le n} P^d(V_i) &\le \max_{1 \le i \le n} P^d(B^d(X_{i|d}, \rho_i )) \\
        &\le L d^u \max_{1 \le i \le n} \lambda^d(B^d(X_{i|d}, \rho_i )) \\
        &= L d^u \max_{1 \le i \le n} \lambda^d(B^d(x_i, \rho_i )).
        \label{eq:lamnbadB}
    \end{align}
    We again apply Assumption~\ref{assum:e1} and the maximal inequality~\eqref{eq:maxPd} to obtain the following inequality with probability at least $1-\delta$: 
    \begin{align}
        \eqref{eq:lamnbadB}
        &\leq L^2 d^\ell d^u  \max_{1 \le i \le n} P^d(B^d(x_i, \rho_i ))   \\
        &\lesssim d^\ell d^u \frac{1}{n}\left( d \log n + \log (1/\delta) \right) .
    \end{align}
    With $d \asymp \log n / \log \log n$ and $\delta = 1/n^2$, we obtain the following bound with probability at least $1-1/n^2$:
    \begin{equation*}
        \max_{1 \le i \le n} P(V_i)  \le C_3\frac{ (\log n)^{2+\ell + u}}{n}, 
    \end{equation*}
    for some constant $C_3>0$.

    Going back to~\eqref{eq:finiteUpper}, we choose $m = C_3(\log n)^{2+\ell+u}/n$ so that
    \begin{align}
        \Ep \lVert \hat{T}^{\textsf{NN}}_{n,d} - T_{0,d}\rVert^2_{L^2(P^d)} &\lesssim \frac{1}{(\log n)^{2b+\ell+u+1}} + \frac{1}{n} + (\log n)^{2+\ell+u}\Ep\left[W^2_2(P^d_n,P^d) + W^2_2(Q^d_n,Q^d)\right]
    \end{align}
    We now employ the upper bound of distribution estimation in functional spaces~\cite[Theorem 1]{fournier2015rate} that $\Ep W^2_2(P^d_n,P^d) \lesssim d  n^{-2/d}$ and $\Ep W^2_2(Q^d_n,Q^d) \lesssim d  n^{-2/d}$ (here, we also track in the proof the dimension-dependent constants). With $d=\lceil \log n/(1+b+\ell+u)\log \log n) \rceil$, we obtain:
    \begin{align*}
               \Ep \lVert \hat{T}^{\textsf{NN}}_{n,d} - T_{0,d}\rVert^2_{L^2(P^d)} &\lesssim \frac{1}{(\log n)^{2b+\ell+u+1}} +\frac{(\log n)^{2+\ell+u}d}{n^{2/d}} \\
               &\lesssim \frac{1}{(\log n)^{2b+\ell+u+1}} +\frac{(\log n)^{2+\ell+u}(\log n)}{(\log n)^{2  + 2\ell + 2u + 2}} \\
               &\lesssim \frac1{(\log n)^{2b + \ell + u-1}} .
    \end{align*}
    In the final step, we use the following lemma, which bounds the error of approximating $T_0$ by the $d$-dimensional OT map extended to infinite dimensions: $\widetilde{T}_{0,d} \coloneqq  \iota^{-1}_d \circ T_{0,d} \circ \iota_d$. The proof is deferred to Appendix~\ref{sec:misclemmas}.
    \begin{lemma}\label{lem:Td_bdd}
    If the measures $P,Q,P^d,Q^d$ satisfy Assumption~\ref{assum:e1} and the optimal transport maps $T_0,T_{0,d},T_d$ satisfy Assumption~\ref{assum:e15}, then we have the following bound:
        \begin{align*}
            \int \lVert T_0(x) - \tilde{T}_{0,d}(x) \rVert^2 dP(x)  \lesssim \frac1{d^{2b-1}}.
        \end{align*}
    \end{lemma}
    Using this lemma with $d\asymp \log n / \log \log n$, we derive the claimed upper bound:
    \begin{align*}
        \Ep \lVert \iota^{-1}_d \circ \hat{T}^{\textsf{NN}}_{n,d} \circ \iota_d - T_0 \rVert^2_{L^2(P)} &\lesssim \Ep \lVert \iota^{-1}_d \circ \hat{T}^{\textsf{NN}}_{n,d} \circ \iota_d - \widetilde{T}_{0,d}\rVert^2_{L^2(P)} + \Ep \lVert \widetilde{T}_{0,d} - T_0 \rVert^2_{L^2(P)} \\
        &= \Ep \lVert \hat{T}^{\textsf{NN}}_{n,d}  - T_{0,d}\rVert^2_{L^2(P^d)} + \Ep \lVert \widetilde{T}_{0,d} - T_0 \rVert^2_{L^2(P)} \\
        &\lesssim \frac1{(\log n)^{2b + \ell + u-1}} + \left(\frac{\log \log n}{\log n}\right)^{2b-1} \\
        &\lesssim \left(\frac{\log \log n}{\log n}\right)^{2b-1}.
    \end{align*}
\end{proof}

\section{Miscellaneous proofs}\label{sec:misclemmas}

\begin{proof}[Proof of Lemma~\ref{lemma:ddbdd}]
We use the following bound for dyadic sums with weighted $\ell^1$-type exponents:
\begin{lemma}[{\cite[Lemma 18]{okumoto2021learnability}}] \label{lemma:sumbdd}
    Suppose that $(b_i)_{i=1}^{\infty}$ is a monotonically nondecreasing sequence such that $b_1 = 1$. Then, for any $J>0$ and $\beta>0$, the following inequality holds:
    \begin{align}
       \sum_{s\in \N^{\infty}_0: \sum_i b_is_i  < J} 2^{\sum_i s_i} &\le 8 \left( \prod_{i=2}^{\infty} \frac{1}{1-2^{-(b_i-1)}} \right) 2^J.  \label{eq:ddlo} 
   \end{align}
   whenever the products on the right-hand sides converge.
\end{lemma}

We now prove Lemma~\eqref{lemma:ddbdd}. For the case $\gamma(s) = \gamma^{a,1}(s) = \sum_i a_i s_i$, we have $\alpha(\gamma) = a^{-1}_1$. The inequality follows directly from~\eqref{eq:ddlo} with $b_i = a_i/a_1$ by observing that
    \begin{equation*}
        \prod_{i=2}^{\infty} \frac{1}{1-2^{-(b_i-b_1)}} = \prod_{i=2}^{\infty} \frac{1}{1-2^{-\frac{a_i - a_1}{a_1}}},
    \end{equation*}
    which converges absolutely since $a_i = \Omega(i^q)$. 
    It thus follows from~\eqref{eq:ddlo} that
    \begin{align}
        \sum_{s\in \N^{\infty}_0: (1+2\alpha(\gamma))\gamma^{a,1}(s) < J} 2^{\sum_i (s_i - 1)^+}  
        &\le \sum_{s\in \N^{\infty}_0: \sum_i a_ia^{-1}_1(a_1+2) s_i < J} 2^{\sum_i s_i} \\
        &= \sum_{s\in \N^{\infty}_0: \sum_i b_is_i < J/(a_1+2)} 2^{\sum_i s_i} \\
        &\lesssim 2^{\frac{1}{a_1+2}J} \\
        &=2^{\frac{\alpha(\gamma)}{1+2\alpha(\gamma)}J}.
    \end{align}
    Now we consider $\gamma^{a,\infty}(s) = \max_i a_i s_i$. We observe that the condition $(1+2\alpha(\gamma))\gamma^{a,\infty}(s) = (1+2\alpha(\gamma))\max_i a_i s_i < J$ implies $s_i < J/(a_i(1+2\alpha(\gamma)))$ for all $i$. Therefore,
    \begin{align}
        \sum_{s\in \N^{\infty}_0: (1+2\alpha(\gamma))\gamma^{a,\infty}(s) < J} 2^{\sum_i (s_i - 1)^+} &\leq \prod_{i=1}^{\infty} \sum_{s_i = 0}^{\lfloor J/(a_i(1+2\alpha(\gamma))) \rfloor} 2^{(s_i - 1)^+}  \\
        &= 2^{\sum_i \left(\lfloor J/(a_i(1+2\alpha(\gamma))) \rfloor \right)} \\
        &\leq 2^{\frac1{1+2\alpha(\gamma)}J\sum_i 1/a_i}  \\
        &= 2^{\frac{\alpha(\gamma)}{1+2\alpha(\gamma)}J}.
    \end{align}

\end{proof}

\begin{proof}[Proof of Lemma \ref{lem:covering_number}]
  
    We shall use the following well-known result on the covering number of the unit ball:
    There exists a universal constant $C>0$ such that for any $\epsilon \le 1$ we have
    \begin{equation}\label{eq:seqcover}
        \log N(\epsilon, B_1(\R^d),B_1(\R^d)) \asymp d \log_2 \frac{C}{\epsilon}.
    \end{equation}

    For any $s\in\N^{\infty}_0$ with $(1+2\alpha(\gamma))\gamma(s) < J$, we denote 
    \[\mathcal{F}_{J,s} \coloneqq \biggl\{\sum_{\lfloor 2^{s_i-1} \rfloor \leq \lvert l_i \rvert < 2^{s_i} } \omega_l\psi_l : 2^{2(1+2\alpha(\gamma))\gamma(s)} \omega^2_l \leq 1\biggr\}. \]
    We can represent each function in $\mF_{J,s}$ by its vector of coefficients $\bm{\omega} \in \R^d$. By simple counting, we have $d = 2^{\sum_i (s_i - 1)^+}$ (recall that there are only finitely many nonzero $s_i$'s). Let $N=N(\epsilon2^{-\frac{\alpha(\gamma)}{1+2\alpha(\gamma)}J+(1+2\alpha(\gamma))\gamma(s)}, B_1(\R^d),B_1(\R^d))$. We denote an $\epsilon2^{-\frac{\alpha(\gamma)}{1+2\alpha(\gamma)}J+(1+2\alpha(\gamma))\gamma(s)}$-covering of the unit ball $B_1(\R^d)$ by: 
    \[ \left\{ \bm{\omega}_i \coloneqq (\omega_{il})_{l \in \Z^{\infty}_0: \lfloor2^{s_i-1} \rfloor \le \lvert l_i \rvert < 2^{s_i}} \in B_1(\R^d): 1 \le i \le N \right\},\] 
    Consider the set of functions $\mathcal{I}_s = \{ \phi_i, 1 \le i \le N\}$ where
    \[ \phi_i \coloneqq 2^{-(1+2\alpha(\gamma))\gamma(s)}\sum_{l \in \Z^{\infty}_0: \lfloor2^{s_i-1} \rfloor \le \lvert l_i \rvert < 2^{s_i}}\omega_{il}\psi_l,  \]
    where the multiplicative factor in front of the summation was chosen so that $\mathcal{I}_s \subset \mF_{J,s}$. Consider another function $\phi' \in \mF_{J,s}$ whose Fourier frequencies $l$ are restricted to $\{ l: \lfloor 2^{s-1} \rfloor \leq \lvert l_i \rvert <2^s\}$. Then it can be written as $2^{-(1+2\alpha(\gamma))\gamma(s)}\sum_{l \in \Z^{\infty}_0: \lfloor2^{s_i-1} \rfloor \le \lvert l_i \rvert < 2^{s_i}}\omega'_l\psi_l$ for some $\bm{\omega'} \coloneqq (\omega'_l)_{l \in \Z^{\infty}_0: \lfloor2^{s_i-1} \rfloor \le \lvert l_i \rvert < 2^{s_i}} \in B_1(\R^d)$. As there exists $i \in [N]$ such that $\lVert \bm{\omega}_{i} - \bm{\omega'}\rVert \le \epsilon2^{-\frac{\alpha(\gamma)}{1+2\alpha(\gamma)}J+(1+2\alpha(\gamma))\gamma(s)}$, we have 
    \[
        2^{\frac{\alpha(\gamma)}{1+2\alpha(\gamma)}J}\lVert \phi_i - \phi'\rVert_{L^2([0,1]^{\infty})} = 2^{\frac{\alpha(\gamma)}{1+2\alpha(\gamma)}J-(1+2\alpha(\gamma))\gamma(s)}\lVert \bm{\omega}_{i} - \bm{\omega'}_{il}\rVert \le \epsilon. 
    \]
    This inequality and the covering number of the unit ball \eqref{eq:seqcover} imply that
    \begin{align}
        &\log N(\epsilon, \mF_{J,s}, 2^{\frac{\alpha(\gamma)}{1+2\alpha(\gamma)}J}\lVert \cdot\rVert_{L^2([0,1]^{\infty})}) \\
        &\le \log N(\epsilon2^{-\frac{\alpha(\gamma)}{1+2\alpha(\gamma)}J+(1+2\alpha(\gamma))\gamma(s)}, B_1(\R^d),B_1(\R^d)) \\
        &\asymp d \log_2 \left(\frac{C}{\epsilon 2^{-\frac{\alpha(\gamma)}{1+2\alpha(\gamma)}J+(1+2\alpha(\gamma))\gamma(s)}} \right)\\
        &\le 2^{\sum_i (s_i - 1)^+} \log_2 \left(\frac{C}{\epsilon 2^{-\frac{\alpha(\gamma)}{1+2\alpha(\gamma)}J+(1+2\alpha(\gamma))\gamma(s)}} \right) \\
        &\le 2^{\sum_i (s_i - 1)^+} \log_2 \left(\frac{C2^{\frac{\alpha(\gamma)}{1+2\alpha(\gamma)}J}}{\epsilon} \right) .
    \end{align}
    We then take the sum over $s$ and apply~\eqref{eq:dysmall} to obtain
    \begin{align}
        &\log N(\epsilon, \mF_J(\tau), \ell^{\infty}_s(2^{\frac{\alpha(\gamma)}{1+2\alpha(\gamma)}J}\lVert \delta_s(\cdot)\rVert_{L^2([0,1]^{\infty})}))  \\
        &\le \sum_{s \in \N^{\infty}_0, (1+2\alpha(\gamma))\gamma(s) \le J}     \log N(\epsilon, \mF_{J,s}, 2^{\frac{\alpha(\gamma)}{1+2\alpha(\gamma)}J}\lVert \cdot\rVert_{L^2([0,1]^{\infty})}) \\
        &\le \sum_{s \in \N^{\infty}_0, (1+2\alpha(\gamma))\gamma(s) \le J} 2^{\sum_i (s_i - 1)^+} \log_2 \left(\frac{C2^{\frac{\alpha(\gamma)}{1+2\alpha(\gamma)} J}}{\epsilon} \right) \\
        &\lesssim  2^{\frac{\alpha(\gamma)}{1+2\alpha(\gamma)} J} \log_2 \left(\frac{C2^{\frac{\alpha(\gamma)}{1+2\alpha(\gamma)}J}}{\epsilon} \right) \\
        &\lesssim J 2^{\frac{\alpha(\gamma)}{1+2\alpha(\gamma)} J} + 2^{\frac{\alpha(\gamma)}{1+2\alpha(\gamma)} J} \log_2 \left(\frac{1}{\epsilon} \right) .
    \end{align}
\end{proof}

\begin{proof}[Proof of Lemma \ref{lem:Td_bdd}]
    The proof relies on the following stability bound:

\begin{lemma}[{$W_2$ Stability Bound}] \label{thm:stability}
Given $d \in \N \cup \{\infty\}$ and $b>0$. Let $P$ and $Q$ be probability measures on $\Theta^d(b)$ such that:
\begin{enumerate}
    \item If $d<\infty$, then $P$ and $Q$ are measures on $\Theta^d(b)$ that are absolutely continuous with respect to the $d$-dimensional Lebesgue measure $\lambda^d$.
    \item If $d=\infty$, then $P$ and $Q$ are measures on $\Theta^\infty(b)$ that satisfy the assumptions in Theorem~\ref{thm:brenier_infinite}. In addition, $P$ is bounded and $Q$ is regular.
\end{enumerate}
In addition, assume that the optimal transport map $T_0$ from $P$ to $Q$ is $L$-bi-Lipschitz, that is, for any $x,y \in \Theta^{d}(b)$,
\[ L^{-1} \lVert x - y \rVert \le \lVert T_0(x) - T_0(y) \rVert \le L \lVert x - y \rVert, \]
and $T_0 = \nabla \varphi_0$ for some convex function $\varphi_0 \in \mathcal{C}^2(\Theta^d(b);\R)$.

Consider any probability measure $\widehat{Q}$ on $\Theta^d(b)$. Let $\widehat{T}$ be the unique optimal transport map from $P$ to $\widehat{Q}$. Then,
\begin{align}\label{eq:W2_bdd}
\lVert \widehat{T} - T_0 \rVert_{L^2(P)}^2 &\lesssim L^2 W_2^2(\widehat{Q}, Q).
\end{align}
\end{lemma}
\begin{proof}
    The proof for the case $d<\infty$ follows directly from \cite[Theorem 6]{manole2021plugin}. For the case $d=\infty$, the assumptions in Theorem~\ref{thm:brenier_infinite} guarantee the existence of a Brenier potential $\varphi_0$. In addition, since $P$ is supported on the bounded set $\Theta^\infty(b)$ and $Q$ is regular, the map $T_0$ is invertible, and its inverse is given by $S_0 \coloneqq \nabla \varphi^*_0$. The remainder of the proof then proceeds analogously to the finite-dimensional case.
\end{proof}
To apply this lemma, we first define $\tilde{P}^d \coloneqq (\iota_d^{-1})_\# P^d$, $\tilde{Q}^d = (\iota_d^{-1})_\# Q^d$ and $\tilde{T}_{0,d} \coloneqq \iota_d^{-1} \circ T_{0,d} \circ \iota_d$.
By these definitions, $T_{0,d}$ being the optimal transport map from $P^d$ to $Q^d$ implies that $\tilde{T}_{0,d}$ is the optimal transport map from $\tilde{P}^d$ to $\tilde{Q}^d$. 

We decompose the squared error as follows:
\begin{equation} \label{eq:triangle_sq}
\lVert T_0 - \tilde{T}_{0,d} \rVert_{L^2(P)}^2 \lesssim \lVert T_0 - T_d \rVert_{L^2(P)}^2 + \lVert T_d - \tilde{T}_{0,d} \rVert_{L^2(P)}^2.
\end{equation}
Applying \eqref{eq:W2_bdd} with $\widehat{T} = T_d$ and $\widehat{Q} = \widetilde{Q}^d$ to the first term, we obtain:
\begin{equation}
\left\lVert T_0 - T_d \right\rVert_{L^2(P)}^2 \lesssim L^2 W_2^2(Q, \tilde{Q}^d).
\end{equation}
Focusing on the second term, we first claim that both $\tilde{P}^d$ and $\tilde{Q}^d$ are regular measures. To prove this, we let $A \subset \Theta^\infty(b)$ be any Gaussian null set. Consider any measure $\mu$ on $\R^\infty$ of the form $\mu = \mu_{\le d} \otimes \mu_{>d}$ where $\mu_{\le d}$ is an arbitrary Gaussian measure on the first $d$ coordinates, and $\mu_{> d}$ is a Gaussian measure on the remaining coordinates. Therefore, 
\[ \mu(A) \le \mu(\iota^{-1}_d(\iota_d(A))) = \mu_{\le d}(\iota_d(A)) = 0. \] 
As the projection map $\iota_d:\R^\infty \to \R^d$ is the composite of $d$ continuous linear functionals, for any Gaussian measure $\nu$ on $\R^\infty$, $\nu \circ \iota^{-1}_d$ is a Gaussian measure on $\R^d$. Consequently, $\nu(\iota^{-1}_d(\iota_d(A))) = 0$; in other words, $\iota^{-1}_d(\iota_d(A))$ is also Gaussian null. Therefore, 
\[ \tilde{P}^d(A) = P(\iota^{-1}_d(\iota_d(A))) =0.\] 
The same holds for $\tilde{Q}^d$ for all Gaussian null set $A$, allowing us to conclude that $\tilde{P}^d$ and $\tilde{Q}^d$ are regular. 

Consequently, there are inverse maps $S_d = T_d^{-1}$ and $\widetilde{S}_{0,d} = \widetilde{T}_{0,d}^{-1}$. By applying a change of variables using $P = (S_d)_{\#} \tilde{Q}^d$:
\begin{align}
\lVert T_d - \tilde{T}_{0,d} \rVert_{L^2(P)}^2 &= \int \lVert T_d(x) - \tilde{T}_{0,d}(x) \rVert^2 \d P(x) \\
&= \int \lVert T_d(S_d(y)) - \tilde{T}_{0,d}(S_d(y)) \rVert^2 \d \widetilde{Q}^d(y) \\
&= \int \lVert y - \tilde{T}_{0,d}(S_d(y)) \rVert^2 \d \tilde{Q}^d(y) \\
&= \int \lVert \tilde{T}_{0,d}(\tilde{S}_{0,d}(y)) - \tilde{T}_{0,d}(S_d(y)) \rVert^2 \d \tilde{Q}^d(y). \label{eq:TSy}
\end{align}
We now use fact that $\tilde{T}_{0,d}$ is $L$-Lipschitz and the stability bound \eqref{eq:W2_bdd} to obtain:
\begin{align}
\eqref{eq:TSy} &\le L^2 \int \lVert \tilde{S}_{0,d}(y) - S_d(y) \rVert^2 \d \tilde{Q}^d(y) \\
&\le L^2 \left( L^2 W_2^2(P, \tilde{P}^d) \right) = L^4 W_2^2(P, \tilde{P}^d).
\end{align}
We can bound the Wasserstein distance by exploiting the decaying tail of the Sobolev ellipsoid:
\begin{align*}
W_2^2(P, \tilde{P}^d) \le \int \left\lVert x - \iota_d^{-1}(\iota_d(x)) \right\rVert^2 \d P(x) &= \int \sum_{j=d+1}^{\infty} x_j^2 \d P(x)\\
&\le \sum_{j=d+1}^\infty \frac1{j^{2b}} \\
&\lesssim \frac1{d^{2b-1}}.
\end{align*}
We also have the same bound for $W_2^2(Q, \tilde{Q}^d)$. Substituting these bounds back into \eqref{eq:triangle_sq} yields:
\begin{align*}
\int \lVert T_0(x) - \tilde{T}_{0,d}(x) \rVert^2 dP(x) &\lesssim W^2_2(P, \widetilde{P}^d) + W^2_2(Q, \widetilde{Q}^d) \lesssim \frac1{d^{2b-1}}.
\end{align*}
\end{proof}

\section{Additional experiments}

\subsection{Accuracy of the estimator} 
We assess how close the estimator is to the actual OT map. We fit the neural estimator with $d=200$, $q=0$ and varying sample sizes $n\in\{10,30,100\}$. To visualize the true and estimated maps, we plot each component function by varying the input along the corresponding coordinate axis while keeping other coordinates fixed at zero. Specifically, for $i\in \{1,10,100\}$, we plot the graphs of $T^i_0(c) \coloneqq (T_0(ce_i))_i$ and $\hat{T}^i(c) \coloneqq (\hat{T}(ce_i))_i$, where $e_i$ is the $i$-th standard basis vector and $c \in [0,1]$.

The plots in Figure~\ref{fig:1} compare the fit of the neural estimator between $n=10,30$ and $100$ for $i=1,10$ and $100$. We can see that, even with $200$-dimensional inputs, with $T_0^{100}$ being close to the identity function, our estimator (dark blue line) accurately estimates the true map (red line) with only $100$ sample points.  

\begin{figure}[t]
    \centering
    \includegraphics[width=\textwidth]{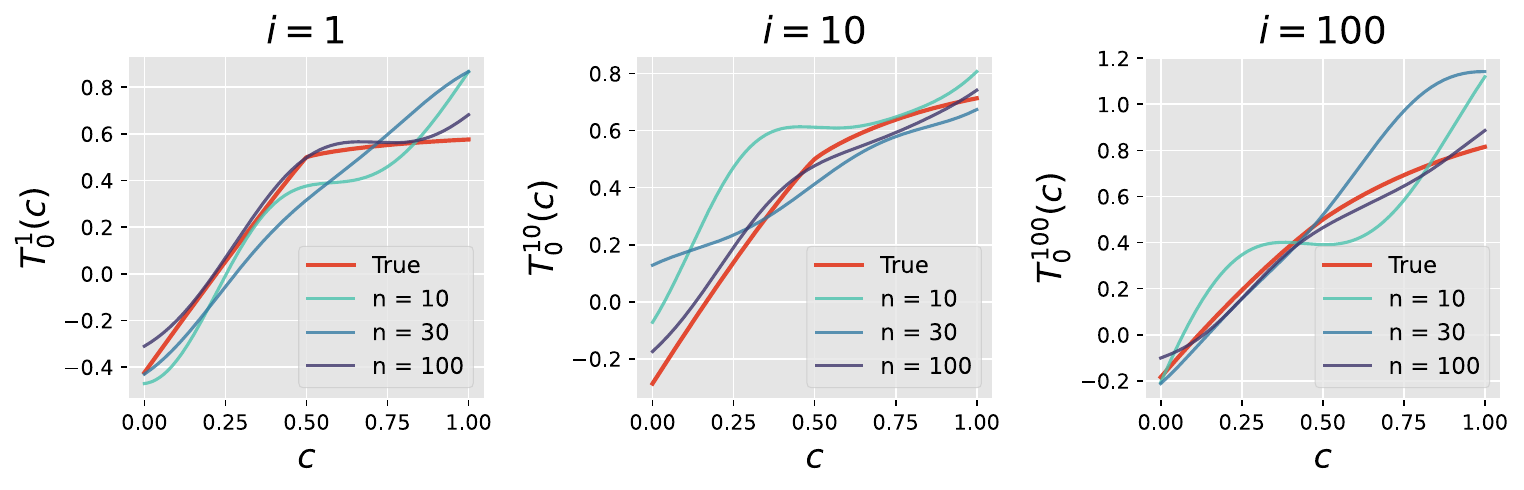}
        \caption{The true and estimated values of $T^{i}_0(c) = (T_0(ce_i))_i$ where $e_i$ is the $i$-th standard basis vector.}
    \label{fig:1}
\end{figure}

\subsection{Quality of the transports}
We compare the functional transports of the function $f(x)=x^2$ defined on the interval $[-0.02\pi, 0.02\pi]$. First, we convert this function into a finite-dimensional representation by applying a discrete cosine transform (DCT), yielding a sequence of 200 coefficients whose values are contained in $[0,1]$. We then train the neural estimator with $d=200$, $q=0$ and varying sample sizes $n\in \{100, 500, 1000\}$. The 200-dimensional input sequences are transported using both the true and estimated maps, followed by an inverse cosine transform to reconstruct a function in the original space.

The plots of the original function ($f(x)=x^2$) and the transported functions are shown in Figure~\ref{fig:2a}, where the transport of the true OT map is displayed in blue and those of the estimated maps with $n=100,500$ and $1000$ are displayed in orange. The plots of the transported functions reveal pronounced oscillatory behavior in both the true and estimated cases. Even minor discrepancies between the transported coefficients can result in substantially different oscillation patterns, primarily due to the fast decay in the DCT coefficients. Specifically, for $n=500$, the transports from both the true and estimated maps show matching oscillation patterns near the domain boundaries of $f$. However, in the domain's interior, the transport of the true map is noticeably smoother than that of the estimated map. Increasing the number of sample points to 1000 yields markedly improved results, with the estimated transport closely approximating the true transport throughout the entire functional domain.   

\begin{figure}[t]
    \centering
    \begin{minipage}[c]{0.99\textwidth}
        \begin{subfigure}{\textwidth}
            \includegraphics[width=0.99\hsize]{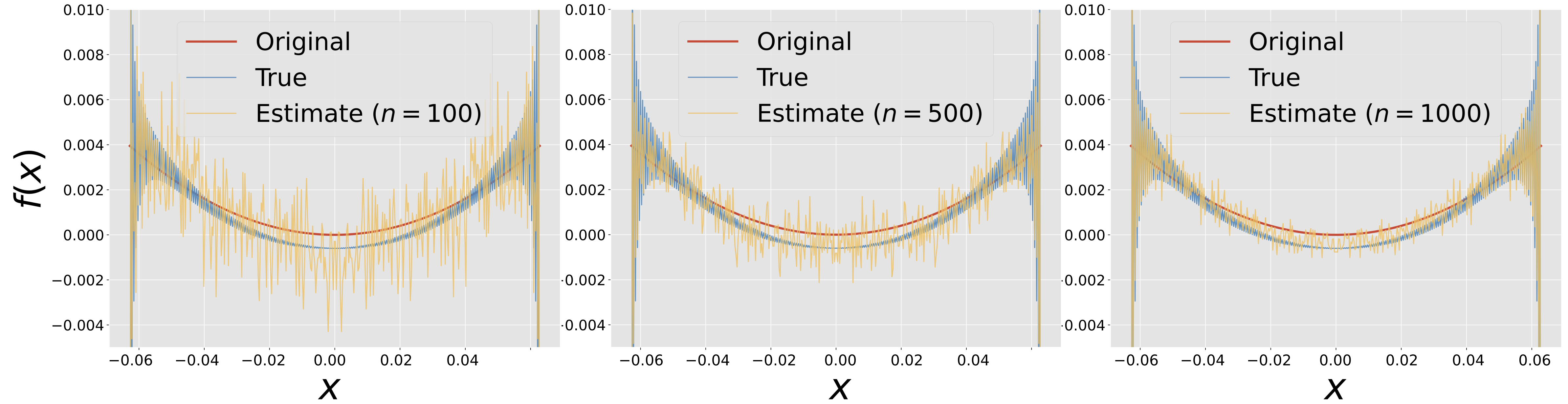}
            \caption{Optimal transport of $f(x) = x^2$.}
            \label{fig:2a}
        \end{subfigure}
    \end{minipage}
    \caption{Qualitative comparison of the optimal transport of the function $f(x)=x^2$ between $n=100,500$ and $1000$}
    \label{fig:2}
\end{figure}

\section{Additional details of the experiment}

We present the details of the hyperparameters used in the experiment in Section \ref{sec:application}. The hyperparameters that we chose for the four methods (RFF2D, ConvICNN, ConvNet and FOT) are shown in Table~\ref{tab:sst-hyperparameters}.

\begin{table}[t]
    \centering
    \begin{tabular}{lcccc}
        \hline
        \textbf{Parameter} & \textbf{RFF2D} & \textbf{ConvICNN} & \textbf{ConvNet} & \textbf{FOT} \\
        \hline
        \multicolumn{5}{l}{\textit{Training Parameters}} \\
        Learning rate & 3e-2 & 5e-5 & 1e-3 & \begin{tabular}[t]{@{}c@{}}2e-2 ($A$)\\5e-6 ($\Pi$)\end{tabular} \\
        Training iterations & 300 & 1000 & 1200 & 20000 \\
        Batch size & 365 & 365 & 365 & -- \\
        Optimizer & SGD & AdamW & SGD & AdamW \\
        \hline
        \multicolumn{5}{l}{\textit{Architecture}} \\
        Kernel sizes & [4,3,3,3,3] & [4,4,4,4,4] & [4,4,4,4,4] & -- \\
        Output channels & [5,5,5,5,5] & [128,256,256,512,1024] & [128,128,256,512,1024] & -- \\
        Embedding dimensions & 20 & -- & -- & -- \\
        \hline
        \multicolumn{5}{l}{\textit{FOT Parameters}} \\
        Eigenfunctions & -- & -- & -- & 13 \\
        $\gamma_{\text{eig}}$ & -- & -- & -- & 1.0 \\
        $\gamma_h$ & -- & -- & -- & 1e-20 \\
        $\gamma_A$ & -- & -- & -- & 1e-20 \\
        Entropic reg. & -- & -- & -- & Yes \\
        \hline
    \end{tabular}
    \caption{Hyperparameters for different models in the sea surface temperature experiments. All neural network models (RFF2D, ConvICNN, and ConvNet) use the same batch size and process normalized DCT coefficients of the input data. For the FOT, $A$ is the matrix of coefficients of the two-dimensional eigenfunctions, and $\Pi$ is the matrix of the transport plan. The FOT algorithm uses AdamW optimization with cosine decay for both $A$ and $\Pi$. Both matrices use the same cosine decay rate ($\alpha=0.1$) but different initial learning rates.}
    \label{tab:sst-hyperparameters}
\end{table}

\bibliography{main}
\bibliographystyle{alpha}

\end{document}